\documentclass[a4wide,11pt]{article}
\usepackage{amsfonts}
\usepackage{amssymb}
\usepackage{amsmath}
\usepackage{amsthm}
\usepackage{verbatim}
\usepackage{hyperref}
\theoremstyle{plain}
\newtheorem{theorem}{Theorem}[section]
\newtheorem{lemma}[theorem]{Lemma}
\newtheorem{corollary}[theorem]{Corollary}
\newtheorem{proposition}[theorem]{Proposition}

\theoremstyle{definition}
\newtheorem{definition}[theorem]{Definition}

\usepackage{color}
\usepackage{xspace}

\usepackage{ stmaryrd }

\newcommand{\vcut}[1]{}
\newcommand{\anote}[1]{{\ \color{blue}\bf{}{AK: #1}}}
\newcommand{\vnote}[1]{{\ \color{red}\bf{}{VG: #1}}}

\newcommand{\fo}{\varphi}
\newcommand{\fob}{\psi}
\newcommand{\foc}{\chi}
\newcommand{\fod}{\theta}
\newcommand{\ga}{\Gamma}
\newcommand{\de}{\Delta}
\newcommand{\extn}[1]{\llbracket #1 \rrbracket}

\newcommand{\prop}{\textit{PROP}\xspace}

\newcommand{\dep}{\ensuremath{\mathsf{D}}\xspace}
\newcommand{\ind}{\ensuremath{\,\mathsf{I}\,}\xspace}
\newcommand{\cdep}{\ensuremath{\mathsf{C}}\xspace}

\newcommand{\cons}{\ensuremath{\mathcal{C}}\xspace}
\newcommand{\Dep}{\ensuremath{\mathsf{Det}}\xspace}
\newcommand{\D}{\ensuremath{\mathcal{D}}\xspace}
\newcommand{\I}{\ensuremath{\mathcal{I}}\xspace}
\newcommand{\tr}{\mathsf{t}^{W}}

\newcommand{\dlang}{\ensuremath{\mathcal{L}_{\dep}}\xspace}
\newcommand{\dfor}{\ensuremath{\textit{FOR}(\dlang)}\xspace}
\newcommand{\clang}{\ensuremath{\mathcal{L}_{\cdep}}\xspace}

\newcommand{\ubox}{\ensuremath{[\mathsf{u}]}\xspace}
\newcommand{\udiam}{\ensuremath{\langle\mathsf{u}\rangle}\xspace}
\newcommand{\Ubox}{\ensuremath{[\mathsf{U}]}\xspace}
\newcommand{\Udiam}{\ensuremath{\langle\mathsf{U}\rangle}\xspace}
\newcommand{\Uman}{\ensuremath{\mathsf{U}}\xspace}
\newcommand{\udiamp}{\ensuremath{\langle\mathsf{u}'\rangle}\xspace}
\newcommand{\uboxp}{\ensuremath{[\mathsf{u}']}\xspace}
\newcommand{\ulang}{\ensuremath{\mathcal{L}_{\mathcal{U}}}\xspace}
\newcommand{\indlang}{\ensuremath{\mathcal{L}_{\ind}}\xspace}
\newcommand{\axdep}{\ensuremath{\textit{AX}(\dlang)}\xspace}
\newcommand{\axcdep}{\ensuremath{\textit{AX}(\clang)}\xspace}
\newcommand{\axindep}{\ensuremath{\textit{AX}(\indlang)}\xspace}
\usepackage[colorinlistoftodos,textwidth=100pt]{todonotes}
\newcommand{\vmnote}[1]{\todo[color=green!45!white]{\ \\\color{black}\bf{}{VG: #1}}}

\newcommand{\conj}{\mathit{Conj}}
\newcommand{\nff}{\mathit{DNF}}

\newcommand{\Iff}{\text{ iff }}

\title{\textbf{Logics for Propositional
Determinacy and Independence}}

\author{Valentin Goranko\\
\text{Department of Philosophy}\\
\text{Stockholm University}\\
\text{Sweden}\\
\text{Department of Mathematics} \\ 
\text{University of Johannesburg} \\
\text{South Africa\footnote{Visiting professorship}} 
\and Antti Kuusisto\\
\text{FB3: Mathematics/Computer Science}\\
\text{University of Bremen}\\
\text{Germany}}

\date{}
\begin{document}
\maketitle

\begin{abstract}
%
%We investigate propositional determinacy and independence from
%
%the point of view of formal logic. 
%
\noindent
This paper investigates formal logics for reasoning
about determinacy and independence.
Propositional Dependence Logic $\D$ and Propositional
Independence Logic $\I$ are recently developed
logical systems, based on team semantics, that 
provide a framework for such reasoning tasks.
We introduce two new logics $\mathcal{L}_{\dep}$ and  $\mathcal{L}_{\ind}$,
based on Kripke semantics, 
and propose them as alternatives for $\D$ and $\I$, respectively. 
We analyse the relative expressive powers of these four logics
%
%$\D$, $\I$, $\mathcal{L}_{\dep}$ and $\mathcal{L}_{\ind}$
%
and discuss the way these systems relate to natural language.
%
%the adequacy of the
%
%formalisations of determinacy and independence provided by them. 
%
We argue that $\mathcal{L}_{\dep}$ and $\mathcal{L}_{\ind}$ naturally
resolve a range of interpretational
problems that arise in $\D$ and $\I$.
We also obtain sound and complete axiomatizations for $\mathcal{L}_{\dep}$
and $\mathcal{L}_{\ind}$.
\end{abstract}
%

%\vnote{Point to add}
%\begin{itemize}
%\item Link with logical consequence relations. 
%\item The general case of $\dep^{X}$. 
%
%Different ways to bring into the syntax: by using definable sets $X$ or by specifying $X$ in the semantics. 
%\item 
%\end{itemize}

\section{Introduction}

%\anote{The intro below is only part 
%
%of the first draft of an intro. It will take several 
%
%rounds before the intro is good...}

In this paper we investigate the notions of 
\emph{propositional determinacy}
and \emph{propositional independence}.
%
%between propositional logical formulae.
%
We begin with a brief overview of related concepts.
\subsection{Dependence and independence: brief historical notes}
%
%Dependence and independence are abstract notions
%
%that have played an important role in mathematics and the
%
%natural sciences since antiquity.
%
%
%
Dependence and independence are abstract notions
that have played an important role in mathematics and the
natural sciences since antiquity.
Today the concepts are
omnipresent in virtually all fields of science.
There exists a wide range of different scientific notions of dependence, e.g.,
statistical correlation,
the causal relationship, 
and functional dependence.
Likewise, the notion of independence has 
different meanings in different contexts, e.g.,
probabilistic and linear independence in mathematics, 
as well as political and behavioural independence in
social vernacular.
An early formal logical analysis of dependence was proposed in \cite{grelling}\footnote{Also available as a reprint in \cite{Smith}.}.  Dependencies in relational databases have been studied since (at least)  the pioneering work of Codd in the early 1970's. Of the many relevant
references we only note here Armstrong's work (\cite{Armstrong74})
which provides a set of axioms for the notion of
functional dependence in databases.
The concept of dependence has also appeared in various philosophical
contexts under various different names.
From the  point of view of the current article,
the notion of \emph{supervenience} is
perhaps the most important related concept.
The paper \cite{SEP-Supervenience} describes supervenience formally as follows:

\medskip

``\textit{A set of properties $A$ supervenes upon another set $B$ just in case no two things can differ with respect to $A$-properties without also differing with respect to their $B$-properties. In slogan form, `there cannot be an $A$-difference without a $B$-difference'}.''

\medskip

On the formal level, perhaps the closest in spirit to the present study 
regarding  dependence is Humberstone's logical formalisation and
study of supervenience as a generalisation of logical consequence in, e.g.,
\cite{Humberstone92,Humberstone93,Humberstone98}.
Another notion related to the current paper is
the notion of \emph{contingency}.
Formal investigations of contingency can be found in, inter alia, \cite{montgomery},
\cite{Humberstone95,Humberstone02}, as well as \cite{Pizzi07,Pizzi13}. 
For a recent study of the contingency operator in various modal logics see \cite{FanWD15}.
We also mention two very recent and closely related papers
written after the appearance of the earlier version \cite{GorankoKuusistoArxiv2016}
of the present paper. The two papers have at least partially 
been written as a response to \cite{GorankoKuusistoArxiv2016}.
The first one of them is
\cite{JieFan2016} which develops a formal modal logic of
supervenience and also addresses some research questions raised in \cite{GorankoKuusistoArxiv2016} and reiterated here in Section \ref{futuredirections}.  
The other one is \cite{Humberstone2017} which explores, inter alia, connections between supervenience and dependence and discusses in detail some aspects of 
\cite{GorankoKuusistoArxiv2016}.
In the context of logical semantics,
the notion of \emph{independence} has 
been investigated perhaps most prominently in 
\emph{Independence Friendly (IF) Logic}
originally defined in \cite{hisa89}; see also \cite{Hintikka96}.
%
%and follow-up work  based on game-theoretic semantics.
%
IF logic was first formulated in terms of
game-theoretic semantics, and no compositional
semantics for that logic was originally available. 
Later on, Hodges developed (\cite{Hodges97a}) a
compositional semantics for the system, 
currently know as \emph{team semantics}.
%
%also further developed in \cite{Vaananen+Hodges:2010}.
%
The idea of team semantics, in turn, lead to V\"a\"an\"anen's development of
\emph{Dependence Logic} in \cite{va07}.
Dependence logic sparked a renewed interest in
logical formalisation and analysis of dependence and
initiated an active related research programme.
For an overview of the work in that
direction, see \cite{kontisurvey, gallianivaa} and the references therein.
\subsection{Propositional logics of dependence and independence based on team semantics} 

V\"a\"an\"anen's Dependence Logic extends classical first-order logic with \emph{dependence atoms} $$\dep(x_1,...,x_k;y)$$ with the intuitive meaning that the choice of an
interpretation for $y$ is \emph{functionally determined} by the choices of 
interpretations for $x_1,...,x_k$ in evaluation games based on game-theoretic semantics.
Since the introduction of dependence atoms, research on logics based on team semantics has flourished and several kinds of related logical systems have been investigated.
%

%
\begin{comment}
%
The idea, restricted to propositional formulae, leads to the notion of propositional determinacy as follows. A  set 
$\{p_1,...,p_k\}$ of propositions is said to \emph{determine (logically)} proposition $q$, if any pair of possible worlds that agree on the truth values of $p_1,...,p_k$, must also agree on the truth value of $q$.
%
\end{comment}
%

%
A propositional modal variant of Dependence Logic, called \emph{Modal Dependence Logic}, was defined in \cite{vaa07}. That logic extends the syntax of ordinary modal logic with a new operator $\dep$ and formulae $\dep(p_1,\dots,p_k;q)$ with the intuitive interpretation that the truth values of $p_1,\dots,p_k$  \emph{determine} the truth value of $q$.
%
\begin{comment}
%
More formally, a set $\{p_1,\dots,p_k\}$ of
%
propositions \emph{determines} proposition $q$ in a set $T$ of
%
possible worlds
%
if any pair of worlds in $T$ that
%
agree on the truth values of $p_1,\dots,p_k$,
%
must also agree on the truth value of $q$.
%
The sets  $T$ of possible worlds are called \emph{teams} in
%
Modal Dependence Logic.
%
\end{comment}
%
The propositional fragment of 
Modal Dependence Logic extended with $\dep$
gives rise to \emph{Propositional Dependence Logic  $\D$}.
In the logic $\D$, sets of \emph{propositional assignments}
are called \emph{teams}; recall that a propositional assignment is
simply a function from a set of atomic proposition
symbols to the Boolean domain $\{0,1\}$.
Intuitively, a team can be regarded as a set of possible worlds.
A formula $\dep(p;q)$ is then 
defined to be \emph{true in a team} $W$ if and only if each pair of 
possible worlds $w,u\in W$ that give the same truth value to $p$,
must also give the same value to $q$.
The more complex atoms $\dep(p_1,\dots,p_k;q)$ are
interpreted in a similar fashion: any pair of worlds in $W$
that agree on the truth values of $p_1,\ldots,p_k$,
must also agree on the value of $q$ (see Section \ref{preliminaries}
for the formal definition.)
The propositions $p_1,\dots,p_k$ are said to
\emph{determine} $q$.
%

%
\begin{comment}
%
While the treatment of formulae $\dep(p_1,...,p_k;q)$ in $\mathcal{D}$
%
formalises the notion of propositional determinacy in a rather straightforward and natural
%
manner, the treatment of composite formulae is more intricate.
%
For example, the formula $p\vee q$ is defined to be true in a team $W$ if
there exist two teams $U,V\subseteq W$ such that every point in $U$ satisfies $p$
and every point in $V$ satisfies $q$, and furthermore, $U\cup V = W$. This may seem strange at first, but in fact turns out to be quite natural when the semantic choices underlying team semantics are analysed.
%
\end{comment}
%

The notion of \emph{independence} which we 
investigate in this paper 
originates from \emph{Independence Logic}
defined in \cite{GV13}.
Analogously to  Dependence Logic, Independence Logic extends first-order logic by
\emph{independence atoms} $$(x_1,\dots,x_i)\, \ind_{(y_1,\dots,y_j)}(z_1,\dots,z_k),$$
with the intuitive meaning that for any fixed set of values for 
$y_1,\dots,y_j$, the possible values for $x_1,\dots,x_i$ are
independent of the possible values for $z_1,\dots,z_k$. 
A propositional 
variant of Independence Logic, called \emph{Propositional Independence Logic} (and denoted
here by $\I$) has been investigated in the literature  
(in, e.g., \cite{Konti14}, \cite{Y14}),
and it relates to Independence Logic the same way Propositional Dependence Logic 
relates to Dependence Logic.

%
%\begin{comment}
%
Dependence Logic and Independence Logic, together with the \emph{Inclusion Logic} of
\cite{galli},  are currently the central logical systems studied in the framework of team
semantics.
%
%\end{comment}
%
It is also worth noting here 
that Propositional Dependence Logic is 
closely related to \emph{inquisitive logic}
\cite{ciardellimasters, ciardelliroelofsen}.
In particular, the system $\mathrm{InqL}$ of inquisitive logic is a
propositional team-based logic equi-expressive with $\mathcal{D}$.
%employs a similar team-based approach to semantics.
For investigations on the relations between inquisitive logic and $\mathcal{D}$, see, e.g.,
\cite{Y14, vaanayan, ciardellitwo, giardelli}.

\subsection{The content and contributions of this paper}

While many of the semantic choices underlying Propositional Dependence Logic \D
are natural and justified, we will identify in 
this paper a range of issues that are problematic.
One such issue is the interpretation of formulae that
use \emph{combinations} of determinacy operators $\dep$
and disjunctions $\vee$.
We will argue that while team semantics gives a sensible 
interpretation to formulae $\dep(p_1,...,p_k;q)$
as well as disjunctive formulae $\varphi\vee\psi$
free of operators $\dep$,
interpretations of certain simple formulae
that combine $\dep$ and $\vee$ become 
strange from the point of view of natural language.
We also discuss similar issues related to the team semantics interpretation of
negation $\neg$ (together with $\dep$).
%

%\subsection{Propositional logics of determinacy and independence based on Kripke semantics} 

Motivated by the interpretational problems of $\mathcal{D}$,
we develop here an alternative natural logic of
determinacy called \emph{Propositional Logic of Determinacy}
and denoted by $\mathcal{L}_{\dep}$.
%
%The logic operator of propositional determinacy
%
%between formulae in the object language. 
%
%We compare it with the above mentioned
%
%Propositional Dependence Logic $\D$,  both in
%
%terms of formal expressiveness and in relation to natural language.
%
The logics  $\D$ and $\mathcal{L}_{\dep}$ have essentially the same
set of formulae,\footnote{Strictly speaking,
$\mathcal{L}_{\dep}$ has more formulae than
$\mathcal{D}$ because of the typical syntactic restrictions
applied in $\mathcal{D}$ and team semantics in general.
This issue is discussed in more detail in the sections below.}
but the semantic approaches differ.
Instead of team semantics, the system $\mathcal{L}_{\dep}$ is
essentially based on Kripke
semantics. A formula $\dep(\varphi_1,...,\varphi_k;\psi)$ is true\footnote{This
truth definition was first suggested as an alternative to team
semantics in \cite{doubleteam1} and its 
later versions such as, e.g., \cite{doubleteam2}.} in a
possible world $w$ if the set $R(w)$ of \emph{accessible alternatives} of $w$ satisfies the
determinacy condition: for all $u,v\in R(w)$, if $u$ and $v$ agree
on the truth values of each $\varphi_i$, they also agree on
the truth value of $\psi$.
The Boolean connectives
as well as proposition symbols 
are interpreted in $\mathcal{L}_{\dep}$ in
the same way as in Kripke semantics, and thus $\mathcal{L}_{\dep}$ can
be regarded as a modal logic with a \emph{generalized modality} $\dep$
that talks about determinacy rather than possibility or necessity.
Mainly in order to keep matters technically simple in this
initial work on $\mathcal{L}_{\dep}$, we
assume the accessibility
relation $R$ to be the universal relation, so  the set of successors of
any world $w$ is in fact the whole domain of the model. At the end of the paper we briefly discuss the general case with other kinds of accessibility relations.
As an important part of our discourse on $\mathcal{L}_{\dep}$,
we present a range or arguments
for the naturalness of $\mathcal{L}_{\dep}$ in
relation to natural language.
In particular, we argue that $\mathcal{L}_{\dep}$ resolves
reasonably well the 
interpretational problems that we identify for $\mathcal{D}$. 

It turns out that Propositional Independence Logic \I
is burdened by virtually
the same issues as Propositional Dependence Logic \D, and these issues
can be remedied by defining
\emph{Propositional Logic of Independence} $\mathcal{L}_{\ind}$  
analogously to $\mathcal{L}_{\dep}$ but based on the
independence operator $\ind$ rather than the
dependence operator $\dep$.
In addition to introducing the logics 
$\mathcal{L}_{\dep}$ and $\mathcal{L}_{\ind}$ 
and discussing how they, as well as $\D$ and $\I$,
relate to each other and to natural language, we
also provide a comparative analysis of the
expressive powers of these four logics. We show that, while
$\D$ and $\I$ are both strictly contained in $\mathcal{L}_{\dep}$
and $\mathcal{L}_{\ind}$, the latter two logics are equally expressive.
In fact, we establish in Sections \ref{depsection} and \ref{indsection}
that $\mathcal{L}_{\dep}$ and $\mathcal{L}_{\ind}$
are \emph{maximally expressive}, or \emph{expressively
complete}, in a certain natural sense. 
Since it is well known from the literature on team semantics 
that $\D$ is strictly contained in $\I$, we eventually obtain a complete
classification of the relative expressive powers of the four logics.
We also prove that while both $\D$ and $\I$
translate into both $\mathcal{L}_{\dep}$ and $\mathcal{L}_{\ind}$,
there exists no \emph{compositional translation}\footnote{See
 Section \ref{compositionality} for the definition of compositional translations.}
from either of the team-semantics-based logics into
$\mathcal{L}_{\dep}$ or $\mathcal{L}_{\ind}$.
Intuitively, this indicates that team semantics and the
Kripke-style semantics of $\mathcal{L}_{\dep}$ and $\mathcal{L}_{\ind}$
are substantially different logical frameworks, at least from the technical point of view.
In addition to studying expressivity issues, we provide
sound and complete axiomatizations for $\mathcal{L}_{\dep}$
and $\mathcal{L}_{\ind}$. This turns out to be a relatively 
straightforward, yet interesting exercise, due to certain close
connections---to be identified
below---between $\mathcal{L}_{\dep}$, $\mathcal{L}_{\ind}$ and  
\emph{Contingency Logic} from \cite{montgomery}.
Contingency Logic is the variant of modal logic
with a modality $\cdep$, where $\cdep\varphi$  is interpreted to mean that $\varphi$ is \emph{non-contingent} at the state of evaluation, i.e., has the same truth value at every successor of that state. We also show that there do not exist finite
axiomatisations for $\mathcal{L}_{\dep}$
and $\mathcal{L}_{\ind}$, assuming a standard
notion of an axiomatic deduction system.
%
\begin{comment}
%
The history of this logic goes back at
%
least to \cite{montgomery}. For more information and recent results on this logic, see \cite{FanWD14}
and the references therein.
%
\end{comment}
%

%
The structure of this paper is as follows.
In Section \ref{preliminaries} we provide the necessary
background for the rest of the paper, including definitions of
the logics $\mathcal{D}$ and $\mathcal{I}$ based on team semantics.
In Sections \ref{depsection} and \ref{indsection} we 
define the logics $\mathcal{L}_{\dep}$ and $\mathcal{L}_{\ind}$
and study their basic properties.
Section \ref{naturallanguageinterpretations}
analyses the logics $\mathcal{D}$, $\mathcal{I}$, $\mathcal{L}_{\dep}$
and $\mathcal{L}_{\ind}$ in relation to natural language.
Section \ref{express} investigates expressivity issues
and Section \ref{validities} provides 
sound and complete axiomatizations
for $\mathcal{L}_{\dep}$ and $\mathcal{L}_{\ind}$.
Section \ref{futuredirections} briefly discusses a range of
future research directions and concludes the paper.
The agenda and main idea of the current paper, i.e., investigating dependence logic
with a Kripke-style semantics instead of team semantics,
has been first mentioned and briefly motivated in
\cite{doubleteam1}
and its later incarnations such as, e.g., \cite{doubleteam2}.
Similar ideas have subsequently been developed in \cite{giardelli}.
The current paper develops the ideas of \cite{doubleteam1} in detail,
and furthermore, provides an extensive
collection of related technical results concerning the
expressivity and axiomatizability of $\mathcal{L}_{\dep}$
and related systems.
Section 6.7.2 of \cite{giardelli} contains an explicit comparison of $\mathcal{L}_{\dep}$
with a logic $\mathrm{InqB}^{\Rightarrow}$ that employs a Kripke-style approach (with some
extra machinery) to
semantics and thereby exhibits the same principle as $\mathcal{L}_{\dep}$ that dependence
statements are \emph{modal statements}. One of the main
differences between $\mathrm{InqB}^{\Rightarrow}$ and $\mathcal{L}_{\dep}$ is
that $\mathrm{InqB}^{\Rightarrow}$ contains an explicit machinery for
questions and thereby allows semantically elaborate assertions about, e.g., dependence.
Dependence statements in $\mathcal{L}_{\dep}$ are relations between
statements, while in $\mathrm{InqB}^{\Rightarrow}$, dependence statements are
construed as a relation between questions.
See Sections 6.7.2 and 6.5 of \cite{giardelli} for further details.

\section{Preliminaries and background}\label{preliminaries}

\subsection{Functional determinacy}

Determinacy of a function by a set of functions is a central
concept in this article. We will
define it here in the general setting, though will use it further only on Boolean functions. 

%
%\anote{We should mention it in the case of first-order dependence atoms
%
%$\dep(x_1,\ldots,x_k;y)$.}
%

\begin{definition}\label{determinacy definition}
Let $k\in\mathbb{Z}_+$ be a positive integer and let $X,X_{1},\ldots,X_{k},U$ be nonempty sets.
Let $f: U \to X$ be function, and consider a family of functions 
$$\{ f_{i}: U \to X_{i} \mid i = 1,\ldots, k\}.$$ 
Given a set $W\subseteq U$, we say that \emph{the function $f$ is determined by the family of functions 
$\{f_{1},\ldots,f_{k}\}$ on $W$}, or that 
\emph{the family 
%$\{f_{i}\}_{i \in \{1,\ldots, k\}}$ 
$\{f_{1},\ldots,f_{k}\}$
determines the function $f$ on $W$}, if there exists a function $F: X_{1}\times\ldots\times X_{k} \to X$ such that $f$ is the composition of $F$ and the functions $f_{1},\ldots,f_{k}$ on $W$, that is, $f(w) = F(f_{1}(w),\ldots, f_{k}(w))$ for every $w\in W$. 
We also fix this definition in the following special case:
we say that a function $g:U\rightarrow X$ is
determined by $\emptyset$ on $W\subseteq U$ if $g$ is
constant on $W$, i.e., $g(w_1) = g(w_2)$ for all $w_1,w_2\in W$.

This definition generalises straightforwardly to determinacy of a function by any family of functions
$\{ f_{i}: U \to X_{i} \mid i \in I\}$ indexed with an arbitrary (possibly infinite) set $I$.  
%$\{ f_{i}: U \to X_{i} \mid i \in I\}$ indexed with an arbitrary set $I$. Let $\Pi_{i\in I} X_{i}$ be the product of the family of sets $\{ X_{i} \mid i \in I\}$. We say that \emph{the function $f$ is determined by the family 
%$\{ f_{i} \mid i \in I\}$ on $W$}, or that \emph{the family $\{ f_{i} \mid i \in I\}$ determines the function $f$ on $W$}, if there is a mapping $F: \Pi_{i\in I} X_{i} \to X$ such that $f$ is the composition of $F$ with $\{ f_{i} \mid i \in I\}$ on $W$, that is, $f(w) = F(\langle f_{i}(w),\ldots,  f_{k}(w) \rangle)$ for every $w\in W$. 
\end{definition}
Equivalently, $f$ is determined by the (possibly empty) family $\{f_{1},\ldots,f_{k}\}$ on $W$ if
and only if the following condition holds.

\medskip

\noindent
\textbf{Det}:
For every $w_{1}, w_{2} \in W$, it holds that
if $f_{i}(w_{1}) = f_{i}(w_{2})$ for each $ i = 1,\ldots, k$, then $f(w_{1}) = f(w_{2})$.

\medskip

Indeed, if $f$ is the composition of a mapping $F$ with $f_{1},\ldots,f_{k}$ on $W$ and  $f_{i}(w_{1}) = f_{i}(w_{2})$ for each $ i = 1,\ldots, k$, then 
$F(f_{1}(w_{1}),\ldots, f_{k}(w_{1})) = F(f_{1}(w_{2}),\ldots, f_{k}(w_{2}))$, i.e., $f(w_{1}) = f(w_{2})$. Conversely, if the condition \textbf{Det} holds, then we can define a mapping 
 {$F: X_{1}\times\ldots\times X_{k} \to X$} as follows: 
\[ F(x_{1},\ldots, x_{k}) =  
 \begin{cases}
    f(w) & \text{ if } x_{1} = f_{1}(w), \ldots,  x_{k} = f_{k}(w), \text{ for some } w\in W, \\
   x & \text{ otherwise (where } x\in X \text{ is arbitrarily fixed).}  
 \end{cases}
\]
The condition \textbf{Det} guarantees that this is well-defined.
\qed

%\vnote{Add remark linking the two definitions, by using the diagram from the ``First Isomorphism Theorem''.}

We let $\Dep_{W}(f_{1}, \ldots, f_{k} ; f)$
denote the assertion that $f$ is determined by the family of functions 
$\{f_{1},\ldots,f_{k}\}$ on $W$.
When $k=0$, we write $\Dep_{W}(\emptyset\, ; f)$.
\subsection{Preliminaries concerning propositional logic}\label{prelipreliwhatever}
We typically denote formulae by $\fo,\fob,\foc,\fod,\alpha,\beta$ and sets of formulae by
$\Phi, \Psi$.
Throughout the paper, we
let $\prop$ denote a fixed countably infinite set of proposition symbols. All formulae considered in
the paper will be assumed to be built over $\prop$.

Let $\Phi =  \{\varphi_1,\ldots,\varphi_k\}$ be a finite nonempty
set of formulae. We define a set $\nff(\Phi)$ as follows. 
%%
%denote the finite set of all Boolean combinations
%%
%of the formulae $\varphi_1,\ldots,\varphi_k$
%%
%in disjunctive normal form, i.e., $\nff(\Phi)$ is  
%%
%formally defined as follows.
%

\begin{enumerate}
\item
For each subset $S\subseteq \{1,\ldots,k\}$,
let $\psi_{S}$ denote the conjunction $\psi_1\wedge\ldots\wedge\psi_k$
such that
\[
 \psi_i\ =\ \begin{cases}
    \varphi_i & \text{ if } i\in S,\\
    \neg\varphi_i & \text{ if } i\not\in S.
 \end{cases}
\]
\item
Let $\conj(\Phi)\, :=\, \{\, \psi_S\ |\ S\subseteq\{1,\ldots,k\}\, \}$.
The formulae in $\conj(\Phi)$ are called \emph{types over $\Phi$}.  

\item
Define $\nff(\Phi)\, :=\, \{\, \bigvee U\ |\ U\subseteq \conj(\Phi)\}$.
\end{enumerate}
We call the formulae in $\nff(\Phi)$ \emph{type normal form formulae  over $\Phi$}.

In the above definition, $\bigvee\emptyset$ is assumed to be the formula
$p\wedge\neg p$ for some proposition symbol $p\in\mathit{PROP}$. 
We will not assume that the logical constant symbols $\top,\bot$ are available
as primitives in the languages we consider. However, we will use these
symbols as abbreviations for the formulae $p\vee\neg p$ and $p\wedge\neg p$,
respectively.
%We will also assume that the logical constant symbol $\top$ as primitive, 
%while $\bot$ will be assumed definable as $\lnot \bot$. 
%\anote{We will, however, use the symbols $\bot,\top$ as abbreviations for the formulae %$p\wedge\neg p$ and $p\vee\neg p$, respectively.}
For technical convenience, we define 
$\nff(\emptyset):= \{\top, \bot \}$ and $\mathit{Conj}(\emptyset) := \{\top\}$,
and thus we let $\top$  be the unique type over $\emptyset$.

When we write a formula $\fod(q_1,\ldots,q_k)$, we indicate that all the propositional variables occurring in $\fod$ are amongst $q_1,\ldots,q_k$. 
Given a formula $\fod = \fod(q_1,\ldots,q_k)$ and a tuple of formulae 
$\fo_1,\ldots,\fo_k$, we denote by $\fod(\fo_1,\ldots,\fo_k)$ the result of a uniform substitution of $\fo_1,\ldots,\fo_k$, respectively, for $q_1,\ldots,q_k$ in the formula $\fod$.

\begin{definition}
\label{def:equivalentreplacements}
Let $L$ be a logic. A relation of equivalence $\equiv_{L}$ between formulae of $L$ 
satisfies the \emph{equivalent replacements property} (ER)
(with respect to $\equiv_{L}$), if for every 
$L$-formula $\fod(q_1,\ldots,q_k)$ and all tuples of   
$\fo_1,\ldots,\fo_k,\fob_1,\ldots,\fob_k$ such that $\fo_i \equiv_{L} \fob_i$ for each $i = 1,\ldots,k$, it holds that $\fod(\fo_1,\ldots,\fo_k) \equiv_{L} \fod(\fob_1,\ldots,\fob_k)$. 
\end{definition}

%
\begin{comment}
%
It is well known that ER holds for the standard 
%
equivalence of standard propositional logic, hereafter denoted PL.
%
\end{comment}
%

%\vnote{Please define here the logic \ulang. }

\subsection{State description models}

Recall that $\prop$ denotes a fixed countably infinite set
of proposition symbols.
%We consider so far the purely propositional language \lang built on $\prop$. 
An \emph{assignment} for $\prop$ is any mapping $f: \prop \to \{0,1\}$.
Following Carnap, we occasionally call assignments
also \emph{state descriptions}. 
%

%Take any Kripke model  $M = (W,R,V)$ and let $w\in W$. By $R(w)$ we denote the set of $R$-successors of $w$. 

Any (possibly empty) set $W$  of  state descriptions 
will be called a  \emph{state description model} (or \emph{SD-model}).
%Later when we define the logics $\mathcal{L}_{\dep}$ and $\mathcal{L}_{\ind}$,
%we will regard these models simply as Kripke models with a universal
%accessibility relation $W\times W$.
The reason we include the empty model in the picture is
technical and related to the the fact that we will
deal, inter alia, with logics based on team semantics. In team semantics, as we will see, 
the \emph{empty team} plays an important role.

Especially in the more technical parts of the paper,
we often talk about \emph{points} or \emph{worlds} of $W$ 
rather than assignments or state descriptions.
%
%We also often simply talk about assignments rather
%
%than Boolean assignments. 

\subsection{Universal modality}

In this paper we use a variant $\mathcal{L}_{\mathcal{U}}$
of the modal logic with the 
\emph{universal modality} from \cite{GP92}.
Formally, the syntax of
the logic $\mathcal{L}_{\mathcal{U}}$  is
given by the grammar
$$\varphi\ ::=\ p\ |\ \neg\varphi\ |\ (\varphi
\rightarrow\varphi)\ |\ \Ubox\varphi,$$
where $p\in\mathit{PROP}$. We define $\Udiam$ to be the dual of $\Ubox$, i.e.,
$\Udiam \varphi := \neg \Ubox \neg\varphi$.

The semantics of $\mathcal{L}_{\mathcal{U}}$ is defined with respect
to SD-models $W$ and assignments $w\in W$ as follows.
\[
\begin{array}{lll}
W,w\models_{\mathcal{U}} p & \Iff & w(p) = 1\\
W,w\models_{\mathcal{U}} \neg\varphi & \Iff & W,w
\not\models_{\mathcal{U}} \varphi\\
W,w\models_{\mathcal{U}} \varphi\rightarrow\psi& \Iff & W,w
\not\models_{\mathcal{U}}\varphi
\text{ or }W,w\models_{\mathcal{U}}\psi\\
W,w\models_{\mathcal{U}} \Ubox\varphi & \Iff  &
W,u\models_{\mathcal{U}}\varphi\text{ for all }u\in W
\end{array}
\]
Thus, 
$W,w\models_{\mathcal{U}} \Udiam \varphi$ iff  
$W,u\models_{\mathcal{U}}\varphi$ for some $u\in W$.
As customary in modal logic, we define $W\models_{\mathcal{U}}\varphi$ iff
$W,w\models_{\mathcal{U}}\varphi$ 
for all $w\in W$. We also define, in the standard way,
that $\models_{\mathcal{U}}\varphi$ iff $W\models_{\mathcal{U}}\varphi$
for all SD-models $W$.
\subsection{Propositional Dependence Logic $\D$}
%

%
%\anote{
We now define
\emph{Propositional Dependence Logic} $\D$, which first appeared in the literature  on team semantics as a syntactic
fragment of \emph{Modal Dependence Logic}, defined in
\cite{vaa07}. 
The paper \cite{vaa07} did not make explicit references to
 $\D$, and the semantics for
Modal Dependence Logic---including its propositional fragment---was
formulated in \cite{vaa07} in terms of Kripke models rather than SD-models.
%
%
%\vmnote{Some discussion is needed on how Modal Dependence Logic relates to $\dlang$.}
%
%
%
%which is the way
%Propositional Dependence Logic has
%recently been formulated.

Propositional Dependence Logic, with that explicit name, and variants
of the logic have recently been studied in, e.g.,
\cite{He14, Vir,Y14,vaanayan}.
The models for Propositional Dependence Logic are 
currently typically defined in the literature as SD-models where the set of proposition symbols  in consideration is finite; thus the related SD-models are sets of
\emph{finite} state descriptions, i.e., finite assignments.
For most  purposes, it makes little difference 
whether SD-models with finite or infinite sets of proposition symbols are used.
Similarly, it is mostly unimportant whether the models under consideration are Kripke models or SD-models. Such distinctions could, however,  become more important in extensions and variants of the logics considered in this paper.

%
%Let $\mathit{DEP}$ be the set of formulae
%
%$\dep(p_1,\ldots,p_k;q)$, where $p_1,\ldots,p_k;q$
%
%are proposition symbols in $\Pi$.
%
The syntax of \emph{Propositional Dependence Logic} $\D$ 
is given by the following grammar (cf. \cite{vaa07}),
$$\varphi ::=  p\ |\ \neg p\ |\ \dep(p_1,\ldots,p_k;q)\ |\ \neg\dep(p_1,\ldots,p_k;q)\ |
\ (\varphi\vee\varphi)\ |\ (\varphi\wedge\varphi)$$
where $p,q, p_1,\ldots,p_k\in\prop$ and $k\in\mathbb{N}$.
When  considering formulae $\dep(p_1,\ldots,p_k;q)$ where $k=0$,
we write\footnote{Here $\cdep$ stands for `constancy' but, as Humberstone has noted, it might be confused with `contingency' which is, in fact, its opposite.} $\cdep\varphi$ 
instead of $\dep(;q)$ or $\dep(\epsilon\, ;\, q)$, where $\epsilon$ denotes the empty sequence of
proposition symbols.
We let $\mathit{FOR}(\D)$ denote the
set of formulae of  $\D$.
Note that formulae of $\D$
are in negation normal form, and the operator $\dep$
takes as inputs only proposition symbols.
Let $W$ be a state description model.
The semantics of Propositional Dependence Logic $\D$,
to be defined below, is based on \emph{team semantics}\footnote{
Throughout this paper, the turnstile $\Vdash$ is reserved
for logics based on team semantics and the turnstile $\models$
for logics with a Kripke-style semantics. The reader is warned
that we often use different turnstiles when comparing different logics,
and the related change of turnstile may sometimes be difficult to spot at first.},
%
%We will use the two turnstiles $\Vdash$ and $\models$ throughout
%the article for different semantic readings.
%by putting respective indices, wherever necessary. 
%
given by the following clauses (cf. \cite{vaa07,vaanayan}).

%\vnote{I have introduced subscripts to $\Vdash$ and $\models$ to make the notation  more precise.}

%

\[
\begin{array}{lll}
W\Vdash_{\D} p & \Iff & w(p) = 1\text{ for all }w\in W\\
W\Vdash_{\D} \neg p & \Iff & w(p) = 0\text{ for all }w\in W\\
W\Vdash_{\D} \dep(p_1,\ldots,p_k;q) & \Iff  & 
\mbox{for } \text{ all } u,v\in W, \mbox{ if }u(p_i) = v(p_i)
%\bigwedge\limits_{i\, \in\, \{1,\ldots,k\}}
\\
& &\text{holds }\text{for }
\mbox{all $i\leq k$, then } 
u(q) = v(q)
\\
%
%W\models\dep(p_1,\ldots,p_k;q)\\
%\end{array}
%\]
%
%\[
%\begin{array}{lll}
%
W\Vdash_{\D} \neg \dep(p_1,\ldots,p_k;q) & \Iff\ & W=\emptyset \\
W\Vdash_{\D} \varphi\wedge\psi & \Iff & W\Vdash_{\D}\varphi\text{ and }
W\Vdash_{\D}\psi\\
W\Vdash_{\D} \varphi\vee\psi & \Iff & U\Vdash_{\D}\varphi\text{ and }
V\Vdash_{\D}\psi \text{ for some } U,V\subseteq W\\
& &\text{such that }U\cup V = W\\
\end{array}
\]
We observe that $W\Vdash\cdep p$ iff $v(p) = u(p)$ for all $u,v\in W$, i.e.,
the truth value of $p$ is constant in $W$.
The rationale for the truth condition of $\neg \dep(p_1,\ldots,p_k;q)$, 
as stated in \cite[p.24]{va07} and reiterated in \cite{LV13}, is as follows.
Suppose we 
wish to maintain the same duality
between $\dep(p_1,\ldots,p_k;q)$ and $\neg\dep(p_1,\ldots,p_k;q)$ as
the one that holds for the truth conditions for $p$ and $\neg p$.
We then end up with the definition that $W\Vdash_{\D} \neg \dep(p_1,\ldots,p_k;q)$
iff for all $u,v\in W$,
%
%
%
%
%
%
%\begin{equation}\label{negateddependencecondition}
%
$$u(p_1) = v(p_1) \wedge \ldots \wedge u(p_k) = v(p_k) \wedge u(q) \not=v(q),$$
%
%\end{equation}
%
where the expression above 
%Equation \ref{negateddependencecondition} is
%
is obtained by negating the condition provided in the truth definition of  
$\dep(p_1,\ldots,p_k;q)$ \emph{after} the
universal quantification of $u$ and $v$.\footnote{In other words,
if we abbreviate the truth condition for $p$ by
$\forall w\in W:\Psi$ and the truth condition
for $\dep(p_1,\ldots ,p_k;q)$ by $\forall u,v\in W:\Psi'$,
then the truth condition for $\neg p$ is
$\forall w\in W:\neg\Psi$ and thus we
define the condition for $\neg\dep(p_1,\ldots ,p_k;q)$
to be $\forall u,v\in W:\neg\Psi'$. How natural this
choice is exactly, is a question that calls for further analysis.
We will briefly discuss issues 
related to this matter in Section \ref{naturallanguageinterpretations}.
}  
%Equation \ref{dependencecondition}\, .
%
We then observe that according to the
obtained definition, $W\Vdash_{\D} \neg \dep(p_1,\ldots,p_k;q)$ iff $W=\emptyset$.
Another possible rationale for the truth
definition of formulas $\neg\dep(p_1,\dots ,p_k;q)$
can be obtained via an algebraic interpretation of
formulae, as given in \cite{Roelofsen2013}. Here 
formulae are associated with non-empty and downwards
closed sets of $\mathrm{SD}$-models. Such sets correspond to
possible meanings of formulae. This approach leads to a Heyting algebra.
Negation is interpreted as the pseudo-complement operation,
and for the formula $\neg\dep(p_1, \dots, p_n;q)$ this gives an
interpretation that is equivalent to the clause given above, i.e., $\neg\dep(p_1, \dots, p_n;q)$ is
satisfied by an $\mathrm{SD}$-model $W$ iff $W=\emptyset$.
See \cite{Roelofsen2013} for further details.
While team semantics may seem strange at first, the following proposition
justifies its naturalness \emph{with respect to propositional
logic}, i.e., the sublanguage of $\D$ without formulae of the type 
$\dep(p_1,\ldots,p_k;q)$ and $\neg\dep(p_1,\ldots,p_k;q)$.
Recall that the turnstile $\models_{\mathcal{U}}$
refers to $\mathcal{L}_{\mathcal{U}}$.
\begin{proposition}\label{classicalproposition}
Let $\varphi$ be a formula of  propositional logic in
negation normal form.
Then $W\Vdash_{\D}\varphi$ iff
$W,w\models_{\mathcal{U}}\varphi$ for all $w\in W$. In other words,
 $W\Vdash_{\D}\varphi$ iff $W\models_{\mathcal{U}}\varphi$.
\end{proposition}

This  proposition shows that
team semantics simply lifts the semantics of propositional logic
(in negation normal form) from the level of
individual assignments onto the level of \emph{sets} of assignments.
Thus, team semantics can be used in scenarios
where assertions (encoded by formulae of
propositional logic) are made about \emph{sets} of possible worlds,
and the intention of the assertions is to claim that
any world in the set satisfies the formula.
We will consider examples of such scenarios in
Section \ref{naturallanguageinterpretations},
where we discuss the relation between natural language
and the logic $\D$.

%\anote{A discussion responding Section 3 of Ciardelli's review will be added here.}
%\cite{giardelli}
%

\subsection{Propositional Independence Logic \I}\label{indeplogic}

We next present \emph{Propositional Independence Logic} $\I$
which was conceived as a fragment of
\emph{Modal Independence Logic} in \cite{Konti14}
and studied further in, e.g., \cite{Kont15}.
Propositional Independence Logic
relates to \emph{Independence Logic} of \cite{GV13} in
the same way \D relates to V\"{a}\"{a}n\"{a}nen's 
Dependence Logic.
 There are, of course, different kinds of notions of (propositional) independence, and the
logic \I provides a formal approach to
\emph{a particular such notion}.
The logic is similar in spirit to $\D$, being based on team semantics.
The syntax of $\I$ is given by the following grammar.
$$\varphi ::=  p\ |\ \neg p\ |\  (p_1,\ldots,p_k)\, \ind_{(r_1,\ldots,r_m)}
(q_1,\ldots,q_n)\ |\  
\ (\varphi\vee\varphi)\ |\ (\varphi\wedge\varphi)$$
where $p$ and each of the symbols $p_i,r_i,q_i$
are proposition symbols in $\prop$. We denote by $\mathit{FOR}(\I)$ the
set of formulae of $\I$.
The numbers $k$ and $n$ are positive integers and $m$ a non-negative integer.
When $m = 0$, the formula
$$(p_1,\ldots,p_k)\, \ind_{(r_1,\ldots,r_m)}
(q_1,\ldots,q_n)$$
is written $(p_1,\ldots,p_k)\, \ind\, 
(q_1,\ldots,q_n)$. Also, when any of the three tuples of
proposition symbols in the formula $(p_1,\ldots,p_k)\, \ind_{(r_1,\ldots,r_m)}
(q_1,\ldots,q_n)$ contains exactly one formula, the brackets
around the tuple are usually left out, as
for example in the formula $p\, \ind_{r}\, q$.
Notice that, in line with the definition of $\I$ in \cite{Konti14, Kont15},
negation and $\ind$ can only be applied to propositional symbols. In particular, 
formulae $$(p_1,\ldots,p_k)\ind_{(r_1,\ldots,r_m)} (q_1,\ldots,q_n)$$ may not occur negated. 
The same convention applies to \emph{independence atoms} in
Independence Logic \cite{GV13}.
%
%
%In this respect $\D$
%
%and $\I$ are not symmetric analogues of each other
%

%
The semantics of $\I$ is defined with 
respect to SD-models.
Intuitively, the formula 
$$(p_1,\ldots,p_k)\ind_{(r_1,\ldots,r_m)} (q_1,\ldots,q_n)$$
asserts that when the truth values of the proposition
symbols $r_1,\ldots,r_m$ are fixed, then the
tuples of truth values of $(p_1,\ldots,p_k)$
and $(q_1,\ldots,q_n)$ are \emph{informationally
independent} in a way explained further  
after the formal truth definition of $\I$.
%

%
\begin{comment}
%
The truth definition of the operator $\ind$ extends the one in the logic \I, as follows.
%
$W,w\models_{\indlang} 
%(\varphi_1,\ldots,\varphi_k)\, \ind_{(\psi_1,\ldots,\psi_m)} (\chi_1,\ldots,\chi_n)$ 
(\fo_1,\ldots,\fo_k)\, \ind_{(\fod_1,\ldots,\fod_m)}(\fob_1,\ldots,\fob_n)$ 
iff for all $w_{1},w_{2}\in W$ that agree on $\fod_1,\ldots,\fod_m$, i.e., 
are such that $w^{*}_{1}(\fod_i)=w^{*}_{2}(\fod_i) = b_{i}$ for $i = 1, \ldots m$, 
there exists $v\in W$ such that 
%
\[
\big(\bigwedge_{i\leq m}
%
v^{*}(\fod_i)= b_i \wedge
%
\bigwedge_{i\leq k} v^{*}(\fo_i) = w^{*}_{1}(\fo_i) \wedge 
%
\bigwedge_{i\leq n} v^{*}(\fob_i)=w^{*}_{2}(\fob_i) \big). 
\]
%
\end{comment}
%%%%%%%%

%
We use $\Vdash_{\I}$ as the semantic
turnstile of $\I$.
The formal semantic clauses of $\I$  
for propositional literals and Boolean connectives
are exactly the same as those for the logic $\D$, while 
the semantics of the formulae
$$(p_1,\ldots,p_k)\ind_{(r_1,\ldots,r_m)} (q_1,\ldots,q_n)$$
is defined as follows. 
Let $W$ be a state description model.
We define
$$W\Vdash_{\I}
%(\varphi_1,\ldots,\varphi_k)\, \ind_{(\psi_1,\ldots,\psi_m)} (\chi_1,\ldots,\chi_n)$
(p_1,\ldots,p_k)\, \ind_{(r_1,\ldots,r_m)}(q_1,\ldots,q_n)$$
iff for all $w_{1},w_{2}\in W$ that agree on $r_1,\ldots,r_m$ (i.e., 
are such that $w_{1}(r_i)=w_{2}(r_i)$ for each $i\in \{1, \ldots, m\}$), 
there exists some $v\in W$ such that 
\[
\big(\bigwedge_{i\leq m}
v(r_i)= w_1(r_i)\ \wedge\
\bigwedge_{i\leq k} v(p_i) = w_1(p_i) \wedge
\bigwedge_{i\leq n} v(q_i)=w_2(q_i) \big). 
\]
Thus, the formula $(p_1,\ldots,p_k)\ind_{(r_1,\ldots,r_m)} (q_1,\ldots,q_n)$ 
asserts that for every tuple $(b_1,\ldots,b_m)$ of truth values 
for the propositions $r_1,\ldots,r_m$, if we restrict attention to the 
set $S$ of those assignments in $W$ that
assign the values $b_1,\ldots,b_m$ to $r_1,\ldots,r_m$,
%\vmnote{Uniformise terminology: ``valuation'' or ``assignment''?} 
then the following condition holds:
the tuples of truth values for
$(p_1,\ldots,p_k)$ and $(q_1,\ldots,q_n)$ are informationally
independent of each other on $S$
in the sense that for every two assignments $w_{1},w_{2}\in S$, there is  an
assignment $v\in S$ that combines $w_{1}$ restricted to $(p_1,\ldots,p_k)$ with $w_{2}$ restricted to $(q_1,\ldots,q_n)$.
%
%
%
%For the reader that sees this definition for the
%first time, it may be instructive to
%consider the meaning of the simple formula $p\, \ind_r\, q$ first.
%

%
It is worth noting that, intuitively, the formula $p\, \ind\, q$ can be interpreted to state that \emph{nothing new can be concluded} about the truth value of $p$ in a possible world $w$
by finding out the truth value of $q$ in $w$ (and vice versa): an agent who fully knows the model $W$ but has no idea which $w\in W$ is
the \emph{actual} world, cannot conclude anything new 
about the truth value of $q$ in
the actual world by learning the truth value of $p$ in that world.

This interpretation explains the initially perhaps counterintuitive fact that 
the formula $p\, \ind\, p$ is satisfiable even in nonempty models: 
$W\Vdash_{\I} p\, \ind\, p$ holds iff $p$ is constant in $W$, i.e., if every
assignment in $W$ gives the same truth value to $p$.
Indeed, $p$ being constant
means exactly that nothing \emph{new}
can be concluded about the truth value $p$ in the actual world by learning the
truth value of $p$ in the actual world.
If $p$ was not constant, the truth value  of $p$ in the actual world would
obviously reveal new information.
It is worth  pointing out here that in the  semantics of $\D$ and $\I$,
there is no explicit actual world present.
Next we will consider the logics $\mathcal{L}_{\dep}$ and $\mathcal{L}_{\ind}$
whose semantics are given in a way similar to Kripke semantics in
terms of pairs $(W,w)$, where $W$ is a SD-model
and $w\in W$ an assignment which can be
taken to correspond to an appointed actual world.
\section{Propositional logic of determinacy $\dlang$}\label{depsection}
We now introduce a new logic which extends propositional logic PL with 
dependence formulae $\dep(\varphi_1,\ldots,\varphi_k ; \psi)$, where
$\varphi_1,\ldots,\varphi_k,\psi$ are arbitrary formulae in the language.
We call this logic 
\emph{Propositional Logic of Determinacy}
and denote it by $\mathcal{L}_{\dep}$.

Recall that $\prop$ denotes a fixed countably infinite set of proposition symbols. 
The formulae of $\dlang$ over $\prop$ 
are defined by the following grammar.
\[
\varphi ::= \ p\ |\ \neg\varphi\ |\
(\varphi\rightarrow\varphi)\ |\
\dep(\varphi,\ldots,\varphi\, ;\varphi)
\]
where $p\in\prop$, and where the tuple $(\varphi,\ldots,\varphi\, ;\varphi)$
contains $k + 1$ formulae for any $k\in\mathbb{N}$.
We consider the Boolean connectives $\land, \lor, \leftrightarrow$ as
abreviations in the usual way. When $k = 0$, we write $\cdep\varphi$
instead of $\dep(\epsilon ;\varphi)$, where $\epsilon$ is the empty sequence of formulae. 
We let $\dfor$ denote the set of formulae of $\dlang$. Notice indeed that
each of the operators $\neg, \rightarrow, \dep$ can be freely used in the language of $\mathcal{L}_{\dep}$; no syntactic restrictions apply.
Intuitively, $\dep(\varphi_{1},\ldots,\varphi_{k}; \psi)$ means that the truth value of $\fo$ is
determined by the set of truth values of the formulae $\varphi_{1},\ldots,\varphi_{k}$ on the
SD-model in consideration.  
In particular, $\cdep\fo$  means that the
truth value of $\fo$ is constant in the model.

We define truth of an \dlang-formula $\fo$ at a state
description $w$ in an SD-model $W$, denoted\footnote{We
sometimes write $W,w \models_{\mathcal{L}_{\dep}} \fo$ instead or even
simply $W,w \models\fo$.} $W,w \models_{\dep} \fo$,
inductively on the structure of formulae as follows.

\[
\begin{array}{lll}
W,w\models_{\dep} p & \Iff & w(p) = 1\\
W,w\models_{\dep} \neg\varphi & \Iff & W,w\not\models_{\dep}\varphi\\
W,w\models_{\dep} \varphi\rightarrow\psi& \Iff & W,w\not\models_{\dep}\varphi
\text{ or }W,w\models_{\dep}\psi\\
W,w\models_{\dep} \dep(\varphi_1,\ldots,\varphi_k;\psi) & \Iff  & 
\mbox{for } \text{ all } u,v\in W, \mbox{ if } 
%\bigwedge\limits_{i\, \in\, \{1,\ldots,k\}}
\text{ the }\text{ equivalence }\\
& & (W,u\models_{\dep} \varphi_i\Leftrightarrow W,v\models_{\dep} \varphi_i)
\text{ }\text{ holds }\text{ for}  \\ 
                  &   & 
\mbox{all $i\leq k$, then } 
(W,u\models_{\dep} \psi \Leftrightarrow W,v\models_{\dep} \psi)
\end{array}
\]

%with the following semantic definition of  $\dep$: 
%\anote{
%$W,w \models \dep(\fo_{1},\ldots,\fo_{k}; \fob)$ \Iff 
%%
%\begin{equation}\label{dependencecondition}
%%
%\biggl(\bigwedge\limits_{i\, \in\, \{1,\ldots,k\}}
%(W,u\models\varphi_i\Iff W,v\models\varphi_i) \biggr)\ \Rightarrow\ 
%(W,u\models\fo \Iff W,v\models\fo )
%%
%\end{equation}
%%
%}

When an SD-model $W$ is fixed, 
the assignments $w\in W$ can be
extended to truth assignments $w^{*}: \dfor \to \{0,1\}$
in the natural way. The same, of course, 
applies to $\mathcal{L}_{\mathcal{U}}$.
%
%Like in $\mathcal{L}_{\mathcal{U}}$, we
%define $W\models_{\dep}\varphi$ iff $W,w\models_{\dep}\varphi$ for all $w\in W$,
%and we define $\models_{\dep}\varphi$ iff $W\models_{\dep}\varphi$ for all SD-models $W$.

Note that the truth definition of $\dep(\fo_{1},\ldots,\fo_{k};\fo)$ extends the semantics of \dep in the logic \D and does not here depend on the current state description $w$ but only on the entire SD-model $W$. In particular, the semantics of $\cdep$ is as
follows: $W,w \models \cdep\fo$ iff for all $u,v\in W$, we have that $W,u\models\fo \Iff W,v\models\fo$. In other words, $W,w \models \cdep\fo$ iff the truth value of $\fo$ is constant in the model. 
%Therefore, the truth of $\cdep\fo$ does not depend on $w$. 

Let us rephrase the semantic definition  of  $\dep$ above in terms of explicit functional dependence. 
%Each  SD-model $W$ determines a global truth function
%$$\tr: \dfor \times W \to \{0,1\}$$
%such that $\tr(\fo,w) = 1$ iff $W,w\models\varphi$.
Every formula $\fo \in \dfor$ determines the
function $\tr_{\fo}: W \to \{0,1\}$
such that $\tr_{\varphi}(w) := w^*(\varphi)$.
%
%The function $\tr_{\varphi}$ projects $\tr$ into
%
%the truth function of $\varphi$  on $W$.
%
Now we have $W,w \models \dep(\fo_{1},\ldots,\fo_{n}; \fo)$ if and only if  
$\tr_\fo$ is determined by $\tr_{\fo_1},\ldots,\tr_{\fo_n}$ on $W$ 
in the sense of Definition \ref{determinacy definition}, i.e., 
$$W,w \models \dep(\fo_{1},\ldots,\fo_{n}\, ;\, \fo) \Iff  \Dep_W(\tr_{\fo_{1}},\ldots,
\tr_{\fo_{n}}\, ;\, \tr_\fo).$$
The semantics of $\mathcal{L}_{\dep}$ is similar to Kripke semantics (in
the case where the accessibility relation of a model $W$ is the universal
relation $W\times W$). Therefore it is natural to define for $\mathcal{L}_{\dep}$
the notions of 
validity
%\emph{validity in a model} 
and
satisfiability analogously to the corresponding definitions in modal logic:
\begin{definition}\label{validityetc.}
Let $\fo\in \dfor$.
\begin{enumerate}
\item $\fo$ is \emph{valid in an SD-model $W$}  
if $W,w \models \fo$ for every $w\in W$. We
write $W\models\varphi$ if $\varphi$ is valid in $W$.

\item $\fo$ is \emph{valid} (or \emph{SD-valid}),
if $W \models \fo$ for every SD-model $W$.
We write $\models\varphi$ if $\varphi$ is valid.
 
\item $\fo$ is \emph{satisfiable}
if $W,w \models \fo$ for some SD-model $W\not= \emptyset$ and some $w\in W$.
\end{enumerate}
\end{definition}

The following definition is analogous to the definition of local
equivalence in modal logic:
\begin{definition}\label{equivalencedefinition}
Let $\varphi$ and $\psi$ be formulae of $\mathcal{L}_{\dep}$.
We write $\varphi\equiv_{\mathcal{L}_{\dep}}\psi$ if
the equivalence $W,w\models \varphi\Leftrightarrow W,w\models\psi$
holds for all SD-models $W$ and all $w\in W$. We will often omit 
the subscript $\mathcal{L}_{\dep}$ when no confusion arises.
\end{definition}
We note that validity in an SD-model can be expressed locally in \dlang in
the sense that
for all assignments $w\in W$, we have $W\models \fo$ iff $W,w\models \fo \land \cdep \fo$.
If we know that $W\not=\emptyset$,
then $W\models \fo$ iff $W,w\models \fo \land \cdep \fo$ for some $w\in W$.
Thus $\fo \land \cdep \fo$ plays the role of the \emph{universal modality}.
Hence we define the \emph{notation} $\ubox\varphi := \varphi\wedge\cdep\varphi$
as well as
$\udiam\varphi := \neg\ubox\neg\varphi$ (which can easily be seen to be
equivalent to $\varphi\vee\neg\cdep\varphi$). Note that $\ubox\varphi$
and $\udiam\varphi$ are simply abbreviations of formulae of $\mathcal{L}_{\dep}$,
they are not formulae of $\mathcal{L}_{\mathcal{U}}$. (The box and diamond
operators in $\mathcal{L}_{\mathcal{U}}$ are denoted by $\Ubox$ and $\Udiam$
instead of $\ubox$ and $\udiam$.)
%

%%%%%%%%%%%%%
\vcut{
The language of \emph{propositional logic with the universal modality} $\ulang$ is
given by the grammar
$$\varphi\, ::=\, p\, |\, \neg\varphi\, |\, (\varphi\rightarrow\varphi)\, |\, \ubox\varphi,$$
where $p\in\prop$. The semantics can be given classically, with respect to SD-models treated as Kripke models with universal accessibility relation, by the clause for $\ubox$: 
\[
\begin{array}{lll}
%%
%W,w\models p & \Iff & w(p) = 1\\
%%
%W,w\models \neg \varphi& \Iff & W,w\not\models\varphi\\
%%
%W,w\models \varphi\rightarrow\psi & \Iff & W,w\not\models\varphi \text{ or }
%%
%W,w\models\psi\\
%
W,u\models\ubox\varphi & \Iff & W,u\models\varphi \text{ for all } u\in W
\end{array}
\]

Again, we write $W\models\varphi$ iff  $W,w\models\varphi$ for all $w\in W$.
}
%%%%%%%%%%%

\medskip
In turn, \dep is expressible in terms of $\ubox$ in the following sense:

\begin{proposition}\label{prop:Dep}
$$
\dep(\varphi_1,\ldots, \varphi_k; \psi)\ \equiv
\bigvee\limits_{\chi\, \in\, \nff(\varphi_1,\ldots,\varphi_k)}
\ubox (\chi\leftrightarrow\psi).
$$
\end{proposition}
\begin{proof}
Let us denote the formula on the right hand side by $\Delta(\varphi_1,\ldots, \varphi_k; \psi)$. 
Consider any SD-model $W$. Define $\extn{\psi}_{W}
%W(\psi) 
 := \{ w\in W \mid W,w\models \psi \}$. 
For each  $w\in W$, let $\chi_{w}$ be the unique type over 
$\{ \varphi_1,\ldots, \varphi_k \}$ which is true at $w$; in the case $k=0$,
let $\chi_w :=\top$. Define 
\[ \chi_{(W,\psi)} := \bigvee \{\chi_{w} \mid w  \in \extn{\psi}_{W} \} .\]
%\[ \chi := \bigvee\limits_{\chi_{w}  \in W(\psi)} \chi_{w}.\]

Now, suppose $W,w \models \dep(\varphi_1,\ldots, \varphi_k; \psi)$. We claim that  $W,w \models \ubox (\chi_{(W,\psi)} \leftrightarrow\psi)$.
Indeed, take any $v\in W$. 
If $W, v \models \psi$, then $v\in \extn{\psi}_{W}$ and $W, v \models \chi_{v}$ by definition, whence  $W, v \models \chi_{(W,\psi)}$. 
Conversely, suppose  $W, v \models \chi_{(W,\psi)}$. Then 
$W, v \models \chi_{u}$ for some $u\in \extn{\psi}_{W}$.  Therefore 
$W, u \models \psi$ and $\chi_{v} = \chi_{u}$, i.e., for each $i = 1,\ldots, k$, we have that $W,u\models\varphi_i$  iff  $W,v \models\varphi_i$. Since  $W,w\models \dep(\varphi_1,\ldots, \varphi_k; \psi)$, we thus have $W,u\models\psi$  iff $W,v \models \psi $. Therefore we
infer that $W, v \models \psi$. Thus $W, v \models \chi_{(W,\psi)} \leftrightarrow\psi$ for every $v\in W$. Hence $W,w \models \ubox (\chi_{(W,\psi)} \leftrightarrow\psi)$, so $W,w \models \Delta(\varphi_1,\ldots, \varphi_k; \psi)$.

Conversely, suppose $W,w\models \Delta(\varphi_1,\ldots, \varphi_k; \psi)$. 
Thus $W,w\models \ubox (\chi \leftrightarrow\psi)$ for some formula 
$\chi\, \in\, \nff(\varphi_1,\ldots,\varphi_k)$.  Suppose $u,v \in W$ are such that 
$W,u\models\varphi_i$  iff  $W,v \models\varphi_i$ for each $i = 1,\ldots, k$.  Then 
$W,u\models\chi$  iff  $W,v \models\chi$, whence  $W,u\models\psi$  iff  $W,v \models\psi$ because $W,u \models \chi \leftrightarrow\psi$ and  
$W,v \models \chi \leftrightarrow\psi$. Thus we have proved that $W,w\models \dep(\varphi_1,\ldots, \varphi_k; \psi)$.
\end{proof}

Proposition \ref{prop:Dep} establishes that the operator $\dep$ is
definable in terms of $\ubox$.
As a particular case of this definability,
we obtain that $\cdep \psi \equiv \ubox \psi \lor   \ubox \lnot \psi$.
Since $\dep$ and $\ubox$ are interdefinable, it follows that the logics $\mathcal{L}_{\dep}$ and $\mathcal{L}_{\mathcal{U}}$
are clearly equiexpressive: we can translate $\mathcal{L}_{\dep}$
into $\mathcal{L}_{\mathcal{U}}$ by making use of Proposition \ref{prop:Dep},
and, on the other hand, we can translate $\mathcal{L}_{\mathcal{U}}$
into $\mathcal{L}_{\dep}$ with the help of our earlier observation that
the formula $\varphi\wedge\cdep\varphi$ simulates the
universal modality. Below we will make significant use of
this interdefinability of the operators $\dep$ and $\ubox$.
The above results suggest obvious equivalence-preserving
translations between $\mathcal{L}_{\dep}$ and $\mathcal{L}_{\mathcal{U}}$.
However, while this connection
between $\mathcal{L}_{\dep}$ and $\mathcal{L}_{\mathcal{U}}$ is
interesting and useful,
the primary aim of this paper is to study the notion of
determinacy and the operator $\dep$ taken as a primitive.
Indeed, one of our principal objectives is to compare Propositional
Dependence Logic $\mathcal{D}$ and the new logic $\mathcal{L}_{\dep}$ and
investigate how well they can be used in order to formalise statements about
propositional determinacy and how well the two logics
relate to natural language. The interdefinability of $\mathcal{L}_{\dep}$
and $\mathcal{L}_{\mathcal{U}}$
can thus be regarded as an interesting fact that nevertheless
will play mainly a technical role in this paper.
Furthermore, neither of the above mentioned translations is polynomial,
and in the more general framework of Kripke semantics (see Section \ref{Kripke} below) it is
not always possible to define the box modality $\Box$ in terms of the
determinacy operator $\dep$. We leave
investigations of that more general framework for the future.
%
%Alternative polynomial translations may exist, but we will
%
%not concern ourselves with that issue here.
%

%
\begin{comment}
%
In this paper we are
interested in the operator $\dep$ taken as a primitive,
the focus being on in the comparison of propositional
dependence logic $\D$ and the novel logic $\mathcal{L}_{\dep}$.
%
So far, note that even though they turn out equally expressive, the size of the translation of $\dep$ into $\mathcal{L}_{\mathcal{U}}$ described above is not polynomially bounded.
%[Why "so far"?]
\end{comment}
%

%
The following is a straightforward
observation about $\mathcal{L}_{\dep}$ which we will use later on.
%

%
%The definition of \dep in terms of $\ubox$ given above, however, produces an exponential %blowup in the length of the formula. Besides, we are interested here in the properties of \dep %taken as a primitive.
%

%
\begin{lemma}
\label{lem:Dep-ER}
The logic \dlang satisfies the equivalent replacement property ER with
respect to the equivalence given in Definition \ref{equivalencedefinition}.
\end{lemma}

We say that a class $\mathcal{C}$ of SD-models is
\emph{definable} in $\dlang$ (resp., in $\ulang$), if
there exists a formula $\varphi$ of $\dlang$ (resp., of $\ulang$)
such that $W\models\varphi$ iff $W\in\mathcal{C}$.

Let $w:\mathit{PROP}\rightarrow \{0,1\}$ be an assignment function,
and let $\Phi\subseteq\prop$.
%
%$\Phi\not=\emptyset$.
%
We let $w|_{\Phi}$ denote the restriction of $w$
to $\Phi$, i.e., the function
$f:\Phi\rightarrow \{0,1\}$ defined so that $f(p) = w(p)$ for all $p\in\Phi$.
We define $W |_{\Phi} := \{\, w|_{\Phi}\ |\ w\in W\, \}$.
\begin{definition}\label{equivalentmodelsdefinition}
Two SD-models $W_{1}$ and $W_{2}$ are \emph{$\Phi$-equivalent}, denoted $W_{1} \equiv_{\Phi} W_{2}$, 
if  $W_{1} |_{\Phi} = W_{2} |_{\Phi}$. 
\end{definition}
\begin{definition}\label{finpropequivalence}
A class of SD-models $\mathcal{C}$ is
\emph{closed under finite propositional equivalence}, if the following conditions hold.
\begin{enumerate}
\item
$\emptyset\in\mathcal{C}$. 
\item
There exists a finite set
$\Phi\subseteq\prop$ such that for all nonempty SD-models $W_{1}$ and $W_{2}$, if
$W_{1}\in\mathcal{C}$ and $W_{1} \equiv_{\Phi} W_{2}$, then $W_{2}\in\mathcal{C}$.
\end{enumerate}
\end{definition}
The first condition above has been included for technical convenience.
Note that the empty model satisfies every
formula of $\ulang$ and $\mathcal{L}_{\dep}$.
%

%%
%\anote{It may be worth it to rename finite propositional
%%
%equivalence, since it is not a property between SD-models (since
%%
%the particular finite proposition symbol set must be fixed).}

%
\begin{proposition}\label{expressivitycharacterization}
A class of SD-models is definable in $\ulang$
iff it is closed under finite propositional equivalence.
\end{proposition}

\begin{proof}
Suppose a class of SD-models $\mathcal{C}$ is definable in $\ulang$ by some formula $\varphi$ and let $\Phi\subseteq\prop$ be the set of  proposition symbols that occur in $\varphi$. Then  for all SD-models $W_{1}$ and $W_{2}$, if
$W_{1}\in\mathcal{C}$ and $W_{1} \equiv_{\Phi} W_{2}$,
then $W_{1}\models\varphi$ and thus
$W_{2}\models\varphi$, whence $W_{2}\in\mathcal{C}$. 
Thus $\mathcal{C}$ is closed under finite propositional equivalence.

Now, suppose  $\mathcal{C}$ is closed under finite propositional equivalence and let $\Phi\subseteq\prop$ be a finite set such that for all
nonempty SD-models $W_{1}$ and $W_{2}$, if $W_{1}\in\mathcal{C}$ and $W_{1} \equiv_{\Phi} W_{2}$, then $W_{2}\in\mathcal{C}$. Define a \emph{characteristic formula} 
$\varphi_{W}$ of a model $W\not=\emptyset$ as
follows. For each $w\in W$, let $\chi_w$ be the
unique propositional
type in $\mathit{Conj}(\Phi)$ such that $W,w\models\chi_w$. Define
$$\varphi_{W}\ :=\ \bigl(\bigwedge\limits_{w\in W}\udiam \chi_w\bigr)\ 
\wedge\ \ubox\bigl(\bigvee\limits_{w\in W}\chi_w\bigr),$$
which of course can be regarded as a finite formula since we can eliminate repeated
conjuncts and disjuncts.
Let 
%$\Delta(\Phi,\mathcal{C})$ be the disjunction
%
$$\Delta(\Phi,\mathcal{C}) := \bigvee\limits_{W\ \in\
\mathcal{C}\setminus\{\emptyset\}}\varphi_W,$$
which again can be regarded as a finite formula.
Then $\mathcal{C}$ is defined by $\Delta(\Phi,\mathcal{C})$. Indeed, $\Delta(\Phi,\mathcal{C})$ is true in every model $W \in \mathcal{C}$. Conversely, every SD-model $U\not=\emptyset$ satisfying $\Delta(\Phi,\mathcal{C})$
satisfies some disjunct, i.e., some characteristic formula $\varphi_W$ of some model $W\in\mathcal{C}\setminus\{\emptyset\}$, because the truth value of
each disjunct of $\Delta(\Phi,\mathcal{C})$ is constant across the
worlds of a given SD-model. Clearly
$U \equiv_{\Phi} W$, whence we have $U \in \mathcal{C}$.
\end{proof}

Consequently, 
%Since $\dlang$ and $\ulang$ have the same expressiveness, 
we obtain a characterisation of the expressive power of $\mathcal{L}_{\dep}$:

\begin{corollary}\label{expressivitycharacterization2}
A class of SD-models is definable in $\dlang$
iff it is closed under finite propositional equivalence.
\end{corollary}
Let $\Phi\not=\emptyset$ be a \emph{finite} subset of $\mathit{PROP}$.
Let $\mathcal{C}$ denote the set of all SD-models in restriction to $\Phi$,
i.e., the set $\{\, W|_{\Phi}\ |\ W\text{ is an SD-model}\, \}$.
We call $\mathcal{C}$ 
the set of \emph{$\Phi$-models} and denote it by $\mathcal{M}(\Phi)$.
Let $\mathcal{S}\subseteq\mathcal{M}(\Phi)$.
We say that $\mathcal{S}$
is definable in $\mathcal{L}_{\mathcal{U}}$ $(\mathcal{L}_{\dep}$)
in restriction to $\mathcal{M}(\Phi)$, if there is a
formula $\varphi$ of $\mathcal{L}_{\mathcal{U}}$ ($\mathcal{L}_{\dep}$)
Such that for all $W\in\mathcal{M}(\Phi)$, we
have $W\models\varphi$ iff $W\in\mathcal{S}$.
\begin{proposition}
\label{expressivecompleteness}
$\mathcal{L}_{\mathcal{U}}$ and $\mathcal{L}_{\dep}$ are
expressively complete in the sense that for any finite nonempty
$\Phi\subseteq\mathit{PROP}$ and any $\mathcal{S}\subseteq\mathcal{M}(\Phi)$
with $\emptyset\in\mathcal{S}$, the set $\mathcal{S}$ is
definable in restriction to $\mathcal{M}(\Phi)$ in
both $\mathcal{L}_{\mathcal{U}}$ and $\mathcal{L}_{\dep}$.
\end{proposition}
\begin{proof}
The claim for $\mathcal{L}_{\mathcal{U}}$ is
established by an argument that is almost identical to
the proof of Proposition \ref{expressivitycharacterization}.
The claim for $\mathcal{L}_{\dep}$ then follows by
the equiexpressivity of 
the logics $\mathcal{L}_{\dep}$ and $\mathcal{L}_{\mathcal{U}}$.
\end{proof}

\section{Propositional Logic of Independence \indlang}\label{indsection}

We have defined the logic of $\dlang$ as 
%Like the Propositional Dependence Logic $\D$, it is 
an extension of propositional logic PL with the operator $\dep$. 
Next we introduce 
\emph{Propositional Logic of Independence} $\mathcal{L}_{\ind}$ which extends 
PL in a similar way, but now with the operator $\ind$ instead of $\dep$.
The logic $\mathcal{L}_{\ind}$ relates to Propositional Independence Logic $\I$
analogously to the way $\mathcal{L}_{\dep}$ relates to
Propositional Dependence Logic $\D$.
%

%\anote{There are of  course several notions of independence,
%%
%and the logics below only aim to  capture one such notion.
%%
%A short analysis of this particular notion is provided at the end
%%
%of Section \ref{indsection}.}
%\vnote{What are you referring to?} 

%
The language of \emph{Propositional Logic of Independence}
$\indlang$ is given by the following grammar.
$$\varphi ::=  p\ |\ \neg\varphi\ |\
(\varphi\rightarrow\varphi)\ |\
(\varphi,\ldots,\varphi)\, \ind_{(\varphi,\ldots,\varphi)}
(\varphi,\ldots,\varphi),$$
where $p\in\prop$, and each of the three tuples $(\varphi,\ldots,\varphi)$ in
the expression
$$(\varphi,\ldots,\varphi)\, \ind_{(\varphi,\ldots,\varphi)}(\varphi,\ldots,\varphi)$$
is a finite tuple of formulae; the tuples in the same expression may be of different lengths,
but only the tuple in the subscript may possibly be empty.
Instead of writing
$(\varphi,\ldots,\varphi)\, \ind_{\emptyset}\, (\varphi,\ldots,\varphi),$
we simply write $(\varphi,\ldots,\varphi)\, \ind_{}\, (\varphi,\ldots,\varphi)$.
As in \dlang, we consider the Boolean
connectives $\land, \lor, \leftrightarrow$ definable as usual. 
Sometimes we  leave out brackets of formulae of $\mathcal{L}_{\ind}$, following the convention that the operator $\ind$ has a higher priority than all binary connectives, while negation has a higher priority than all other operators (including $\ind$).
We let $\mathit{FOR}(\mathcal{L}_{\ind})$ denote the  set of
formulae of $\mathcal{L}_{\ind}$.
%

%%
%where the first and third tuples (from left to right)
%%
%in $(\varphi,\ldots,\varphi)\, \ind_{(\varphi,\ldots,\varphi)}
%%
%(\varphi,\ldots,\varphi)$ are nonempty finite tuples,
%%
%and the second tuple 
%%
%is a finite, possibly empty tuple.
%

%
The semantics of $\indlang$ is similar to that of $\mathcal{L}_{\dep}$ and also defined with
respect to SD-models $W$ and assignments $w\in W$.
Propositional symbols and Boolean operators are interpreted exactly as in $\dlang$.
To define the semantics for $\ind$, recall
that in an SD-model $W$ and for a logic with a  Kripke-style semantics,
each $w\in W$ can be extended to a truth function 
$w^{*}$ from the set of formulae to $\{0,1\}$ such that $w^{*}(\varphi)=1$ iff $W,w\models\varphi$.
The truth definition of the operator $\ind$ extends the one in the logic \I as follows. 
We define $W,w\models_{\indlang} 
%(\varphi_1,\ldots,\varphi_k)\, \ind_{(\psi_1,\ldots,\psi_m)} (\chi_1,\ldots,\chi_n)$ 
(\fo_1,\ldots,\fo_k)\, \ind_{(\fod_1,\ldots,\fod_m)}(\fob_1,\ldots,\fob_n)$ 
iff for all $w_{1},w_{2}\in W$ that agree on $\fod_1,\ldots,\fod_m$
(i.e., are such that $w^{*}_{1}(\fod_i)=w^{*}_{2}(\fod_i)$ for all $i\in\{1,\ldots, m\}$),
there exists some $v\in W$ such that 
\[
\big(\bigwedge_{i\leq m}
v^{*}(\fod_i)= w_1^*(\theta_i)\wedge
\bigwedge_{i\leq k} v^{*}(\fo_i) = w^{*}_{1}(\fo_i) \wedge 
\bigwedge_{i\leq n} v^{*}(\fob_i)=w^{*}_{2}(\fob_i) \big). 
\]

%\begin{multline*}
%%
%\big(\bigwedge_{i\leq m} u(\psi_i)=u'(\psi_i)\big)\\
%%
%\Rightarrow
%%
%\big(\bigwedge_{i\leq m}
%%
%v(\psi_i)=u(\psi_i) \wedge
%%
%\bigwedge_{i\leq k} v(\fo_i)=u(\fo_i) \wedge 
%%
%\bigwedge_{i\leq n} v(\fob_i)=u'(\fob_i) \big).
%%
%\end{multline*}
Thus the operator $\ind$ of $\mathcal{L}_{\ind}$ extends $\ind$ of $\I$
so that in $\mathcal{L}_{\ind}$ the operator can be 
applied to all formulae, not 
only proposition symbols.
Note that the semantics of $\ind$
(which we defined with respect to the model $W$ and 
world $w\in W$) does
not directly depend on the world $w$ but is global in
the model.
Analogously to the conventions fixed in Definition \ref{validityetc.}
for $\mathcal{L}_{\dep}$, we say that a
formula $\fo$ of $\indlang$ is \emph{valid in a model} W (denoted 
$W \models_{\indlang}\fo$), if $\fo$ is true in every world of $W$, and that 
$\varphi$ is \emph{valid} (denoted $\models_{\mathcal{L}_{\ind}}\varphi$)
if $W\models_{\mathcal{L}_{\ind}}\varphi$ for every model $W$.
Two formulae $\fo$ and $\fob$ of $\indlang$ are \emph{equivalent}, denoted 
$\fo \equiv_{\indlang} \fob$, if the equivalence $W,w\models_{\mathcal{L}_{\ind}}\varphi
\Leftrightarrow W,w\models_{\mathcal{L}_{\ind}}\psi$ holds for all models $W$
and all $w\in W$.
The following two lemmas are straightforward to prove.
For Lemma \ref{blocklemma}, recall the definition of $\Phi$-equivalence
of SD-models from Definition \ref{equivalentmodelsdefinition}.
\begin{lemma}
\label{lem:Ind-ER}
The logic \indlang satisfies the equivalent replacements property ER
with respect to $\equiv_{\mathcal{L}_{\ind}}$.
\end{lemma}

%\begin{proof}
%Straightforward.
%\end{proof}

\begin{lemma}\label{blocklemma}
Let $\varphi$ be a formula of $\mathcal{L}_{\dep}$ or $\mathcal{L}_{\ind}$.
Let $\Phi$ be the set of proposition symbols occurring in $\varphi$.
For nonempty models $W$ and $U$ and points $w\in W$
and $u\in U$, if $W\equiv_{\Phi} U$, then we have $W,w\models\varphi$ iff $U,u\models\varphi$.
\end{lemma}
%
%
%%
%\begin{proof}
%%
%Straightforward.
%%
%\end{proof}
%%
%

Since both logics \dlang and \indlang are interpreted
with respect to SD-models $W$ and assignments $w\in W$, it is
easy to compare them. We define the following simple translation $t$
from $\mathcal{L}_{\dep}$ into $\mathcal{L}_{\ind}$:
\begin{enumerate}
\item
$t(p)\ :=\ p$ for $p\in\mathit{PROP}$
\item
$t(\neg\varphi)\ :=\ \neg t(\varphi)$
\item
$t(\varphi\rightarrow\psi)\ :=\ t(\varphi)\rightarrow t(\psi)$
\item
$t(\dep(\varphi_1,\ldots,\varphi_k\,;\, \psi))\, :=\,
t(\psi)\, \ind_{(t(\varphi_1),\ldots,t(\varphi_k))}\, t(\psi)$.
\end{enumerate}
In particular, we have $t(\cdep\psi) := t(\psi)\, \ind\, t(\psi)$.
(Recall our discussion in
Section \ref{indeplogic} concerning formulae of the type $p\, \ind\, p$.)
Now we make some simple but interesting observations.
\begin{proposition}\label{expr}
$\mathcal{L}_{\dep}$ embeds into $\mathcal{L}_{\ind}$
and $\D$ embeds  into $\I$, in the following sense. 
\begin{enumerate}
\item
For each formula $\varphi$ of $\mathcal{L}_{\dep}$,
there exists a formula $\psi$ of $\mathcal{L}_{\ind}$ equivalent to $\varphi$, i.e., such 
 that $W,w\models\varphi$ iff $W,w\models\psi$ 
 for all $W$ and all $w\in W$. 
 
\item
For each formula $\varphi$ of $\D$,
there exists a formula $\psi$ of\, $\I$
equivalent to $\varphi$, i.e., such that $W\Vdash\varphi$ iff $W\Vdash\psi$ for all $W$. 
\end{enumerate}
\end{proposition} 

\begin{proof}
The proof of the first claim is straightforward, 
using the translation $t$ defined above. 
The second claim is also straightforward, based on the obvious variant $t'$ of
the translation $t$ that keeps proposition symbols and
Boolean connectives the same and translates
$\dep(p_1,\ldots,p_k\, ;\, q)$ to $q\, \ind_{(p_1,\ldots,p_k)}\, q$.
\end{proof}

We note that the translation from $\D$ into
$\I$ mentioned in the above proof is well known from the literature on team semantics.
%
%\vmnote{Reference? Later}
%

%
We define $\uboxp\varphi$ to be an abbreviation
for the formula $\varphi\wedge\,  \varphi \ind \varphi$ of $\mathcal{L}_{\ind}$,
and we let $\udiamp\varphi$ denote $\neg\uboxp\neg\varphi$.
Note that $\uboxp$ corresponds to the universal
modality in an obvious way. To see how the independence 
operator $\ind$ can be expressed in terms of $\udiamp$,
consider a formula
$(\fo_1,...,\fo_k)\, \ind_{(\fod_1,...,\fod_m)}(\fob_1,...,\fob_n)$ 
of $\mathcal{L}_{\ind}$. 
Recall the notation $\conj(\Phi)$ from Section \ref{prelipreliwhatever},
including also the special case for $\Phi =\emptyset$ which
stipulates that $\conj(\emptyset) = \{\top\}$. Define
$$B\, := \conj(\{\varphi_1,\ldots,\varphi_k\}) \times \ \conj(\{\theta_1,
\ldots,\theta_m\}) \times  \conj(\{\psi_1,\ldots,\psi_n\}).$$
%

%
\begin{comment}
%
In the case $m = 0$, let
%
$$B\, := \conj(\{\varphi_1,\ldots,\varphi_k\}) \times \ \{\top\}
%
\times  \conj(\{\psi_1,\ldots,\psi_n\}).$$
%
%
%
The following proposition defines $\ind$ in terms of $\udiamp$.
%
\end{comment}
%

%
\begin{proposition}
\label{prop:Ind}
\begin{multline*}
(\fo_1,...,\fo_k)\, \ind_{(\fod_1,...,\fod_m)}(\fob_1,...,\fob_n) \ 
\\ \\ 
\equiv_{\mathcal{L}_{\ind}}\ \bigwedge\limits_{
(\varphi,\theta,\psi)\, \in\, B}
\Big(\bigl(\udiamp(\fod \wedge \fo)
\wedge\udiamp(\fod \wedge \fob)\bigr)
\rightarrow\ \udiamp(\fod \wedge \fo \wedge \fob) \Big).
\end{multline*}
\end{proposition}
\begin{proof}
It is easy to see that  the bottom formula 
%
%the right hand side of\, $\equiv_{\mathcal{L}_{\ind}}$ 
%
describes the semantics of the operator $\ind$ in
terms of $\udiamp$ in a rather direct way.
\end{proof}

We will complete the expressivity analysis of the logics $\mathcal{L}_{\dep}$,
$\mathcal{L}_{\ind}$, $\D$, and $\I$ in 
Section \ref{express}.

{\begin{comment}
%
The semantics of $\dlang$ in
%
the case of proposition symbols and Boolean connectives, is given
%
by the following clauses.
%
%
%
\[
\begin{array}{lll}
%
W,w\models p & \Iff & w(p) = 1\\
%
W,w\models \neg \varphi& \Iff & W,w\not\models\varphi\\
%
W,w\models \varphi\rightarrow\psi & \Iff & W,w\not\models\varphi \text{ or }
%
W,w\models\psi\\
%
\end{array}
\]
%
\end{comment}}
%
%

%\vnote{Antti, I think the text in blue below is irrelevant here and suggest that you remove it.}
%\anote{This is problematic. One of the two choices below can be used; if we use the first one, we must
%redefine SD-models so that the domain of each state description is a fixed finite set $\Phi$ of proposition symbols.
%\begin{enumerate}
%%
%\item
%Actually, since $\dlang$ and $\ulang$ are equally expressive and can define every class of SD-models
%over $\Phi$, it follows that $\indlang$ embeds into $\dlang$, as well.
%\item
%Actually, since it is clear that any class of SD-models definable in $\mathcal{L}_{\ind}$ is closed
%under finite propositional equivalence, and since $\mathcal{L}_{\dep}$ is
%equiexpressive with $\mathcal{L}_{[u]}$, we observe by
%Proposition \ref{expressivitycharacterization} that $\mathcal{L}_{\ind}$ embeds into
%$\mathcal{L}_{\dep}$, as well.
%%
%\end{enumerate}}
%

%
%\section{Natural language and
%logics of determinacy and independence}\label{naturallanguageinterpretations}
%
\section{Natural language and
logics of determinacy and independence}\label{naturallanguageinterpretations}
In this section we interpret the logics $\D$ and $\dlang$ in relation to natural language and
compare their respective properties. We will not discuss $\I$
explicitly here, but since $\I$ is technically quite similar to $\D$,  
many of the  observations below concerning $\D$ apply to $\I$ as well.
Here we take $W\Vdash\varphi$ (respectively, $W\models\psi$) to
mean that a sensible agent who considers $W$ to be the set of all possible scenarios, 
considers $\varphi$ (resp., $\psi$) to hold. This kind of
reading of $\mathrm{SD}$-models as information states is in line
with the the intuitions of modal logic and also inquisitive logic.
The principal argument of the section is that $\mathcal{L}_{\dep}$ is---at
least in some important respects---a better match than
propositional dependence logic $\mathcal{D}$
with natural language intuitions concerning statements about logical determinacy.
It is sufficient for our purposes to consider formulae with
only the connectives $\neg$ and $\vee$
together with determinacy assertions of the type $\dep(p;q)$.
We make the assumption that
the desirable natural language counterparts of $\neg$ and $\vee$ 
should always be 
``\emph{it is not the case that}"
and ``\emph{or}," respectively.
Formulae $\dep(p;q)$
should correspond to assertions stating that 
``\emph{whether $P$ holds, determines whether $Q$ holds}." 
Here $P$ and $Q$ denote suitable natural language
interpretations of $p$ and $q$.
A different kind of analysis would arise if, for example, $\neg$
was to be read as ``\emph{it is never the case that}"  or
``\emph{it is impossible that}."  
Our argument will proceed as follows. We first argue
that the semantics of formulae of the type $p\vee q$ is a good match
with natural language intuitions in both logics $\mathcal{L}_{\dep}$
and $\mathcal{D}$. We then turn to examples concerning
formulae of the type $\dep(p;q)$, and again argue that
the semantics of both logics is reasonable. (In this 
context we also briefly discuss more complex formulae of the type $\dep(p,q;r)$,
but this is not crucial from the point of view of our discussion.)
We then argue that, despite both $p\vee q$ and $\dep(p;q)$
having a reasonable semantics in $\mathcal{D}$,
formulae of the type $\dep(p;q)\vee\dep(p';q')$, which combine $\vee$ and $\dep$,
are problematic. In fact, we show this even for the 
formula $\dep(p;q)\vee\dep(p;q)$, where both disjuncts are the same.
We then continue by arguing that $\mathcal{L}_{\dep}$, in turn,
gives natural interpretations for these problematic examples.
Finally, we briefly discuss formulae of the type $\neg\dep(p;q)$.
As we saw in the previous section (cf. Proposition \ref{classicalproposition}), team
semantics is simply classical semantics lifted to the level of sets: if $\varphi$ is a formula of
propositional logic PL, then, in the setting of team semantics, $W\Vdash\varphi$ simply
means that each world in $W$ satisifies $\varphi$.
Even the semantics of disjunction $\vee$, which may appear strange at first, makes 
perfect sense from that perspective. Let us consider an example
where team semantics seems to give a correct interpretation (from the 
natural language perspective) to the disjunctive formula $p\vee q$.
Consider the following propositions.
\begin{itemize}
\item
``\emph{The patient has an ear infection}'', encoded by $p$.
\item
``\emph{The patient has high blood pressure}'', encoded by $q$.
\end{itemize}
Assume a set $W = \{w_1,w_2,w_3\}$ of possible scenarios has been
identified by a clinician after inspecting a patient with vertigo,
where 

\medskip

$w_1(p) = 1, w_1(q) = 0$; \  
$w_2(p) = 0, w_1(q) = 1$; \ 
$w_3(p) = 1, w_3(q) = 1$. 

\medskip

%
\begin{comment}
%
Hereafter we will identify any state description with the set of propositional symbols to which it assigns value 1. In that sense, we assume that $w_1 = \{p\}$,  $w_2 = \{q \}$ and $w_3 = \{p,q\}$.
%
\end{comment}
%

%
The set $W = \{w_1,w_2,w_3\}$ is assumed here to be the set of
\emph{exactly all} scenarios which the clinician considers possible.
The clinician has informed the patient about his 
situation, so the patient also considers $W$ to
be the set of all possible scenarios.
To summarize the discussion with
the patient, and to  repeat what 
the situation is, the clinician then states to the patient:
%
%
%
%\begin{equation}\label{one}
%
$$``\emph{So, you have an ear infection or high blood pressure}."$$
%
%\end{equation}
%
The clinician seems to be asserting that
%
%
%
%\begin{equation}\label{two}
%
``$W\Vdash p\vee q,$" i.e.,
that the set $W$ of all possible scenarios
\emph{splits} into worlds that satisfy $p$ and worlds that satisfy $q$.
%
%\end{equation}
%
Thus team semantics works correctly here.
The interpretation
%
%\begin{equation}\label{three}
%
``$W\Vdash p\text{ or }W\Vdash q$"
%
%\end{equation}
%
has a different meaning, which is false in this case.
%
%However, the natural language statement
%
%``\emph{Either you have high blood pressure or an ear infection}."
%
%can, arguably, be interpreted to be either of the claims.
%

%
In this example, the assertion $p\vee q$ was
made about the \emph{set} $W$ of
possible states of affairs, i.e., sets of assignments. 
Since team semantics is based on
sets of assignments (rather than individual assignments), it is a
natural  framework for interpreting
determinacy atoms $\dep(p_1,\ldots,p_k;q)$. 
We next give natural language examples that 
should convince the reader that the semantics of the operator $\dep$
given by $\mathcal{D}$ is a good match with intuitions 
concerning statements about propositional determinacy.

Consider now a scenario with
two containers of water in two laboratory ovens.
Fix the following propositions.
\begin{itemize}
\item
``\emph{The temperature is over $100^{\circ}$ Celsius}'', denoted by $p$.
\item
``\emph{The water is boiling}'', denoted  by $q$.
%
%\item
%\emph{The manometer shows (that the pressure is) 1 bar}, encoded by $q$.
%
\end{itemize}
Assume the setting is encoded by the set $W = \{w_1,w_2\}$,
with one possible world for each container, such that
%$w_1 =\{(p,1),(q,1)\}$ and $w_2 = \{(p,0),(q,0)\}$.
$w_1 = \{p,q\}$ and $w_2 = \emptyset$. By $w_1 = \{p,q\}$
we of course mean that $w_1(p) = 1$ and $w_1(q) = 1$,
and analogously,  $w_2 = \emptyset$ means that $w_2(p) = 0$
and $w_2(q) = 0$.
Consider the following assertion:
\begin{center}``\emph{Whether the temperature is over $100^{\circ}$ Celsius}\\
\emph{determines whether the water is boiling}."
\end{center}
It is natural to interpret the assertion to mean that $W\Vdash\dep(p;q)$. 
Thus the semantics of $\dep$ seems to work fine here. For another example, extend the above setting with a third oven $w_3$ with a water container and a new
proposition $r$ which asserts that the air pressure in the oven is over $1$ bar.
Let $W = \{w_1,w_2,w_3\}$, where
\begin{itemize}
\item
$w_1 =   \{p, q \}$  %\{(p,1),(q,1),(r,0)\}$,
\item
$w_2 =   \emptyset$  %\{(p,0),(q,0),(r,0)\}$,
\item
$w_3 =   \{p, r \}$  %\{(p,1),(q,0),(r,1)\}$.
\end{itemize}
This time the temperature being over $100^{\circ}$ Celsius does \emph{not}
determine whether the water is boiling, i.e., $W\not\Vdash\dep(p;q)$,
because both $w_1$ and $w_3$ satisfy $p$, but the two worlds
disagree on the truth value of $q$. However, 
we do have that $W\Vdash\dep(p,r;q)$.
By adding yet another world $w_4 =  \{p, q, r \}$
%
%\{(p,1),(q,1),(r,1)\}$
% 
encoding a fourth oven, we end up with a
laboratory where $W\not\Vdash\dep(p,r;q)$.
A scenario where $\dep(p,r;q)$ seems to hold universally, i.e.,
in every correctly designed set $W$ of possible worlds,
can be obtained for example by considering a setting
where each world is associated with a balance scale
and two equally heavy weights. Let $p$ encode the
assertion that exactly one weight has been placed on the left
tray of the balance scale, and $r$ the corresponding
assertion concerning the right tray; $q$ is the assertion
that the scale is in balance. Now, indeed, $\dep(p,r;q)$ holds
for any collection of physically possible worlds, with $q$ being
true exactly when either $p$ and $r$ are both true or
when they are both false.
We have seen that team semantics works 
fine on simple disjunctive formulae $p\vee q$ and
and determinacy statements $\dep(p;q)$.
We next combine disjunctions and
determinacy statements and show that
this leads to problematic interpretations from the
point of view of natural language (cf. \cite{doubleteam1}).
Let $p$ denote the assertion that the sun is shining
and $q$ the assertion that it is winter.
Consider a setting where $W = \{w_1,w_2, w_3,w_4\}$, 
with all possible distributions of truth values for $p$ and $q$
realized. Now clearly $W\not\Vdash\dep(p;q)$,
so whether the sun is shining does not determine whether it is winter.
However, now $W\Vdash\dep(p;q)\vee\dep(p;q)$ holds in $\mathcal{D}$.
This seems strange. Consider
the following translation of the formula $\dep(p;q)\vee\dep(p;q)$
into natural language.
\begin{multline*}
\emph{`` Whether the sun is shining determines
whether it  is winter}, \emph{or},\\
\emph{whether the sun is shining determines whether it  is winter}."
\end{multline*}
The intuitively correct
interpretation of the above
statement seems to be the (indeed false) assertion that
\begin{center}
``$W\Vdash \dep(p;q)$ or $W\Vdash \dep(p;q)$,"
\end{center}
rather than
the (true) assertion ``$W\Vdash \dep(p;q)\vee\dep(p;q)$"
suggested by team semantics.
The natural language statement \emph{``Whether the sun is shining determines
whether it  is winter}, \emph{or},
\emph{whether the sun is shining determines whether it  is winter}"
seems obviously false. Therefore team semantics here gives an undesired interpretation to
the formula $\dep(p;q)\vee\dep(p;q)$.
In fact, we observe that
the formula $\varphi := \dep(p;q)\vee\dep(p;q)$ is a \emph{validity} of $\D$, i.e.,
we have $W\models\varphi$ for \emph{every} SD-model $W$.
%

%
%We note here that the formula $\dep(p;q)\vee\dep(p;q)$ is 
%
%of course not equivalent to $\dep(p;q)$ in team semantics, and
%thus $\vee$ is not idempotent in $\mathcal{D}$. Pragmatically motivated arguments
%against such redundant operations have been given in 
%
%\cite{Kazir, meyer}. While we refrain from taking a position in
%relation to the arguments in these articles, they are worth mentioning here.
%

%
\begin{comment}
%
This analysis shows that disjunction is problematic in
%
the context of multiple possible worlds. After all,
%
the semantics of $\D$ seemed to work fine
%
with the assertion $p\vee q$ in the example
%
with the patient with vertigo, but the semantics appears
%
somewhat problematic with the formula $\dep(p;q)\vee\dep(p;q)$
%
in the above example.
%

%
It is not certain whether
%
the natural language disjunction is
%
unambiguos in the context of sets of possible worlds
%
(even if we exclude the possibility of 
%
exclusive readings of disjunctions from the picture).
%
In the above examples, we have read the.
%
\end{comment}
%

%
\begin{comment}
%
It is not crucial for our argument that both disjuncts of 
%
the formula $\dep(p;q)\vee\dep(p;q)$ in
%
the our example were the same.
%
Indeed, it is easy to 
%
construct examples where the
%
natural language translation of the formula $\dep(p;q)\vee\dep(r;t)$
%
gives a similar mismatch between $\mathcal{D}$ and natural language.
%
\end{comment}
%

%
For another example, consider 
the formula $\dep(p;q)
\vee\dep(r;q)$,
where $p$, $q$ and $r$ stand for ``It is dark", 
``John is at home", and "It is cold" respectively.
Assume all distributions of truth values of the propositions $p,q,r$ are 
%,
realized in $W$. Now $W\Vdash\dep(p;q)\vee\dep(r;q)$ holds in $\mathcal{D}$.
This is again counterintuitive from the natural language point  of view.
Like the formula $\dep(p;q)\vee\dep(p;q)$, also the
formula $\dep(p;q)\vee\dep(r;q)$ is a validity of $\mathcal{D}$.
In fact, every formula of the
type $\dep(p_1,\dots,p_k;q)\vee\dep(r_1,\dots,r_n;q)$ is a
validity of $\mathcal{D}$, because every SD-model $W$ can
be split into sets $U,V\subseteq W$ such that $U\cup V = W$ and
each assignment in $U$ satisfies $q$ while each assignment in $V$ satisfies $\neg q$.
Before we discuss how $\mathcal{L}_{\dep}$ deals with
the above formulae, we note once more that 
our analysis assumes that $\vee$ should
correspond to the natural language ``\emph{or}."
We do not want to claim that the natural language word \emph{or}
has always a unique interpretation.\footnote{There exist arguments
essentially promoting the uniqueness of the meaning of
disjunction (see \cite{aloni} for an overview). On the
other hand, already for example the inclusive and exclusive
modes of \emph{or} are sometimes
taken to demonstrate ambiguity of disjunction. While we wish to
refrain from taking any definite position in this debate,
this issue is worth mentioning here. (We also want to point out that
in the formal proofs and definitions of this article, the word \emph{or} is
used in the standard inclusive fashion as is customary in
standard mathematical practice.)}
However, $\mathcal{L}_{\dep}$
works quite nicely in the above examples, as we will next demonstrate.
Concerning the formula $p\vee q$ in the beginning of the
section, we have $W\Vdash p\vee q$ iff $W\models p\vee q$.
Also, for determinacy statements $\varphi := \dep(p_1,...,p_k;q)$,
we have $W\Vdash \varphi$ iff $W\models \varphi$.
Finally, for the problematic formulae
$\dep(p;q)\vee\dep(p;q)$ and $\dep(p;q)\vee\dep(r;q)$,
the semantics of $\mathcal{L}_{\dep}$ gives
the desired interpretations: we have $W\models\dep(p;q)\vee\dep(p;q)$ iff
($W\models\dep(p;q)$ \emph{or} $W\models\dep(p;q)$)
and similarly for the formula $\dep(p;q)\vee\dep(r;q)$.
We will not try to give an elaborated
account of how well exactly $\mathcal{L}_{\dep}$
corresponds to natural language, but it is essential to
notice that $\mathcal{L}_{\dep}$ can be considered to  
have a similar level of naturalness as standard S5 modal logic or
modal logic with the universal modality.
The reason for this is that a similar Kripke style semantics is
used, and furthermore, it can be argued
that $\mathcal{L}_{\dep}$ is  simply a
\emph{fragment} of the modal logic S5.
This is because determinacy statements are
naturally definable in terms of statements about possibility:
simply consider the direct natural language 
translation of the equivalence
\begin{align*}
\dep(p;q)&\\  \leftrightarrow\ \ \ \
&\bigl(\neg\bigl(\udiam (p\wedge q)\wedge \udiam (p\wedge \neg q)\bigr)
\wedge\neg\bigl(\udiam (\neg p\wedge q)
\wedge \udiam (\neg p\wedge \neg q)\bigr)\bigr),
\end{align*}
where $\langle u\rangle$ should be read as ``it is possible that."
The natural language translation of this equivalence indeed seems
intuitively immediately appealing, and importantly, the equivalence essentially
just states the formal semantics of $\dep(p;q)$ (determinacy) in terms of
the diamonds $\udiam$ (possibility), thus demonstrating that statements of 
determinacy can be very naturally and directly formulated in terms of possibility statements.
Therefore $\mathcal{L}_{\dep}$ can be considered a
fragment of modal logic (with the universal or S5 modality in 
the particular case of this paper),
and the level of correspondence between natural language and $\mathcal{L}_{\dep}$ is
similar to the corresponding relationship for (S5) modal logic. We note that
the restriction to S5 frames is of no importance here.
We also note that the above equivalence deals
only with the simple determinacy formula $\dep(p;q)$ but it is easy to generalize
our argument to more complex formulae. (See also Proposition \ref{prop:Dep}.)
%
%Note also that the restriction to S5 frames is of no particular importance here, so the
%
%above equivalence demonstrates that $\mathcal{L}_{\dep}$ is a
%
%fragment of modal logic and has the same level of naturality. 
%

%
In addition to disjunction $\vee$,
also the semantics of negation $\neg$ in $\D$ can be counterintuitive if
the reading ``it is not the case that" is desired for $\neg$.
For example, let $p$ and $q$ denote the assertions 
``the Riemann hypothesis holds" and ``it is raining," respectively.
Let $W$ be a nonempty SD-model.
Now we have $W\not\Vdash\neg \dep(p;q)$ in $\mathcal{D}$.
In $\mathcal{L}_{\dep}$, we 
have $W\models\neg \dep(p;q)$ in (for example) every model $W$ where
the truth value of $p$ is the  same in every possible world and where $q$ is
true in some worlds and false in others.\footnote{
Perhaps alternative readings of $\neg$ could work
better in $\mathcal{D}$. For example, in \cite{doubleteam2},
the reading ``\emph{it is never the case that}" for the
negation of $\mathcal{D}$ is suggested.
This reading is not meant to necessarily have any temporal connotations,
but instead could alternatively be read as ``\emph{it is impossible that}."
We shall not try analyze here how well such a reading could actually work.}
Finally, concerning $\mathcal{L}_{\dep}$, it is worth
noting the triviality that just as in S5, $W\models\neg\varphi$ is not in general
equivalent to $W\not\models\varphi$, because $W\models\neg\varphi$
means that $W,w\not\models\varphi$ for all $w\in W$. 
Analogously, $W\models\varphi\vee\psi$ is 
not in general equivalent to $(W\models\varphi\text{ or }W\models\psi)$.
We note that, by Proposition \ref{classicalproposition},
we have $W\models\chi$ iff $W\Vdash\chi$ for all
formulae of propositional logic in
negation normal form, so $\mathcal{D}$ and $\mathcal{L}_{\dep}$
are similar with respect to formulae of propositional logic PL.
%

%
%\begin{comment}
%
To give an example of how the semantics of $\neg$
works in $\mathcal{D}$ and $\mathcal{L}_{\dep}$
in the context of formulae of PL,
let $p$ and $q$ denote the assertions ``John has a cat"
and ``John is married," respectively. Assume a
scenario where it is agreed that the
possible worlds are $w_1 = \{p\}$ and $w_2 = \emptyset$, i.e.,
John may or may not have a cat, but he is definitely not married.
Let $W = \{w_1,w_2\}$.
The assertion ``It is not the case that John is married"
seems correct, and indeed, we have $W\Vdash\neg q$.
Note, however, that even though $W\not\Vdash p$, 
the claim ``It is not the case that John has a cat"
would seem odd. To make the claim $W\not\Vdash p$,
one would have to assert, e.g., that  it is possible
that John does not have a cat, or that it is not
necessarily the case that John has a cat.
\section{Comparing the expressive powers of
$\D$, $\I$, $\dlang$ and $\mathcal{L}_{\ind}$}\label{express}
We have earlier observed (Proposition \ref{expr}) that $\mathcal{L}_{\dep}$ embeds into
$\mathcal{L}_{\ind}$, and also that $\D$ embeds into
$\I$. In this section we will complete our discussion concerning the 
expressive powers of the logics $\mathcal{L}_{\dep},\mathcal{L}_{\ind},\D$ and
$\I$.

\subsection{$\I$ does not embed into $\D$}\label{non-embed}

%
%
%
%\vmnote{Reference? It is folklore because it is trivial.}
%
%
Let $\varphi,\psi\in\mathit{FOR}(\D)\cup\mathit{FOR}(\I)$.
We write $\varphi\equiv_{team}\psi$ if the equivalence
$W\Vdash\varphi\Leftrightarrow W\Vdash\psi$ holds
for all SD-models $W$. The first result we wish to point out is
well-known; we prove it for the sake of completeness. 
\begin{proposition}
The logic $\I$ does not embed into
$\D$: there exists a formula $\varphi$ of $\I$
such that for all formulae $\psi$ of $\D$, the
equivalence $\varphi\equiv_{team}\psi$ fails.
\end{proposition}
\begin{proof}
It is well-known, and easy to show, that $\D$ satisfies
the following \emph{downwards closure property}: if $W\Vdash_{\D}\varphi$ and $U\subseteq W$,
then we have $U\Vdash_{\D}\varphi$.

Now, define an SD-model $W = \{w_1,w_2,w_3,w_4\}$ where
the four states represent all truth assignments for propositions $p$ and $q$. 
More precisely, let $w_1(p) = w_1(q) = w_2(p) = w_3(q) = 1$ and
$w_2(q) = w_3(p) = w_4(p) = w_4(q) = 0$.
Define $U := \{w_1,w_4\}$.
Consider the formula $\varphi := p\, \ind q$.
It is clear that $W\Vdash\varphi$ and $U\not\Vdash\varphi$.
Therefore, due to the downwards closure
property of $\D$, no
formula $\psi$ of $\D$ can satisfy the
equivalence $\varphi\equiv_{team}\psi$.
\end{proof}

\subsection{
Embedding $\mathcal{L}_{\ind}$ into $\dlang$ via a concrete translation}
 \label{embed}

While $\D$ embeds into $\I$ but not vice versa,
the situation is different for $\mathcal{L}_{\dep}$ and $\mathcal{L}_{\ind}$.
In Section \ref{indeplogic} we established that $\mathcal{L}_{\dep}$
embeds into $\mathcal{L}_{\ind}$. We now define a translation 
showing that $\mathcal{L}_{\ind}$ embeds into $\mathcal{L}_{\dep}$, too.
While it is straightforward to observe, based on
Propositions \ref{expressivitycharacterization} and \ref{expressivecompleteness}, 
that $\mathcal{L}_{\ind}$ indeed embeds into $\mathcal{L}_{\dep}$,
the concrete translation below will be of interest later on 
when we discuss \emph{compositional translations} (to be defined) between
the logics under investigation.
We define the following translation $s:\mathit{FOR}(\mathcal{L}_{\ind})\rightarrow
\mathit{FOR}(\mathcal{L}_{\dep})$:
\begin{enumerate}
\item
$s(p) := p$, for $p\in\mathit{PROP}$
\item
$s(\neg\varphi) := \neg s(\varphi)$
\item
$s(\varphi\wedge\psi) := s(\varphi)\wedge s(\psi)$
\item
We then translate the formula
$(\fo_1,\ldots,\fo_k)\, \ind_{(\fod_1,\ldots,\fod_m)}(\fob_1,\ldots,\fob_n)$
in a way that derives from Proposition \ref{prop:Ind}.  
First we define
\begin{multline*}
\mathcal{B} \, :=
\conj(\{ s(\fo_1),\ldots,
s(\fo_k)\}) \times \ \conj(\{s(\fod_1),\ldots,s(\fod_m)\})\\
\times  \conj(\{s(\fob_1),\ldots,s(\fob_n)\}).
\end{multline*}
For the special case 
where $m = 0$, recall that $\mathit{Conj}(\emptyset) = \{\top\}$.
Now, let 

$s((\fo_1,\ldots,\fo_k)\, \ind_{(\fod_1,\ldots,\fod_m)}(\fob_1,\ldots,\fob_n)): = $
$$\bigwedge\limits_{
(\varphi,\theta,\psi)\, \in\, \mathcal{B}}
\Big(\bigl(\udiam(\fod \wedge \fo)
\wedge\udiam(\fod \wedge \fob)\bigr)
\rightarrow\ \udiam(\fod \wedge \fo \wedge \fob) \Big).$$
\end{enumerate}
It is easy to see that the translation of the operator $\ind$ given by
the formula above describes the meaning of $\ind$ quite directly in
terms of $\udiam$.
The proof of the next claim is straightforward,
using the translation $s$ defined above.
\begin{proposition}
$\mathcal{L}_{\ind}$ embeds into $\mathcal{L}_{\dep}$ i.e.,
for each formula $\varphi$ of $\mathcal{L}_{\ind}$,
there exists an equivalent formula $\psi$ of $\mathcal{L}_{\dep}$
in the sense that for all $W$ and all $w\in W$, we
have $W,w\models\varphi$ iff $W,w\models\psi$.
\end{proposition} 

\subsection{
Strict embedding of $\D$ and $\I$ into $\dlang$ and $\mathcal{L}_{\ind}$}
\label{strictembed}

Next we will show that the team-semantics-based logic $\I$
is strictly contained in the Kripke-style
logic $\mathcal{L}_{\dep}$ in the following sense.
\begin{enumerate}
\item
For each formula $\varphi$ of $\I$, there exists a
formula $\varphi'$ of $\dlang$ such
that $\varphi$ and $\varphi'$ define the same class of SD-models, i.e., 
$W\Vdash_{\I}\varphi\Iff W\models_{\dlang}\varphi'$
for all SD-models $W$.
\item
There exists a formula $\psi$ of $\mathcal{L_{\dep}}$
which is not equivalent to any  formula of $\I$,  i.e.,
for all formulae $\psi'$ of\, $\I$, 
there exists a model $W$ such that
($W\Vdash\psi'\text{ and } W\not\models\psi$) or
($W\not\Vdash\psi'\text{ and }W\models\psi$).
\end{enumerate}
The claim 1 above is essentially obvious, since
%---as Claim \ref{propositionalequivalence} below asserts---
all classes of models definable in $\I$ are closed under finite propositional equivalence,
and due to Proposition \ref{expressivitycharacterization}, $\dlang$
can define all such model classes. However, we will provide an
explicit and effective translation of $\I$ into $\dlang$ 
which is interesting in its own right and also elucidates the semantics of $\I$.
Furthermore, despite the
simplicity of our translation, we will show in Section
\ref{compositionality} that there does not exist a 
\emph{compositional translation} from $\I$  into $\mathcal{L}_{\dep}$.
Recall once again the notion of a type normal form
formula and related notions from Section \ref{preliminaries}.
%
%\anote{See that section for comments about the  naming of TNFFs.}
%
Let $\Phi$ be a finite nonempty set of \emph{proposition symbols} and 
$\chi = \bigvee\{\chi_1,\ldots,\chi_k\}$ a formula in $\mathit{DNF}(\Phi)$.
We let $\mathit{SPLIT}(\chi)$ denote the set of pairs $(\alpha,\beta)$
of type normal form formulae in $\mathit{DNF}(\Phi)$
such that if $\alpha =  \bigvee\{\alpha_1,\ldots,\alpha_m\}$
and $\beta = \bigvee\{\beta_1,\ldots,\beta_n\}$, then we have
$$\{\chi_1,\ldots,\chi_k\} = \{\alpha_1,\ldots,\alpha_m\}\cup\{\beta_1,\ldots,\beta_n\}.$$
{
\begin{comment}
%
$W$ be a model and $w\in W$.
%
Notice that $W,w\models\varphi\vee \varphi\, \ind\, \varphi$ iff
%
some point $u$ in $W$ satisfies $\varphi$.
%
Below we let $\udiam\varphi$ denote the formula
%
$\varphi\wedge \varphi\, \ind\, \varphi$. Strictly speaking this
%
clashes with our previous definition where $\udiam$
%
\end{comment} 
}
Let $\chi\in\mathit{DNF}(\Phi)$. We define the following
translation $t_{\chi}$ from $\I$ into $\dlang$:
\begin{enumerate}
\item
$t_{\chi}(p) := \ubox(\chi\rightarrow p)$
and $t_{\chi}(\neg p) := \ubox(\chi\rightarrow\neg p)$
\item
To translate the formula $(p_1,\ldots,p_k)\, \ind_{(q_1,\ldots,q_m)}\, (r_1,\ldots,r_n)$,
we use a suitably modified version of the equivalence in Proposition \ref{prop:Ind}.
We first define
$$B\, := \conj(\{p_1,\ldots,p_k\}) \times \ \conj(\{q_1,
\ldots,q_m\}) \times  \conj(\{r_1,\ldots,r_n\}).$$
%
%In the case $m = 0$, we let
%
%$$B\, := \conj(\{p_1,\ldots,p_k\}) \times \ \{\top\} \times  \conj(\{r_1,\ldots,r_n\}).$$
%
%\anote{Using sets $\nff(\ldots)$ here  is a bit of
%an overkill. I will probably simplify. The arguments below won't change.}
%
%
%
We then define
$t_{\chi} \bigl((p_1,\ldots,p_k)\, \ind_{(q_1,\ldots,q_m)}\, (r_1,\ldots,r_n)\bigr) :=$\\
$$\bigwedge\limits_{
(\varphi,\psi,\theta)\, \in\, B}  
\Big(\bigl(\udiam(\chi\wedge\psi\wedge\varphi)
\wedge\udiam(\chi\wedge\psi\wedge\theta)\bigr)
\rightarrow\ \udiam(\chi\wedge\psi\wedge\varphi\wedge\theta) \Big).
$$
%The formula above translates quite directly the semantic meaning of $\ind$. 
%
%
%
%\item
%$t(\neg\dep(q_1,\ldots,q_k,p)) := \bot$
%
\item
$t_{\chi}(\varphi\wedge\psi) := t_{\chi}(\varphi)\wedge t_{\chi}(\psi)$
\item
$t_{\chi}(\varphi\vee\psi) := \bigvee\limits_{(\alpha,\beta)\, \in\,
\mathit{SPLIT}(\chi)}\bigl(t_{\alpha}(\varphi)\wedge t_{\beta}(\psi)\bigr)$
\end{enumerate}
%

%
%Note that the translation of the operator $\ind$ bears a resemblance to
%
%the translation of $\ind$ in the
%
%translation $s:\mathcal{L}_{\ind}\rightarrow\mathcal{L}_{\dep}$.
%

%
If $\varphi$ is a formula of $\I$ and $\Phi$ the set of
proposition symbols in $\varphi$, we let $\chi(\varphi)$ denote the
formula in $\mathit{DNF}(\Phi)$
that contains as disjuncts \emph{all}
types over $\Phi$. Note that the formula $\chi(\varphi)$ is a
tautology.
We then prove that our translation of $\I$ into $\dlang$
preserves truth. 
%
%
%
%$\Phi$ is the set of
%
%proposition symbols in $\varphi\in\I$ and $\chi$ a TNFF over $\Phi$,
%
%then we have 
%
%$W\Vdash\varphi$ iff
%
%$W\models t_{\chi(\varphi)}(\varphi)$.
%

\begin{lemma}\label{dependencetranslationtheorem}
$W\Vdash_{\I}\varphi$ iff $W\models_{\dlang} t_{\chi(\varphi)}(\varphi)$.
\end{lemma}
\begin{proof}
Let $\varphi$ be a formula of $\I$ and $W$ a model.
Let $\Phi$ be the set of proposition symbols that occur in the
formula $\varphi$. If $\chi\in\mathit{DNF}(\Phi)$, we let 
$W_{\chi}$ 
%$\extn{\chi}_{W}$ 
denote  the set of worlds in $W$ that satisfy $\chi$. 
We will show that for every $\chi\in\mathit{DNF}(\Phi)$
and every subformula $\psi$ of $\varphi$, we have 
$$W_{\chi}\Vdash_{\I}\psi\Iff W\models_{\dlang} t_{\chi}(\psi).$$
The claim of the Lemma will then follow, as 
%$\extn{\chi(\varphi)}_{W}
$W_{\chi(\varphi)} = W$.

The proof proceeds by induction on the structure of $\psi$.
The cases for proposition symbols, negated proposition symbols and conjunctions are straightforward. The argument for $\ind $ is easy, as our
translation in that case captures quite directly the semantics of $\ind$.
We now proceed to the case $\psi = \psi'\vee\psi''$.
When going through the argument below, it helps to keep in mind the
trivial technicality that for an SD-model $U$ and a proper subset $S$ of $\mathit{PROP}$,
there may exist several assignments in $U$ that are
equivalent with respect to $S$, i.e., assignments
that satisfy exactly the same propositions in $S$ (but differ elsewhere).
Assume that $W_{\chi}\Vdash_{\I}\psi'\vee\psi''$.
Thus there exist sets $S',S''\subseteq W_{\chi}$ such that $S'\Vdash_{\I}\psi'$ and $S''\Vdash_{\I}\psi''$, and furthermore, $S'\cup S'' = W_{\chi}$. 
Therefore there exists a pair $(\alpha,\beta)\in\mathit{SPLIT}(\chi)$\,
such that $S' \subseteq W_{\alpha}$ and  $S''\subseteq W_{\beta}$,
and furthermore, $W_{\alpha}\equiv_{\Phi} S'$ and $W_{\beta}\equiv_{\Phi} S''$
(recall Definition \ref{equivalentmodelsdefinition}).
Therefore clearly $W_{\alpha}\Vdash\psi'$ and $W_{\beta}\Vdash\psi''$.
%
%
%
\begin{comment}
%
The claim below is wrong:
%
Therefore we observe, with the help of
Lemma \ref{blocklemma}, that there
exists a pair $(\alpha,\beta)\in\mathit{SPLIT}(\chi)$\, 
such that $W_{\alpha}\Vdash\psi'$ and $W_{\beta}\Vdash\psi''$.
%
\end{comment}
%
%
%
Hence, by the induction hypothesis, we have $W\models_{\dlang} t_{\alpha}(\psi')$
and $W\models_{\dlang} t_{\beta}(\psi'')$.
Thus $W\models_{\dlang} t_{\alpha}(\psi') \land t_{\beta}(\psi'')$,
whence we conclude that $W\models_{\dlang} t_{\chi}(\psi'\vee\psi'')$.
For the converse, assume that $W\models_{\dlang} t_{\chi}(\psi'\vee\psi'')$.
Therefore there exist type normal form formulae $\alpha,\beta$
such that $(\alpha,\beta)\in\mathit{SPLIT}(\chi)$, and furthermore,
$W\models_{\dlang} t_{\alpha}(\psi')$ and $W\models_{\dlang} t_{\beta}(\psi'')$.
By the induction hypothesis, we have $W_{\alpha}\Vdash_{\I}\psi'$
and $W_{\beta}\Vdash_{\I}\psi''$. Since $(\alpha,\beta)\in\mathit{SPLIT}(\chi)$,
we have $W_{\alpha}\cup W_{\beta} = W_{\chi}$,
and therefore $W_{\chi}\Vdash_{\I}\psi'\vee\psi''$.
\end{proof}
We are now ready to prove the following theorem.
%
%that Propositional Dependence Logic $\I$ is strictly contained in $\dlang$.
%

%
\begin{theorem}
The logic $\I$ is strictly contained in $\dlang$,
i.e.:
\begin{enumerate}
\item
For each $\varphi\in\mathit{FOR}(\I)$, there exists a
formula $\varphi'\in\mathit{FOR}(\mathcal{L}_{\dep})$ (which can 
be found effectively), 
such that
for all SD-models $W$, it holds that $W\Vdash\varphi$ iff $W\models\varphi'$.
\item
There exists a formula $\psi\in\mathit{FOR}(\mathcal{L}_{\dep})$
that is not expressible in $\I$, i.e., for all $\chi\in\I$,
there exists an SD-model $W$ such that the equivalence
$W\models\psi\Leftrightarrow W\Vdash\chi$ fails.
\end{enumerate}
\end{theorem}
\begin{proof}
%
%The proof is quite straightforward, but still worth presenting in detail.
%
By Lemma \ref{dependencetranslationtheorem} there exists an
effective translation from $\I$ into $\mathcal{L}_{\dep}$.
Hence we only need to prove the second claim of the theorem.
In fact the claim follows relatively easily from Proposition \ref{expressivitycharacterization}
and the proof of Theorem 4.2 of \cite{kontivan}, but we will establish the claim
here explicitly.
We will show that the $\dlang$-formula $\neg\cdep p$ is not
expressible in $\I$, i.e., there is no formula $\psi$ of $\I$
such that $W\Vdash_{\I}\psi$ iff $W\models_{\dlang}\neg\cdep p$.
We first define two models $U$ and $U'$, where
$U$ consists of two worlds, one satisfying $p$
and the other one not, and $U'$ consists of
single world that does not  satisfy $p$.
Furthermore, for all other proposition
symbols $q$, we define $q$ to be false in each world of
the models $U$, $U'$.
We then show by induction on the structure of formulae 
that for all $\varphi\in\mathit{FOR}(\I)$, we have
$$U\Vdash_{\I}\varphi\ \Rightarrow\ U'\Vdash_{\I}\varphi.$$
For the literals $p$ and $\neg p$ this is immediate, as $U\not\Vdash_{\I} p$ and $U\not\Vdash_{\I}\neg p$.
For other literals $q$, $\neg q$, etc., the implication holds
because $U\not\Vdash_{\I} q$
and $U'\Vdash_{\I}\neg q$.
In order to deal with the
operator $\ind$, notice that $U'$ satisfies all formulae of the type $(p_1,\ldots,p_k)\, \ind_{(q_1,\ldots,q_m)}(r_1,\ldots,r_n)$, since the model $U'$ contains only a single world.
The case for  $\wedge$ follows
immediately by the induction hypothesis.

We then consider the case for $\vee$. 
Assume that $U\Vdash_{\I}\psi\vee\psi'$.
Therefore there exist sets $S,S'\subseteq U$ such that $S\Vdash_{\I}\psi$ and $S'\Vdash_{\I}\psi'$,
and furthermore, $S\cup S'= U$.
We may assume, by symmetry, that $S$ contains
the assignment in $U$ that does not satisfy $p$.
We consider two cases.
1. Assume that $S = U$. Then $U'\Vdash_{\I}\psi$ follows directly by the induction hypothesis.
Furthermore, we have
$\emptyset\Vdash_{\I}\chi$ for every formula $\chi\in\mathit{FOR}(\I)$,
whence the condition
\begin{equation*}
U'\Vdash_{\I}\psi\text{ and }\emptyset\Vdash_{\I}\psi'
\end{equation*}
holds. Therefore $U'\Vdash_{\I}\psi\vee\psi'$.
2. Assume that $S$ is the singleton not satisfying $p$.
Notice now that the world in $S$ and the world in $U'$ satisfy
exactly the same proposition symbols. Thus $S = U'$, whence 
$U'\Vdash_{\I}\psi$. 
Therefore the condition
\begin{equation*}
U'\Vdash_{\I}\psi\text{ and }\emptyset\Vdash_{\I}\psi'
\end{equation*}
holds again,
and we hence conclude that $U'\Vdash_{\I}\psi\vee\psi'$.
Finally, note that $U\models_{\dlang}
\neg\cdep p$, while $U'\not\models_{\dlang}\neg\cdep p$.
Therefore $\neg\cdep p$ cannot be expressible by a formula of $\I$. 
\end{proof}
\begin{corollary}
The logics $\D$
and $\I$ are both strictly contained in both
$\dlang$ and $\mathcal{L}_{\ind}$.
\end{corollary}
In summary, we have shown that 
$\mathcal{D}<\mathcal{I}<\mathcal{L}_{\dep}\equiv\mathcal{L}_{\ind}$,
where $<$ denotes strict containment and $\equiv$ equi-expressivity.
We have also observed that $\mathcal{L}_{\dep}$ and $\mathcal{L}_{\ind}$
are expressively complete in the sense that they can define
exactly all classes of SD-models closed under finite
propositional equivalence. (See Definition \ref{finpropequivalence}
for the exact specification of finite propositional equivalence.)
\subsection{Regular logics and compositionality of translations}\label{compositionality}
Let $t^*$ denote the translation from $\I$ into
$\mathcal{L}_{\dep}$ we defined above.
Recall that in addition to $t^*$, we have also defined the 
translations $t:\mathit{FOR}(\mathcal{L}_{\dep})
\rightarrow\mathit{FOR}(\mathcal{L}_{\ind})$ and
$s:\mathit{FOR}(\mathcal{L}_{\ind})
\rightarrow\mathit{FOR}(\mathcal{L}_{\dep})$.
Furthermore, in the  proof of Proposition \ref{expr},
we described a translation from $\D$ into $\I$;
let us denote that translation by $t'$.
In this section we will take a closer look at the four translations $t^*,t,s$ and $t'$.
We will establish that, in a certain sense, the translation $t^*$ is essentially
different from the other three.
We begin by defining a notion of a \emph{syntactically
regular logic} suitable for our
purposes. To this end, we first need some auxiliary definitions.
Let $\mathbb{N}^*$ be the set of all finite
sequences of numbers in $\mathbb{N}$ 
(including the empty sequence).
Let $C$ be a finite or countably infinite
set of \emph{operator symbols}. 
Let $d$ be a function that associates with each 
symbol in $C$ a nonempty subset of $\mathbb{N}^*$;
the set $d(c)$ is called the \emph{arity type set} of $c$.
For example, in $\mathcal{L}_{\ind}$, the operator $\ind$
always operates on three tuples of formulae, with the 
middle tuple being the only one allowed to be the empty tuple,
and thus the arity type set
associated with $\ind$ is $\mathbb{N}_{+}\times\, \mathbb{N}\, \times\, \mathbb{N}_{+}$. 
The set $C$ together with the function $d$ give rise to a
set $\mathit{FOR}(C,d)$ of formulae, which is defined to be
the smallest set $S$ such that the following conditions hold:
%
\begin{comment}
%
\footnote{
We note that the logics $\mathcal{D}$, $\mathcal{I}$, $\mathcal{L}_{\ind}$,
$\mathcal{L}_{\dep}$ use infix (rather than prefix) notation, but it would be a trivial syntactic exercise to re-define the operators $\dep$ and $\ind$ with prefix notation.}
%
\end{comment}
%
%
%
\begin{enumerate}
\item
If $p\in\mathit{PROP}$, then $p\in S$.
\item
%VG correction: 
If $c\in C$, $(n_1,\ldots,n_k) \in d(c)$, and $\varphi_{1,1},\ldots,\varphi_{1,n_1},\ldots, \varphi_{k,1},\ldots,\varphi_{k,n_k} \in S$, 
then
$$c\bigl( (\varphi_{1,1},\ldots,\varphi_{1,n_1}),\ldots,
(\varphi_{k,1},\ldots,\varphi_{k,n_k})\bigr)\ \in\ S.$$
\end{enumerate}
%

%
\begin{comment}
%
We note that strictly speaking
%
the logics $\mathcal{L}_{\ind}$ and
%
$\mathcal{L}_{\dep}$ do not quite
%
follow this syntactic convention
%
because---to give one reason---the connective $\rightarrow$
%
uses infix (rather than prefix notation).
%
%
%
However, it would be a trivial
%
syntactic exercise to re-define these logics in a way
%
that leads to syntactically regular formula sets.
%
\end{comment}
%

%
We call $\mathit{FOR}(C,d)$ the \emph{syntactically regular set of formulae}
defined by $C$ and $d$. 
We call a logic \emph{syntactically regular} if the set of formulae of
the logic is a syntactically regular set of
formulae for some set $C$ and a related function $d$.
Any logic whose set of formulae is a 
\emph{subformula closed subset} of some
syntactically regular set of formulae, is
called a \emph{syntactically subregular} logic.
Closure of a formula set $F$ under subformulae 
obviously means that if
$c\bigl( (\varphi_{1,1},\ldots,\varphi_{1,n_1}),\ldots,
(\varphi_{k,1},\ldots,\varphi_{k,n_k})\bigr)\ \in\ F,$
then each of the formulae $\varphi_{i,j}$ is in $F$.
For the sake of simplicity, we will below mainly talk about
regular and subregular (rather
than syntactically regular and syntactically subregular) logics.
It is easy to see that $\mathcal{L}_{\dep}$ and $\mathcal{L}_{\ind}$ 
are essentially  regular logics.
Similarly, $\D$
and $\I$ are essentially subregular logics.\footnote{We
acknowledge that strictly speaking $\mathcal{L}_{\dep}$ and $\mathcal{L}_{\ind}$
are not syntactically regular because---to give one reason---the
connective $\rightarrow$ uses infix rather than prefix notation. However, it
would be a trivial exercise to redefine the syntax of these logics in
the required way, and therefore we consider them syntactically regular.
Similarly, $\mathcal{D}$ and $\mathcal{I}$ are considered
syntactically subregular.}
Let $L$ be a subregular logic and $c$ an operator symbol of $L$.
Let $d$ be the function that associates the operators of $L$
with the related arity type set, and let $x\in d(c)$.
The pair $(c,x)$ is called a \emph{base operator} of $L$.
For example, the operator $\dep$ of $\mathcal{L}_{\dep}$
can act on tuples of formulae of all positive finite lengths,
so each pair $(\dep,i)$, where $i$ is a (singleton
tuple containing a) positive integer, is a base operator.
In contrast, the operators $\neg$ and $\rightarrow$ of $\mathcal{L}_{\dep}$
are both associated with only a single arity type,
and thus we can directly \emph{regard} $\neg$ and $\rightarrow$ \emph{as}
base operators.
We now define the notion of a \emph{compositional translation}
from one subregular logic to another. 
Intuitively, a compositional translation
from a logic $L$ to a logic $L'$ has the property that
\emph{each base operator} of $L$ is \emph{described} in $L'$ in a
\emph{uniform way}. Therefore the translation in some sense
\emph{acts only on the base operators} of $L$ rather than
directly on individual  formulae.
Thus a compositional translation can be 
considered to be, in a sense, simple and direct.
For further discussion on compositional
translations, see \cite{janssen}.
%

%
%\begin{comment}
%
Assume $\varphi_1,\dots,\varphi_k$, where $k\in\mathbb{N}$, are
distinct formulae of a subregular logic $L$.
Assume $\psi(\varphi_1,\dots,\varphi_k)$ is a formula of $L$ obtained
from $\varphi_1,\dots,\varphi_k$ by 
composing these formulae with some
collection of base operators.
Let $X_1,\dots,X_k$ be novel symbols.
Then $\psi(X_1,\dots,X_k)$ is called an
\emph{operator term} of $L$; the operator term is
obtained by replacing the \emph{original} ground instances of the
formulae $\varphi_1,\dots,\varphi_k$ in $\psi(\varphi_1,\dots,\varphi_k)$
by $X_1,\dots,X_k$,
respectively. (For example, if $\varphi_1 := p$ and $\varphi_2 := q$,
and if $\psi(\varphi_1,\varphi_2) := (p\rightarrow q)\vee p\ind q$, then we
obtain the operator term $(X_1\rightarrow X_2)\vee X_1\ind X_2$ by
replacing $\varphi_1$ with $X_1$ and $\varphi_2$ with $X_2$.)
%
%\end{comment}
%

%
Let $L$ and $L'$ be subregular logics.
We identify $L$ and $L'$ with their respective sets of formulae.  
A translation $T:L\rightarrow L'$ is \emph{compositional}, if
for each base operator $(c,(n_1,\ldots,n_k))$ of $L$,
there exists an operator term
$$\psi(X_{1,1},\ldots,X_{1,n_1},\ldots,
X_{k,1},\ldots,
X_{k,n_k})$$
of $L'$ such that for all tuples
$(\varphi_{1,1},\ldots,\varphi_{1,n_1}),\ldots,(\varphi_{k,1},\ldots,
\varphi_{k,n_k})$
of formulae of $L$ such
that $c\bigl((\varphi_{1,1},\ldots,\varphi_{1,n_1}),\ldots,(\varphi_{k,1},\ldots,
\varphi_{k,n_k})\bigr)\in L$, we have 
\begin{multline*}
T\bigl(c\bigl((\varphi_{1,1},\ldots,\varphi_{1,n_1}),\ldots,(\varphi_{k,1},\ldots,
\varphi_{k,n_k})\bigr)\bigr)\\
:=\ \psi(\, T(\varphi_{1,1}),\ldots,T(\varphi_{1,n_1}),
\ldots,T(\varphi_{k,1}),\ldots,
T(\varphi_{k,n_k})\, \bigr),
\end{multline*}
i.e., the translated formula is obtained by substituting each
symbol $X_{i,j}$ in $\psi$ by $T(\varphi_{i,j})$.
%
\begin{comment}
%
\footnote{
%
If desired, the formula $\psi(p_{1,1},\ldots,p_{1,n_1},\ldots,
%
p_{k,1},\ldots,
%
p_{k,n_k})$ can also be given in the form 
%
$\psi(X_{1,1},\ldots,X_{1,n_1},\ldots,
%
X_{k,1},\ldots,
%
X_{k,n_k})$, where the symbols $X_{i,j}$ denote fixed 
%
subformulae of $\psi$, with some symbols $X_{i,j}$ possibly having 
%
multiple occurrences in $\psi$, and with
%
the condition that if $X_{i,j}$ is a subformula of $X_{i',j'}$,
%
then $X_{i,j} = X_{i',j}$. 
%
This representation may be desirable for instance if $L'$
%
has only a small finite number of proposition symbols.}
%
%
%
\end{comment}
%
%
%
Furthermore, it is required that for each proposition symbol $p$ in
the syntax of $L$, the translation $T(p)$ contains no other proposition
symbols except for $p$.
This ensures that the translation $T(\varphi)$ of
any formula $\varphi$ contains no other proposition symbols except 
for those in $\varphi$ itself.
This is a natural requirement and
can be essential for example when considering
SD-models with a finite propositional signature, i.e,
models that interpret only a finite number of proposition symbols.
We note that in the
team semantics literature, SD-models are in most cases indeed defined to
interpret only finitely many proposition symbols.
The symbols $X_{i,j}$ in the above definition should be regarded as 
\emph{placeholders} in the operator term $\psi(X_{1,1},\ldots, X_{k,n_k})$.
Intuitively, the operator term $\psi(X_{1,1},\ldots, X_{k,n_k})$
provides a ``uniform description" of
the base operator $(c,(n_1,\ldots,n_k))$ of $L$ in $L'$.
The  above definition of a compositional translation is
suitable for the purposes of the current paper and follows
standard principles of compositional translations.
Note that our translations 
$t:\mathit{FOR}(\mathcal{L}_{\dep})\rightarrow
\mathit{FOR}(\mathcal{L}_{\ind})$ and 
$s:\mathit{FOR}(\mathcal{L}_{\ind})\rightarrow
\mathit{FOR}(\mathcal{L}_{\dep})$
are indeed compositional, as is the
translation $t':\mathit{FOR}(\D)\rightarrow\mathit{FOR}(\I)$
from the proof of Proposition \ref{expr}. However, the
translation $t^*:\mathit{FOR}({\I})
\rightarrow\mathit{FOR}(\mathcal{L}_{\dep})$ is \emph{not}
compositional, despite being relatively simple.
In fact, it will turn out that a sound compositional translation from $\I$
into $\mathcal{L}_{\dep}$ or $\mathcal{L}_{\ind}$ is not possible.
To see this, it is sufficient (due to the existence of
the translations $t'$ and $s$) to show that $\D$
does not translate compositionally into $\mathcal{L}_{\dep}$.
The following theorem does exactly that.
%
%Indeed, the framework of team semantics and the approach
%
%that the logic $\mathcal{L}_{\dep}$ is based on are 
%
%different enough so that the following theorem holds:
%

%
\begin{theorem}
There exists no compositional translation $T$ from $\D$ into
$\mathcal{L}_{\dep}$ which is sound with respect to SD-models in
the sense that for any SD-model $W$ and any formula $\varphi$ of\, $\D$,
we have $W\Vdash\varphi$ iff $W\models T(\varphi)$.
\end{theorem}
\begin{proof}
Let $S$ be an SD-model and $\chi_1,\chi_2$ formulae of $\mathcal{L}_{\dep}$.
We say that $\chi_1$ and $\chi_2$
are \emph{locally equivalent in $S$}, if for all $w\in S$, it holds
that $S,w\models\chi_1$ iff $S,w\models\chi_2$.
Suppose, for the sake of contradiction, that a
compositional translation $T$ from $\D$ into $\mathcal{L}_{\dep}$ exists.
Consider an SD-model $V$ consisting of exactly
two assignments, 
one satisfying $p$ and the other one not.
Then fix a proposition symbol $q$ so that
for each $w\in V$, the assignment $w$ satisfies $q$ if and only if $w$ does not satisfy $p$.
For all other proposition symbols $r\in\mathit{PROP}\setminus\{p,q\}$, we assume that
neither of the worlds in $V$ satisfies $r$.
We then define an SD-model $U$ which is the same as $V$ but with
the interpretation of $q$ redefined so that $q$ is
satisfied by $u\in U$ if and only if $u$ satisfies $p$.
Let $W\in\{U,V\}$.
We will prove the following claims about $W$.
\begin{enumerate}
\item
Let $r\in\mathit{PROP}\setminus\{p,q\}$.
The formula $T(r)$ is not satisfied by either of the points of $W$.
\item
$T(q)$ is satisfied by exactly one point of $W$.
\item
$T(\cdep p)$ is satisfied by exactly one point of $W$.
\end{enumerate}
We begin with the first claim. If $T(r)$ was
satisfied by \emph{both} points of $W$, 
we would have $W\models T(r)$ and thus $W\Vdash r$,
which is a contradiction.
If $T(r)$ was satisfied by exactly one point of $W$,
we would obtain a contradiction due to the fact that $T(r)$ is required 
by definition to contain no other proposition symbols except for $r$,
and since $w(r) = 0$ for both points $w\in W$, the 
situation where only one of the points in $W$ satisfies $T(r)$ is impossible by symmetry.
Therefore the first claim holds.

Concerning the second claim, we first observe that $T(q)$
cannot be satisfied by both points of $W$, because if it was,
then we would have $W\models T(q)$ and thus $W\Vdash q$,
which  is a contradiction. Now assume that neither of the points of $W$
satisfies $T(q)$. Then, using the first claim established above (claim 1.), the
formulae $T(q)$ and $T(r)$, for $r\in\mathit{PROP}\setminus\{p,q\}$, are
locally equivalent in $W$.
Therefore we can now infer, by the
following argument, that the formulae $T(\cdep q)$ and $T(\cdep r)$
must also be locally equivalent in $W$.
Since the translation $T$ is compositional,
there exists an operator term $\psi(X)$ that describes
the translation of the operator $\cdep$,
and thus we have $T(\cdep q) = \psi(T(q))$
and $T(\cdep r) = \psi(T(r))$.
Since we know that $T(q)$ and $T(r)$ are 
locally equivalent in $W$, 
we immediately observe that $\psi(T(q))$
and $\psi(T(r))$ are also locally equivalent in $W$.
Thus $T(\cdep q)$ and $T(\cdep r)$
are locally equivalent in $W$.
Hence, as $W\not\Vdash\cdep q$ and thus
$W\not\models T(\cdep q)$, we infer that $W\not\models T(\cdep r)$.
Therefore $W\not\Vdash \cdep r$.
This is a contradiction, and thus the second claim holds.
Concerning the third claim, assume first that $T(\cdep p)$ is
satisfied by both points of $W$. Then $W\models T(\cdep p)$,
whence $W\Vdash \cdep p$, which is a contradiction.
Assume then that neither of the points in $W$ satisfies $T(\cdep p)$.
Therefore, using the
first claim (claim 1. above), $T(\cdep p)$ and $T(r)$ are locally equivalent in $W$.
As $W\Vdash \cdep p \vee \cdep p$, we have
$W\models T(\cdep p \vee \cdep p)$ and thus
$W\models \psi'(T(\cdep p),T(\cdep p))$,
where $\psi'(X,Y)$ is the operator term
for $\vee$ which demonstrates
that $T$ is indeed a compositional translation.
Since $T(\cdep p)$ and $T(r)$ are locally equivalent in $W$, we
infer that $W\models \psi'(T(r),T(r))$.
Thus $W\Vdash r\vee r$, which is a contradiction.
Therefore we conclude that $T(\cdep p)$ is satisfied by 
exactly one point of the model $W$, and hence the third claim holds.
We have now proved each of the above three claims.
By the last two of the three claims,
recalling that $T(\cdep p)$ can only use the
proposition symbol $p$ and $T(q)$ the symbol $q$, we
now observe that exactly one of the following
conditions hold.
\begin{enumerate}
\item
$T(\cdep p)$ is locally equivalent to $T(q)$ in $U$.
\item
$T(\cdep p)$ is locally equivalent to $T(q)$ in $V$.
\end{enumerate}
We first assume 
that the first one of these conditions holds.
Therefore, since $U\Vdash\cdep p\vee\cdep p$
and thus $U \models T(\cdep p\vee \cdep p)$, we may now conclude
that $U\models T( q\vee q)$ as follows.

We know that $T(Cp)$ and $T(q)$ are satisfied by exactly the
same single point in $U$. We know also
that $U \models  T(Cp \vee Cp)$.
Thus $U\models \psi'(T(Cp),T(Cp))$,
where $\psi'(X,Y)$ is the operator term for $\vee$.
Since $T(Cp)$ and $T(q)$ are locally 
equivalent in $U$, we therefore
have $U \models \psi'(T(q),T(q))$. Since 
$\psi'(T(q),T(q)) = T( q \vee q )$, we have $U\models T( q \vee q )$.
Since $U\models T( q \vee q )$, we have $U\Vdash q\vee q$,
which is a contradiction. Thus we turn to the case where $T(\cdep p)$ and $T(q)$
are locally equivalent  in $V$. Similarly to the above,
since $V\Vdash \cdep p\vee \cdep p$ and
thus $V\models T(\cdep p\vee \cdep p)$, we 
infer using the formula $\psi'(X,Y)$ that $V\models T(q\vee q)$.
Therefore $V\Vdash q\vee q$, which is a contradiction.
\end{proof}

{
\begin{comment}
%
\anote{Translating Dependence Logic with intuitionistic
%
disjunction into $\dlang$:}

%
\begin{enumerate}
%
\item
$t(p)= \Box p$ and $t(\neg p) = \Box\neg p$
%
\item
$t(\dep(q_1,\ldots,q_k,p)) = \dep(q_1,\ldots,q_k,p)$
%
%\item
%$t(\neg\dep(q_1,\ldots,q_k,p)) = \bot$
%
\item
$t(\varphi\wedge\psi) = \bigl(t(\varphi)\wedge t(\psi)\bigr)$
%
\item
$t(\varphi\mathrm{V}\psi) = \bigl(t(\varphi)\vee t(\psi)\bigr)$
%
\end{enumerate}
%
\end{comment}
}
We finish this section by mentioning some relevant related results in
the literature on compositional translations and
uniform definability of operators. Section 3.5 of \cite{ciardellimasters}
establishes that in the propositional inquisitive logic $\mathrm{InqL}$, which is a
team-based logic equi-expressive with $\mathcal{D}$, none of the primitive
operators is definable in terms of the others. In \cite{fanyangten}, it is
shown that the implication and disjunction connectives of $\mathrm{InqL}$ are not
uniformly definable in $\mathcal{D}$, and thus no compositional translation
from $\mathrm{InqL}$ into $\mathcal{D}$ is possible.
In \cite{pietroagainwhocares}, it is shown that the so-called
weak universal quantifier $\forall^{1}$ is not uniformly definable in
first-order dependence logic.
%

%%%%%%%%%%%%%%%%%%%%
%\section{Axiomatizing the validities in $\clang$}
%%%%%%%%%%%%%%%%%%%%
%%%%%%%%%%%%%%%%%%%%
\section{Validities and axiomatizations}\label{validities}
%%%%%%%%%%%%%%%%%%%%

In this section we provide sound and complete 
axiomatizations for $\mathcal{L}_{\dep}$ and $\mathcal{L}_{\ind}$.
We begin by axiomatizing $\mathcal{L}_{\cdep}$, the
fragment of $\mathcal{L}_{\dep}$ with only
operators $\cdep$ instead of general  
determinacy operators $\dep$.
%

%
%%%%%%%%%%%%%%%%%%%%
\subsection{Capturing the validities of $\clang$}
%%%%%%%%%%%%%%%%%%%%

Recall that $\cdep\varphi$ stands for $\dep(\epsilon;\varphi)$, where $\epsilon$ is the empty sequence of formulae.
Recall also the abbreviation $\ubox\varphi :=
\varphi\wedge\cdep\varphi$ and
the equivalence $\cdep \psi \equiv \ubox \psi \lor   \ubox \lnot \psi$
that intuitively demonstrate that the universal modality $\Ubox$ and $\cdep$
are expressible in terms of each other.
%

%
%Thus, the semantics of $\cdep$ is as follows. 
%
%\indent $W,w \models \cdep\fo$ iff for all $u,v\in W$: $(W,u\models\fo \Iff W,v\models\fo)$.
%
%In other words, $W,w \models \cdep\fo$ iff the truth value of $\fo$ is constant in the model. Therefore, the truth of $\cdep\fo$  does not depend on $w$. 

We denote the fragment of $\dlang$ that extends propositional logic PL with $\cdep$ by $\clang$. 
The operator $\cdep$ has been studied previously and in a more general setting as a ``non-contingency'' operator, and also---in epistemic logic---as a
``knowing whether" operator, see \cite{FanWD15} and the references therein.
%
\begin{comment}
%
The logic $\mathcal{L}_{\cdep}$ has been given complete axiomatizations before, but we
%
will anyway provide a rather straightforward complete
%
axiomatization and a related proof of completeness for it.
%
We will do this for the sake of completeness and also because the
%
related proof is nice and particularly simple.
%
\end{comment}
% 

%
\begin{comment}
%see e.g.  \cite{Demri},  \cite{Kuhn}.
%We will present without proof a sound and complete axiomatic system for the validities in \clang that correspond to the case of modal logic with the universal modality, equivalently S5. Several proofs of completeness of equivalent axiomatizations have already been provided, starting with 
\cite{montgomery}. For a comprehensive historical references and recent work on axiomatizations of \cdep in various modal logics, see \cite{FanWD15}. 
%Furthermore, we will obtain further a stronger result, axiomatizing the validities in the full \dlang.  
\end{comment}

We next present a sound and complete axiomatic system  \axcdep
that captures the validities of \clang.
Several proofs of completeness of equivalent
axiomatizations have already been provided in the literature, starting with
\cite{montgomery} and considered again in, e.g., \cite{DBLP:journals/sLogica/Demri97}.
For historical references and recent
related work on axiomatizations of \cdep,
see the above mentioned reference \cite{FanWD15}. 
Nevertheless, we will present here yet another, simple and intuitive axiomatization with a
purely syntactic proof of completeness by means of
reduction to the completeness of S5 (or \ulang).
We will then use the  completeness of \axcdep to
obtain complete axiomatizations for \dlang and \indlang.
The axiomatic system \axcdep is defined as follows.

\medskip

%The following formulae schemes are SD-valid. 
Axiom schemes: 
\begin{enumerate}
\item[Ax0(\cdep)]  A complete set of axioms for $\mathrm{PL}$
%propositional tautologies for completeness of $\mathrm{PL}$
%
\item[Ax1(\cdep)] $\cdep\, \top $ 

\item[Ax2(\cdep)] $\cdep\fo \leftrightarrow \cdep\lnot\fo$ 

\item[Ax3(\cdep)] $\cdep(\fo \land \cdep \fo)$ 

\item[Ax4(\cdep)] $\cdep\varphi \wedge \cdep\psi\ \, \rightarrow\ \,
\cdep(\varphi\wedge\psi)$

\item[Ax5(\cdep)] 
$\varphi\wedge \cdep\varphi \wedge \cdep(\varphi\rightarrow\psi)\ \, \rightarrow\ \, 
 \cdep\psi$
\end{enumerate}

Inference rules: 
\begin{enumerate}
\item[Rul0(\cdep)] Modus Ponens
%(MP). 

\item[Rul1(\cdep)]
EQ$_{\cdep}$:
If $\vdash \fo \leftrightarrow \fob$ then 
$\vdash \cdep \fo \leftrightarrow \cdep \fob$.
\end{enumerate}
We will denote derivability in  \axcdep by $\vdash_{\cdep}$.
%or $\vdash_{\axcdep}$.

The axiomatic system above is not minimal. 
For instance, Ax4(C) can be left out\footnote{Thanks to Jie Fan for noting that.}, as it is derivable (though, not quite trivially) from the others. Nevertheless, 
rather than providing a derivation of the axiom (which 
would not be in the focus of this paper), we prefer to keep it in the system.

\begin{comment}
%
\textcolor{blue}{This paragraph has already been discussed above.}
As noted above, the axiomatic system \axcdep is not very original, either. It is a variation of several others proposed for the non-contingency operator, e.g. the one in \cite{DBLP:journals/sLogica/Demri97}. However, we need to provide an axiomatization here, on which to build further on, and with the one presented here we want to promote the idea of a purely proof-theoretic completeness proof by syntactic reduction, which we also apply to the extensions. 
%
\end{comment}

\begin{proposition}\label{necforcproposition}
\label{cor:NEC}
The following inference rule, which preserves SD-validity,
can be used in \axcdep: 

\medskip

NEC$_{\cdep}$: If $\vdash_{\cdep} \fo$ then $\vdash_{\cdep} \cdep\fo$.
\end{proposition}

\begin{proof}
If $\vdash_{\cdep} \fo$, then $\vdash_{\cdep} \fo \leftrightarrow \top$ by PL
(propositional logic).
Thus $\vdash_{\cdep} \cdep \fo \leftrightarrow \cdep\, \top$ by EQ$_{\cdep}$, whence 
$\vdash_{\cdep} \cdep \fo$ by using Ax1(\cdep) and PL.  
\end{proof}

%
\begin{comment}
%
\textcolor{blue}{This is not immediate. But we do not need this.}
%
We note that the SD-valid schemata $\cdep\cdep \theta$ and 
$\cdep \lnot \cdep \theta$ are derivable in \axcdep.
%
To derive the former, we first derive $\cdep(\theta\wedge \cdep\theta)$
%
using Ax3(\cdep).
%
 The former can be easily derived from Ax5(\cdep) by letting $\fo := \theta\wedge \cdep\theta$ and $\fob := 
\cdep \theta$, and then the latter is derived from the former and Ax2(\cdep).
%
\end{comment}
%   

%\begin{corollary}
%\label{cor:ValidC}
% The following formulae schemes in \clang are SD-valid. 
%\begin{enumerate}
%\item $\cdep\top$ and $\cdep\bot$ 
%
%\item $\cdep\cdep \fo$ and $\cdep \lnot \cdep \fo$
%
%\item $\cdep\varphi \wedge \cdep\psi \rightarrow \cdep(\varphi\wedge\psi)$
%
%\item 
%$\varphi\wedge \cdep\varphi \wedge \cdep(\varphi\rightarrow\psi) \rightarrow 
% \cdep\psi$
%
%\item $\cdep\fo \leftrightarrow \cdep\lnot\fo$ 
%\end{enumerate}
%\end{corollary}

Now, recall the following well-known complete S5 axiomatization for 
 \ulang.  
%(see, e.g., \cite{GP92} and compare this axiomatization, 
%with the standard axiomatizations of S5).
%which is well-known to be equivalent to S5.

\medskip

Axiom schemata: 
\begin{enumerate}
\item[Ax0(\Uman)] A complete set of axioms for $\mathrm{PL}$. 
%enough propositional tautologies for completeness w.r.t. PL
%
\item[Ax1(U)]
$\Ubox(\varphi\ \rightarrow\ \psi)\ \rightarrow\ (\Ubox\varphi\
\rightarrow\ \Ubox\psi)$
\item[Ax2(\Uman)]
$\Ubox\varphi\ \rightarrow\ \varphi$
\item[Ax3(\Uman)] 
$\Udiam\varphi\ \rightarrow\ \Ubox\Udiam\varphi$
\end{enumerate}

Inference rules: 
\begin{enumerate}
\item[Rul0(\Uman)] 
Modus Ponens
\item[Rul1(\Uman)]
NEC$_{\mathrm{U}}$: If $\vdash \fo$ then $\vdash \Ubox\fo$
\end{enumerate}
We will denote derivability in $\mathcal{L}_{\mathcal{U}}$ by $\vdash_{S5}$. 

To show completeness of \axcdep, we first extend the
intuitive interdefinability of $\cdep$ and 
the universal modality to a translation $\varphi\mapsto\varphi^+$ from $\clang$
into $\ulang$ and a translation $\varphi\mapsto\varphi^{\circ}$ from $\ulang$  
into $\clang$.

%\vnote{Note that $\vee$ and $\leftrightarrow$ are not primitives but definable in \clang, so the translation below is formally not quite correct.}

%
The translation $\varphi\mapsto\varphi^+$ from \clang
into  $\mathcal{L}_{\mathcal{U}}$ is defined as follows.
\begin{enumerate}
\item
$p^+ = p$
\item
$(\neg\varphi)^+ = \neg\varphi^+$
\item
$(\varphi\rightarrow\psi)^+ =  (\varphi^+\rightarrow\psi^+)$
\item
$(\cdep\varphi)^+ = \Ubox\,\varphi^+\, \vee\, \Ubox\,\neg\varphi^+$
\end{enumerate}

The translation $\varphi\mapsto\varphi^\circ$
from $\mathcal{L}_{\mathcal{U}}$ into  \clang goes as follows.
\begin{enumerate}
\item
$p^{\circ} = p$
\item
$(\neg\varphi)^{\circ} = \neg\varphi^{\circ}$
\item
$(\varphi\rightarrow\psi)^{\circ} =  (\varphi^{\circ}\rightarrow\psi^{\circ})$
\item
$(\Ubox\varphi)^{\circ} = (\varphi^{\circ}\, \wedge\, \cdep\, \varphi^{\circ})$ 
%\vnote{$\wedge$ is definable, too.}
%
\end{enumerate}
\begin{lemma}
\label{lem:plus}
For every formula $\varphi$ of $\clang$, $\models  \varphi$ iff 
$\models  \varphi^+$,
where the validity statement in each case refers to the semantics of
the language in question.
Moreover, the translation $()^+$ preserves, both ways, truth in states and therefore validity in models. 
\end{lemma}
\noindent
The proof of the lemma is straightforward.
% by structural induction on $\phi$.
%It is easier to apply to the contrapositive
%equivalence: $\phi^+$ is false at some
%state of a Kripke model iff $\phi$ is false at that same state. 

%
The composition of the two
translations, first $()^+$ and then $()^\circ$,
defines the following translation $\varphi\mapsto\varphi^*$ from $\clang$ into $\clang$: 
\begin{enumerate}
\item
$p^* = p$
\item
$(\neg\varphi)^* = \neg\varphi^*$
\item
$(\varphi\rightarrow\psi)^* =  (\varphi^*\rightarrow\psi^*)$
\item
$(\cdep\varphi)^* = (\varphi^*\wedge \cdep\varphi^*)\vee
(\neg \varphi^*\wedge \cdep\neg \varphi^*)$ 
\end{enumerate}
The following lemma shows that
we can derive equivalence of $\varphi$
and its translation $\varphi^*$ in \axcdep.
\begin{lemma}
\label{lem:star}
We have $\vdash_{\cdep} \varphi\leftrightarrow\varphi^*$
for every formula $\varphi \in \clang$.
\end{lemma}
\begin{proof}
The proof proceeds by induction on the structure of $\varphi$.
The case for atoms is trivial and the cases for the Boolean
connectives follow easily from the induction hypothesis
using the fact that \axcdep is complete
with respect to PL. 
We thus discuss the case involving $\cdep\varphi$.
By the induction hypothesis, we
have $\vdash_{\cdep} \varphi\leftrightarrow\varphi^*$.
By the inference rule EQ$_{\cdep}$, we obtain  
\[\vdash_{\cdep} \cdep\varphi\leftrightarrow \cdep\varphi^*.\] 
On the other hand, by PL, we have
\[\vdash_{\cdep} \cdep\varphi\leftrightarrow\bigl
((\varphi\wedge \cdep\varphi)\vee (\neg\varphi\wedge \cdep\varphi)\bigr).\]
Then, again using PL together
with  the induction hypothesis and the equivalences above, we derive 
\[\vdash_{\cdep} \cdep\varphi\leftrightarrow\bigl
((\varphi^*\wedge \cdep\varphi^*)\vee (\neg\varphi^*\wedge \cdep\varphi^*)\bigr).\] 
Finally, using the axiom $\cdep\theta\leftrightarrow \cdep\neg\theta$ with $\theta := \varphi^*$ and PL, 
we get
\[\vdash_{\cdep} \cdep\varphi\leftrightarrow \bigl((\varphi^*\wedge \cdep\varphi^*)
\vee (\neg\varphi^*\wedge \cdep\neg\varphi^*)\bigr).\]
This concludes the proof.
\end{proof}
Next we will show that every
derivation in $\mathcal{L}_{\mathcal{U}}$ can be
simulated by a derivation in \axcdep.
\begin{lemma}
\label{lem:circ}
For every formula $\varphi \in \ulang$,  
if\, $\vdash_{\mathrm{S5}} \varphi$ then  $\vdash_{\cdep}\varphi^{\circ}$.
\end{lemma}

\begin{proof}
The proof proceeds by induction on derivations in $\mathcal{L}_{\mathcal{U}}$.
We will first prove that $\vdash_{\cdep} \varphi^{\circ}$ for each 
axiom $\varphi$ for $\mathcal{L}_{\mathcal{U}}$.
For propositional tautologies this is trivial. 
To deal with Ax1(\Uman), we must show that
\[\vdash_{\cdep}  ((\theta\rightarrow\psi ) \wedge \cdep(\theta\rightarrow\psi))\rightarrow
((\theta\wedge \cdep\theta)\rightarrow(\psi\wedge \cdep\psi))\]
for arbitrary $\theta$ and $\psi$. 
The following derivation does exactly this. 
(The steps after the first one use PL and the preceding steps.)
\begin{enumerate}
\item  
$\vdash_{\cdep} (\theta\wedge \cdep\theta\wedge \cdep(\theta\rightarrow\psi)) 
\rightarrow \cdep\psi$ 
\hspace{\fill} Ax5(\cdep) 

\item 
$\vdash_{\cdep}  (\theta\wedge\psi \wedge \cdep\theta \wedge \cdep(\theta\rightarrow\psi))
\rightarrow (\psi\wedge \cdep\psi)$ 
\hspace{\fill} %by 1 and PL

\item 
$\vdash_{\cdep}  (\theta\wedge (\theta\rightarrow\psi ) \wedge \cdep\theta \wedge \cdep(\theta\rightarrow\psi))
\rightarrow (\psi\wedge \cdep\psi)$ 
\hspace{\fill} %by 2 and PL

\item $\vdash_{\cdep}  ((\theta\rightarrow\psi ) \wedge \cdep(\theta\rightarrow\psi)
\wedge (\theta\wedge \cdep\theta))\rightarrow (\psi\wedge \cdep\psi)$ 
\hspace{\fill} 

\item 
$\vdash_{\cdep}  ((\theta\rightarrow\psi ) \wedge \cdep(\theta\rightarrow\psi))\rightarrow
((\theta\wedge \cdep\theta)\rightarrow(\psi\wedge \cdep\psi))$ 
\hspace{\fill}  
\end{enumerate}

To cover axiom Ax2(\Uman), we must show that $\vdash_{\cdep}(\theta\wedge \cdep\theta) \rightarrow \theta$,
which is a propositional tautology.

To deal with axiom Ax3(\Uman), we must show that
%
%\begin{equation}\label{esfivething}
%
\[
\vdash_{\cdep}\neg(\neg \theta\wedge \cdep\neg \theta)\rightarrow
\bigl(\neg(\neg \theta\wedge \cdep\neg \theta)
\wedge \cdep\neg(\neg \theta\wedge \cdep\neg \theta)\bigr).
\]
%
%\end{equation}
%
Here is the derivation.
\begin{enumerate}
\item 
$\vdash_{\cdep} \neg(\neg \theta\wedge \cdep\neg \theta)\rightarrow
\neg(\neg \theta\wedge \cdep\neg \theta)$
\hspace{\fill} by PL

\item  
$\vdash_{\cdep} \cdep(\neg\theta\wedge\cdep\neg\theta)$
\hspace{\fill} Ax3(\cdep) 

\item  
$\vdash_{\cdep} \cdep\neg(\neg\theta\wedge\cdep\neg\theta)$
\hspace{\fill} by row 2, Ax2(\cdep) and PL

\item 
$\vdash_{\cdep} \neg(\neg \theta\wedge \cdep\neg \theta)\rightarrow
\bigl(\neg(\neg \theta\wedge \cdep\neg \theta)
\wedge \cdep\neg(\neg \theta\wedge \cdep\neg \theta)\bigr)$
\hspace{\fill} by 1, 3 and PL
\end{enumerate}
%Therefore we conclude by propositional logic that  Equation \ref{esfivething} holds.
%

%
Now it remains to establish that NEC$_{\Uman}$ preserves the claim,
that is, we will show that if $\vdash_{\mathrm{S5}} \varphi$
and thus $\vdash_{\cdep}\varphi^{\circ}$ by the induction hypothesis,
then we also have $\vdash_{\cdep}(\Ubox\varphi)^{\circ}$.
Thus we assume that $\vdash_{\cdep}\varphi^{\circ}$.
Using $\mathrm{NEC}_\cdep$ (see Proposition \ref{necforcproposition}),
we infer that $\vdash_{\cdep} \cdep \varphi^{\circ}$, and
using PL, we get $\vdash_{\cdep} \varphi^{\circ}\wedge \cdep\varphi^{\circ}$.
As $(\Ubox\varphi)^{\circ} = \varphi^{\circ}
\wedge \cdep\varphi^{\circ}$, we are done.
\end{proof}

We are now ready to prove the  soundness and completeness of \axcdep.

\begin{proposition}
\label{prop:AxC}
The axiomatic system \axcdep is sound and complete for the validities of \clang.
\end{proposition}
\begin{proof}
The soundness follows by an easy verification of the validity of all the axioms and the
fact that EQ$_{\cdep}$ preserves validity.
To prove completeness, assume that $\models \fo$ for some
formula $\fo$ of $\clang$. Then $\models  \fo^+$ by Lemma \ref{lem:plus}.
By completeness of $\mathcal{L}_{\mathcal{U}}$, we have $\vdash_{S5} \fo^+$.
By Lemma \ref{lem:circ}, we
thus have $\vdash_{\cdep} (\fo^+)^\circ$, i.e., $\vdash_{\cdep} \fo^*$.
Therefore $\vdash_{\cdep} \fo$ by Lemma \ref{lem:star}.
\end{proof}
\subsection{Complete axiomatizations of $\dlang$ and $\indlang$}
We now define sound and complete axiomatic systems \axdep and \axindep
for \dlang and $\mathcal{L}_{\ind}$ by introducing 
new axiom schemata.
The axiomatic system \axdep for \dlang is obtained by extending \axcdep.
The idea is simply to define $\dep$ in
terms of $\cdep$. 
A suitable definition is
obtained from the equivalence established in Proposition \ref{prop:Dep}.
Recall that $\ubox\varphi$ is an abbreviation of $\varphi\wedge\cdep\varphi$.
We define, for each positive integer $k$,
the following axiom schema:
\[
Ax(\dep_k) \ \ \ \ \ \ \ \ \ \ \ 
\dep(\varphi_1,\ldots,\varphi_k,\psi)\ \ \leftrightarrow\
\bigvee\limits_{\chi\, \in\, \nff(\varphi_1,\ldots,\varphi_k)}
\ubox(\chi\leftrightarrow\psi).
%\bigl(\, (\chi\leftrightarrow\psi)\ \wedge\, \cdep(\chi\leftrightarrow\psi)\, \bigr)
\]
The system \axdep consists of the axiom schemata and rules of $\mathcal{L}_{\cdep}$
together with the above axiom schemata for each $k\in\mathbb{Z}_+$.
We obtain an axiomatic system \axindep for $\mathcal{L}_{\ind}$
similarly by essentially extending \axcdep by schemata that
define by Proposition \ref{prop:Ind} the operator $\ind$ in terms of $\cdep$.
The language $\mathcal{L}_{\ind}$ 
does not contain $\cdep$ as a 
primitive, but the translation $t$ given 
before Proposition \ref{expr} shows that the operator $\cdep$ can be
expressed as $\varphi\, \ind\, \varphi$. Thus we
first define $\mathit{AX}_0(\mathcal{L}_{\ind})$ to be the following
system\footnote{Provided here for readers' convenience.} obtained from $\axcdep$ by
the substitution $\cdep\theta\mapsto\theta\ind\theta$.
%
%
%
%\medskip
%

\medskip

Axiom schemes: 
%\footnote{These are obtained from  \axcdep by replacing $\cdep \fo$ with $\fo\ind\fo$, and provided here only for readers' convenience.}:  

% 
\begin{enumerate}
\item[Ax0(\ind)] A complete set of axioms for $\mathrm{PL}$.
%enough propositional tautologies for completeness w.r.t. $\mathrm{PL}$
%
\item[Ax1(\ind)] $\top\, \ind\, \top$ 

\item[Ax2(\ind)] $\fo\ind\fo\, \leftrightarrow\,\neg\fo\ind\neg\fo$

\item[Ax3(\ind)] $(\fo \land \fo\ind\fo)\ind(\fo \land \fo\ind\fo)$ 

\item[Ax4(\ind)] $\varphi\ind\varphi \wedge \psi\ind\psi\ \, \rightarrow\ \,
(\varphi\wedge\psi)\ind(\varphi\wedge\psi)$

\item[Ax5(\ind)] 
$\varphi\wedge \varphi\ind\varphi \wedge
(\varphi\rightarrow\psi)\ind(\varphi
\rightarrow\psi)\ \, \rightarrow\ \, \psi\ind\psi$
\end{enumerate}
Inference rules: 
\begin{enumerate}
\item[Rul0(\cdep)] Modus Ponens
%(MP). 

\item[Rul1(\ind)]
EQ$_{\ind}$:
If $\vdash \fo \leftrightarrow \fob$ then
$\vdash \fo\ind\fo \leftrightarrow \fob\ind\fob$.
\end{enumerate}
Recalling the abbreviation $\udiamp$ from
Section \ref{indsection}, we define \axindep to be
the extension of $\mathit{AX}_0(\mathcal{L}_{\ind})$ by the
following axiom schemata
for all $m\in\mathbb{N}$ and $k,n\in\mathbb{Z}_{+}$:
\begin{multline*}
Ax(\ind_{k,m,n}) \ \ \ \ \ \ 
(\fo_1,\ldots,\fo_k)\, \ind_{(\fod_1,\ldots,\fod_m)}(\fob_1,\ldots,\fob_n) 
\leftrightarrow \\
\bigwedge\limits_{
(\varphi,\theta,\psi)\, \in\, B}  
\Big(\bigl(\udiamp(\fod \wedge \fo)
\wedge\udiamp(\fod \wedge \fob)\bigr)
\rightarrow\ \udiamp(\fod \wedge \fo \wedge \fob) \Big),
\end{multline*}
where $B$ is as in Proposition \ref{prop:Ind}.
We denote derivability in  \axdep by $\vdash_{\dep}$ and 
derivability in  \axindep by $\vdash_{\ind}$.

%\vnote{Can we show that there are no finitary axiomatizations?}

\begin{theorem}
%\label{prop:AxD}
\label{prop:Compl-DI} 
~
\begin{enumerate}
\item  \axdep is sound and complete for the validities of \dlang.
\item  \axindep is sound and complete for the validities of \indlang.
\end{enumerate}
\end{theorem}
\begin{proof}
The proofs of the two claims are very similar, so we will first present the
argument for \axdep and then briefly comment the
claim for \axindep.
Soundness follows from the soundness of \axcdep and Proposition \ref{prop:Dep}. 
To prove completeness, we will use a similar argument as the one
applied in the proof of
Proposition \ref{prop:AxC}.
We will reduce the completeness  of  \axdep to the
already proved completeness of \axcdep. 
We first define the obvious translation $tr$ of \dlang into \clang which leaves all atoms and
Boolean connectives intact
and likewise translates $\cdep\varphi$ to $\cdep\, tr(\varphi)$,
but treats formulae $\dep(\varphi_1,\ldots,\varphi_k;\psi)$ with $k\not=0$ as
follows. Using the equivalence 
established by Proposition \ref{prop:Dep}, we put 
\begin{multline*}
tr(\dep(\varphi_1,\ldots,\varphi_k;\psi)) :=\\  \bigvee\limits_{\chi\, \in\, 
\nff(\varphi_1,\ldots,\varphi_k)} 
\bigl(\, (tr(\chi) \leftrightarrow tr(\psi))\
\wedge\, \cdep(tr(\chi) \leftrightarrow tr(\psi))\, \bigr).
\end{multline*}
%

%
%\[tr_{\dep}(\dep(\varphi_1,\ldots,\varphi_k,\psi)) :=  \bigvee\limits_{\chi\, \in\, \nff(tr_{\dep}(\varphi_1),\ldots,tr_{\dep}(\varphi_k))}
%\bigl(\, (\chi\leftrightarrow tr_{\dep}(\psi))\ \wedge\, \cdep(\chi\leftrightarrow tr_{\dep}(\psi)))\, \bigr)\]

\vcut{
Now, we show by induction on the nesting depth of $\dep$ (or, number of occurrences of $\dep$) in formulae $\theta$ in \dlang that $\models \theta \leftrightarrow  tr_{\dep}(\theta)$ and therefore 
$\models \theta$ \  iff  $\models tr_{\dep}(\theta)$. The only non-trivial step is when $\theta = \dep(\varphi_1,\ldots,\varphi_k,\psi)$, assuming that the claim holds for all formulae of smaller nesting depth of $\dep$, in particular for $\psi$ and for every 
$\chi \in \nff(\varphi_1,\ldots,\varphi_k)$. Then, to prove $\models \dep(\varphi_1,\ldots,\varphi_k,\psi) \leftrightarrow tr_{\dep}(\dep(\varphi_1,\ldots,\varphi_k,\psi))$ we use the inductive hypothesis, 
Proposition \ref{prop:Dep}, the equivalent replacement property in PL and the (sound) semantic version of the rule EQ$_{\cdep}$, claiming that if $\models \fo \leftrightarrow \fob$ then 
$\models \cdep \fo \leftrightarrow \cdep\fob$. 
}

We then  prove by induction on the structure of 
formulae $\theta$ of \dlang that
$$\vdash_{\dep} \theta \leftrightarrow tr(\theta).$$
The cases for proposition symbols and Boolean
connectives are trivial.
To cover the case for $\cdep$,
assume we have shown that $\vdash_{\dep}\varphi \leftrightarrow tr(\varphi)$.
We then conclude that $\vdash_{\dep}\cdep\varphi
\leftrightarrow \cdep\, tr(\varphi)$
directly by the rule $\mathrm{EQ}_{\cdep}$.
To deal with the case for $\dep$,
let $\theta = \dep(\varphi_1,\ldots,\varphi_k;\psi)$,
and let the induction hypothesis be
that $\vdash_{\dep}\varphi_i\leftrightarrow tr(\varphi_i)$ for
each $i\leq k$ and $\vdash_{\dep}\psi\leftrightarrow tr(\psi)$.
From here it is easy to conclude, using
completeness with respect to propositional logic,
that we also have $\vdash_{\dep}\chi\leftrightarrow tr(\chi)$ for
each $\chi\in\mathit{DNF}(\varphi_1,...,\varphi_k)$.
Therefore, using PL, we have
$$\vdash_{\dep} (\chi\leftrightarrow\psi)
\leftrightarrow (tr(\chi)\leftrightarrow tr(\psi)),$$
whence we infer by the rule $\mathrm{EQ}_{\cdep}$ that we have
$$\vdash_{\dep} \cdep(\chi\leftrightarrow\psi)
\leftrightarrow \cdep(tr(\chi)\leftrightarrow tr(\psi)).$$
Using this equivalence and the already 
established fact that $\vdash_{\dep} \chi'\leftrightarrow tr(\chi')$
for all $\chi'\in\{\psi\}\cup\mathit{DNF}(\varphi_1,...,\varphi_k)$, we
then infer by PL that 
\begin{multline*}
\vdash_{\dep}\bigvee\limits_{\chi\, \in\, 
\nff(\varphi_1,\ldots,\varphi_k)} 
\bigl(\, (\chi \leftrightarrow \psi)\
\wedge\, \cdep(\chi \leftrightarrow \psi)\, \bigr)\\
\leftrightarrow\ \bigvee\limits_{\chi\, \in\, 
\nff(\varphi_1,\ldots,\varphi_k)} 
\bigl(\, (tr(\chi) \leftrightarrow tr(\psi))\
\wedge\, \cdep(tr(\chi) \leftrightarrow tr(\psi))\, \bigr).
\end{multline*}
From here we conclude,
using propositional logic and Ax($\dep_k$), that
\begin{multline*}
\vdash_{\dep}   \dep(\varphi_1,...,\varphi_k;\psi)\\
\leftrightarrow\ \bigvee\limits_{\chi\, \in\, 
\nff(\varphi_1,\ldots,\varphi_k)} 
\bigl(\, (tr(\chi) \leftrightarrow tr(\psi))\
\wedge\, \cdep(tr(\chi) \leftrightarrow tr(\psi))\, \bigr).
\end{multline*}
In other words, 
we have $\vdash_{\dep} \dep(\varphi_1,\ldots,\varphi_k,\psi)
\leftrightarrow tr(\dep(\varphi_1,\ldots,\varphi_k,\psi))$,
whence we have now established
that $\vdash_{\dep} \theta\leftrightarrow tr(\theta)$
for all $\theta$ of $\mathcal{L}_{\dep}$.
%
%
%
\begin{comment}
%
we use the inductive hypothesis, Ax(\dep), the
%
deductive version of the equivalent
%
replacement property in PL (which holds by
%
completeness of PL), and the rule EQ$_{\cdep}$.
%

%
By soundness of \axdep, we now conclude that $\models \theta \leftrightarrow  tr_{\dep}(\theta)$ for every formula $\theta$ in \dlang, and therefore $\models \theta$ \  iff  $\models tr_{\dep}(\theta)$.
%
\end{comment}
%

%
To conclude the proof,
assume that $\models \theta$ for some $\theta$ of $\dlang$.
Then $\models tr(\theta)$ by
soundness of the translation $tr$.
Hence, recalling that $tr(\theta)$ is a
formula of $\clang$, we
have $\vdash_{\cdep} tr(\theta)$ by completeness of \axcdep.
Using the fact that $\vdash_{\dep}
\theta \leftrightarrow tr(\theta)$, we
extend the derivation of
$tr(\theta)$ in \axcdep to a
derivation of $\theta$ in \axdep.
Therefore $\vdash_{\dep}\theta$.
The completeness proof of $\axindep$ is similar.
We first prove that $\mathit{AX}_{0}(\mathcal{L}_{\ind})$ is
complete for the notational variant of $\mathcal{L}_{\cdep}$
that replaces $\cdep\varphi$ with $f(\varphi) \ind f(\varphi)$,
where $f$ is a translation that
keeps proposition variables and Boolean connectives intact
but treats $\cdep$ as given here. This proof of 
completeness is virtually identical to the corresponding
argument for $\mathit{AX}(\mathcal{L}_{\cdep})$ given above.
Then the completeness of \axindep is proved 
similarly to the way \axdep was treated above,
the only significant (but uncomplicated) difference being that
the axioms $\mathit{Ax}(\ind_{k,m,n})$ instead of axioms $\mathit{Ax}(\dep_k)$
are used.
\end{proof}

%
%\begin{comment}
%
Our axiomatizations for $\mathcal{L}_{\dep}$ and 
$\mathcal{L}_{\ind}$ are not finite because we have the
schemata $Ax(\dep_k)$ and $Ax(\ind_{k,m,n})$
for infinitely many values of $k,m,n$.\footnote{Indeed, bounding
these values by some small constant would result in  more
elegant axiomatizations that then would, however, only work for
bounded versions of $\dep$ and $\ind$ with bounded arities.}
We will next show that, in fact, neither $\mathcal{L}_{\dep}$
nor $\mathcal{L}_{\ind}$ \emph{has} a finite axiomatization.
To this end, we will first define formally what we mean by a
finite axiomatization.
An \emph{axiom schema}
for $\mathcal{L}_{\dep}$ (respectively, $\mathcal{L}_{\ind}$) is an
object obtained from a formula $\chi$ of $\mathcal{L}_{\dep}$ ($\mathcal{L}_{\ind}$)
by substituting \emph{schema letters} $\varphi_1,\ldots,\varphi_n$
for all proposition symbols in $\chi$. A \emph{proof rule}
for $\dlang$ ($\mathcal{L}_{\ind}$) is an implication of the form
$\vdash \psi_1,\ \ldots\ ,\ \vdash \psi_k\ \ \Rightarrow\ \ \ \vdash\chi,$
where $\psi_1,\ldots,\psi_k,\chi$ are axiom
schemata for $\dlang$ ($\mathcal{L}_{\ind}$) and $k$ a positive integer.
A \emph{finite axiomatization} for $\dlang$ ($\mathcal{L}_{\ind}$)
is a pair $(\Phi,\Psi)$, where $\Phi$ is a finite set of
axiom schemata 
and $\Psi$ a finite set of proof rules for $\dlang$ ($\mathcal{L}_{\ind}$).
\begin{theorem}
Neither $\dlang$ nor $\mathcal{L}_{\ind}$ has a sound and
complete finite axiomatization.
\end{theorem}
\begin{proof}
We discuss $\mathcal{L}_{\dep}$ only. The argument for $\mathcal{L}_{\ind}$ is
similar.
Assume $(\Phi,\Psi)$ is a sound and complete finite
axiomatization for $\mathcal{L}_{\dep}$. Let $k$ be the maximum number
such that some schema in $\Phi\cup\Psi$ contains a
subschema of the type $\dep(\psi_1,\ldots,\psi_k\, ;\, \varphi)$.
%
%
%
\begin{comment}
%
Define a non-standard semantics for $\mathcal{L}_{\dep}$ that interprets each
%
formula of the type $\dep(\varphi_1,\ldots,\varphi_{k+1};\varphi)$ such
%
that $W,w\models \dep(\varphi_1,\ldots,\varphi_{k+1};\varphi)$ is always false,
%
i.e., determinacy statements involving $k+1$ antecedent formulae are never true,
%
but all other types of formulae have their usual semantics.
%
We claim that $(\Psi,\Phi)$ is sound with respect to
%
this semantics. This will conclude the proof
%
since $\dep(p_1,\ldots p_k; q)\vee D(p_1,\ldots, p_k;\neg q)$ is theorem of
%
standard $\mathcal{L}_{\dep}$ but $\bot\vee\bot$ is not a
%
theorem of the nonstandard semantics.
%
\end{comment}
%
%
%
For each formula $\chi$ of $\mathcal{L}_{\dep}$, let $\chi(\bot)$
denote the formula obtained from $\chi$ by replacing
(in any order) each subformula of the type $\dep(\alpha_1,\ldots,\alpha_{k+1};\beta)$
by $\bot$. We will show by induction on deductions
that for all formulae $\chi$ of $\mathcal{L}_{\dep}$, 
if $\chi$ is a theorem of $(\Phi,\Psi)$, then also $\chi(\bot)$ is a theorem of $(\Phi,\Psi)$.
This will conclude the proof for the following reason.
Consider the formula $\dep(p,\ldots, p\, ; p)$, where $p$ is simply
repeated $k + 2$ times. This formula is a
theorem of $\mathcal{L}_{\dep}$, while $\bot$ is not.
The inductive argument is based on the 
following observation: if $\alpha$ is a
formula obtained from a schema $\varphi$ by substitution,
then also the formula $\alpha(\bot)$ can be obtained from $\varphi$
by substitution, because the schema $\varphi$ does not involve any
subschemata of the type $\dep(\chi_1,\ldots,\chi_{k+1};\chi')$.
Therefore the basis of the induction, which deals with the direct use of 
axiom schemata as a first step of a deduction, is clear.
The induction step is based on similar reasoning.
For consider a proof rule $\vdash\varphi_1,
\ldots,\vdash\varphi_m\Rightarrow\ \vdash\psi$ and
assume that we have deduced some formula $\beta$
by applying an instance
$\vdash\alpha_1,
\ldots,\vdash\alpha_m\Rightarrow\ \vdash \beta$
of this rule to some formulae $\alpha_1,\ldots,\alpha_m$
such that $\vdash\alpha_1,\ldots,\vdash\alpha_m$.
Since the schemata $\varphi_1,\ldots,\varphi_m,\psi$
do not contain subschemata of the type $\dep(\chi_1,\ldots,\chi_{k+1},\chi')$,
we observe that the implication $\vdash\alpha_1(\bot),
\ldots,\vdash\alpha_m(\bot)\Rightarrow\ \vdash \beta(\bot)$ is
also an instance of the rule $\vdash\varphi_1,
\ldots,\vdash\varphi_m\Rightarrow\ \vdash\psi$.
Since $\vdash\alpha_1,\ldots,\vdash\alpha_m$, we
have $\vdash\alpha_1(\bot),\ldots,\vdash\alpha_m(\bot)$ by
the induction hypothesis.
Thus $\vdash \beta(\bot)$, as required.
\end{proof}
\section{The road ahead}\label{futuredirections}
%%%%%%%%%%%%%%%%%%%%%%%%%%%%%%%

%
In this paper we have defined the logics $\mathcal{L}_{\dep}$
and $\mathcal{L}_{\ind}$ as alternatives for $\mathcal{D}$ and $\mathcal{I}$.
We have comprehensively studied the
expressive powers of these four logics and argued for the naturalness of $\mathcal{L}_{\dep}$
and $\mathcal{L}_{\ind}$ in relation to $\mathcal{D}$ and $\mathcal{I}$.
We have also provided sound and complete axiomatizations for $\mathcal{L}_{\dep}$
and $\mathcal{L}_{\ind}$.
Here we discuss briefly a range of natural future developments of the present work.
%

%
%%%%%%%%%%%%%%%%%%%%%%%%%%%%%%%%%%%%%
\subsection{Relativised determinacy operators}
%%%%%%%%%%%%%%%%%%%%%%%%%%%%%%%%%%

The determinacy of a formula by a set of formulae
can be relativised to a set of possible worlds.
Let $\extn{\theta}_{W}$ denote 
the set $\{w\in W\, |\, W,w\models\theta\}$ 
and define the \emph{relativised 
determinacy operator} $\dep^{\theta}(\varphi_1,\dots,\varphi_k ; \psi)$
such that
$W,w \models \dep^{\fod}(\varphi_1,\dots,\varphi_k ; \psi)$  iff
for all $u,v\in\extn{\theta}_{W}$, if the
equivalence $$W,u\models\varphi_i \Leftrightarrow W,v\models\varphi_i$$
holds for each $i\leq k$, then we have
$$W,u\models\psi\Leftrightarrow W,v\models\psi.$$
In particular, we define $\cdep^{\fod}\fo = \dep^{\fod}(\emptyset,\fo)$.  
Using $\ubox$, we notice that 
\[\cdep^{\fod}\fo \equiv \ubox (\fod \to \fo) \lor \ubox (\fod \to \lnot \fo).\] 
We also notice  that $\dep(\fod,\fo) $ is
definable in terms of $\cdep^{\fod}$ as follows.
\[\dep(\fod,\fo) \equiv \cdep^{\fod} \fo \land  \cdep^{\lnot \fod}\fo.\]
Furthermore, $\dep^{\fod}$ is inductively definable in
terms of $\cdep^{\fod}$ as follows.
\begin{enumerate}
\item
$\dep^{\fod}(\emptyset,\fo)  = \cdep^{\fod} \fo$
\item
$\dep^{\fod}(\varphi_1,\ldots,\varphi_{k+1},\psi)\\ \equiv
\dep^{\fod\land \varphi_{k+1}}(\varphi_1,\ldots,\varphi_k,\psi)
\land \dep^{\fod\land \lnot\varphi_{k+1}}(\varphi_1,\ldots,\varphi_k,\psi)$
%
%\end{multline*}
%
%
%
\end{enumerate}
We provide a simple example illustrating the use of relativised
determinacy in natural language. Consider a scenario with the following propositions.
\begin{itemize}
\item
\emph{There are road blocks}, denoted by $p$.
\item
\emph{It is rush hour}, denoted by $q$.
\item
\emph{John will be on time}, denoted by $r$.
\end{itemize}
%

%%
%\subsubsection{Relativised determinacy and natural language}
%%

%
Assume the set of possible worlds in the scenario is 
$W = \{w_1,w_2,w_3,w_4,w_5\}$,
where
\begin{itemize}
\item
$w_1 = \{(p,0), (q,0),(r,1)\}$,
\item
$w_2 = \{(p,0), (q,1),(r,1)\}$,
\item
$w_3 = \{(p,0), (q,1),(r,0)\}$,
\item
$w_4 = \{(p,1), (q,0),(r,1)\}$,
\item
$w_5 = \{(p,1), (q,1),(r,0)\}$.
\end{itemize}
Consider the following claim.

\medskip

\emph{If there are road blocks, then,
whether it is rush hour determines whether John will be on time.}

\medskip

Perhaps the  most natural interpretation for this sentence is given by 
$\dep^{p}(q,r)$, which is true in each world of the  model $W = \{w_1,\ldots,w_5\}$,
and thus we have $W\models \dep^{p}(q,r)$.
Interestingly, however, there is at least one sensible
alternative interpretation, $p\rightarrow \dep(q,r)$, which
is true in the worlds $w_1,w_2,w_3$,
and false in $w_4,w_5$, whence we have $W\not\models p\rightarrow \dep(q,r)$.
%

%The validities and axiomatization of the extension of \dlang with relativised determinacy operator will be studied in the follow-up to the current work. 

%%%%%%%%%%
%\subsubsection{Some valid principles of relativised determinacy}
%%
%
%
%\begin{proposition}
%\label{prop:Valid-relD}
%The following formulae schemes for $\dep^{\fod}$  are SD-valid. 
%
%\vnote{ADD} 
%
%%\begin{enumerate}
%%\item 
%%\end{enumerate}
%\end{proposition}
%
%
%Some additional rules preserving validity:  
%
%\vnote{ADD} 
%
%%\begin{enumerate}
%%\item 
%%\end{enumerate}
%
%
%%%%%%%%%%%%%%%%%%%%%
%\subsubsection{A complete axiomatic system for $\dep^{\fod}$}
%%%%%%%%%%%%%%%%%%%%%
%
%\vnote{ADD} 

%\begin{proposition}
%The axiomatic system \axdep, consisting of the schemes in Proposition \ref{prop:ValidD} is sound and complete for the SD-validities in 
%\dlang. 
%\end{proposition}
%
%\begin{proof}
%
%\vnote{Add proof. If necessary, add more axioms.}
%
%\end{proof}

%%%%%%%%%%%%%%%%%%%%%%%%%%%%%%%%
\subsection{Determinacy and independence in the general modal setting}\label{Kripke}
%%%%%%%%%%%%%%%%%%%%%%%%%%%%%%%

As already mentioned in the introduction section, the determinacy operator $\dep$ 
can be naturally generalised to the general modal setting. 
%The assignment $V$ extends inductively over the set of all modal formulae in the standard Kripke semantics way, so that 
%\[V(\fo) := \{w \in W \mid M,w \models \fo \} \]
%
%Respectively, $V$ determines a global truth function in the model $M$ 
%\[\tr: \mfor \times W \to \{0,1\},\] 
%such that $\tr(\fo,w) = 1$ if $w\in V(\fo)$, otherwise $\tr(\fo,w) = 0$. 
%
%In turn, $\tr$ determines for every formula $\fo \in \mfor$ its truth function in $M$, which we denote by $\tr_{\fo}$, where $\tr_{\fo}: W \to \{0,1\}$ is defined by 
%$\tr_{\fo}(w) = \tr(\fo,w)$. 
%Given a set of formulae $\ga = \{\fo_{1},\ldots,\fo_{k}\}$ we denote 
%\[\tr(\ga) := \{\tr{\fo_{1}},\ldots,\tr{\fo_{k}} \}.\]  
%\vmnote{Maybe some of these will not be needed here and can be removed}
%
%\vnote{Relationship between relativised determinacy and independence operators?} 
%
%%%%%%%%%%
%\subsection{Relativised modal determinacy operators}
%%
Given a Kripke model $M = (W,R,V)$ and a possible world $w\in W$, we
define the semantics of $\dep$ as follows:
$M,w \models \dep(\varphi,\dots,\varphi_k;\psi)$  iff for all $u,v\in W$
such that $wRu$ and $wRv$,
if the equivalence $M,u\models\varphi_i\Leftrightarrow M,v\models\varphi_i$
holds for each $i\leq k$,
then $M,u\models\psi\Leftrightarrow M,v\models\psi$.

It is easy to see that in a propositional language 
extended with both $\dep$ and the standard box modality $\Box$,
we have the equivalence
$$
\dep(\varphi_1,\ldots, \varphi_k; \psi)\ \equiv
\bigvee\limits_{\chi\, \in\, \nff(\varphi_1,\ldots,\varphi_k)}
\Box (\chi\leftrightarrow\psi),
$$
i.e., for all Kripke models $M$ and points $w$ in the domain of $M$,
the pointed model $M,w$ satisfies either both or neither of
the above formulae.
On the other hand, on all models with a reflexive accessibility relation (but not in general),
the box modality $\Box$ is definable in
terms of $\dep$ by $\Box \fo := \fo \land \cdep\fo$,
where $\cdep\varphi$ of course denotes $\dep(\emptyset;\varphi)$.
Studying $\dep$ over different classes of Kripke models 
provides an interesting research direction. 
In fact, the recent study \cite{JieFan2016} has already taken up that direction. 
%

%
%
%
%\medskip
%A final note: while $\dep^{\Box}$ and $\dep^{\fod}$ are not mutually definable in general, %$\dep^{\Box}$ is definable in terms of $\dep^{\fod}$ in
%sufficiently expressive hybrid modal logics. 
%
%In the follow-up work we will study the properties of $\dep^{\Box}$ as the only primitive operator added to \lang, as well as its relationships with the relativised determinacy and independence operators. 
%
%%%%%%%%%%
%\subsubsection{Valid principles of relativised modal determinacy over some important modal logics}
%%
%
%\vnote{Add validities and axiomatisations of $\dep^{\Box}$ on some important modal logics: K, T, K4, S4, S5\ldots} 

%%%%%%%%%%%%%%%%%%%%
%\begin{comment}
\subsection{Logical determinacy and consequence}
%%%%%%%%%%%%%%%%%%%%

It is possible to extend the scope of $\dep$ to cover arbitrary
sets of formulae as follows. Let $\ga$ denote a possibly infinite 
set of formulae of $\mathcal{L}_{\dep}$.  Define $W,w\models\dep(\Gamma,\psi)$ if
for all assignments $u,v\in W$ it holds that if the equivalence $W,u\models\varphi\Leftrightarrow
W,v\models\varphi$ holds for all $\varphi\in\Gamma$, 
then $W,u\models\psi\Leftrightarrow
W,v\models\psi$.
The determinacy operator \dep now parallels in a natural way Tarski's notion of a logical
consequence operator \cons defined so that $\cons(\ga,\psi)$
holds if and only if $W,w\models\psi$ for every state
description model $W$ and assignment $w$
such that $W,w\models\varphi$ for all $\varphi\in\Gamma$. The parallel is in the sense that \dep  
satisfies the same defining properties (Reflexivity, Monotonicity and Cut) which Tarski postulated for \cons (while $\dep$ has some interesting extra properties that \cons lacks).  
Indeed, this is not accidental, because \dep and \cons bear technically similar ideas: \cons preserves truth, whereas \dep preserves invariance of truth values. 
For a further discussion of this and related issues, see 
\cite{Humberstone92,Humberstone93,Humberstone98}, as well as 
\cite{ciardellitwo, giardelli} for an argument presenting dependence as a case of logical 
consequence applied to questions instead of propositions.
Thus, an interesting research direction extending the present work involves
relating determinacy and logical consequence in the more 
general modal setting outlined here. 
%
%\end{comment}
%
%%%%%%%%%%%%%

%%%%%%%%%%%%%%%%%%%%%
\subsection{Determinacy operators and conditional knowledge}
%%%%%%%%%%%%%%%%%%%%%

The determinacy operator \dep has a natural epistemic reading: the
determinacy formula $\dep(\varphi_1,\dots,\varphi_k,\psi)$ can be interpreted to mean that an agent knows the truth value of $\psi$ relative to the truth
values of the formulae $\varphi_1,\dots,\varphi_k$ in the sense that the
agent can always deduce the truth value of $\psi$ if
she learns the truth values of $\varphi_1,\dots,\varphi_k$.
This interpretation of $\dep$ leads to yet another open research direction. 
It is worth noting that a uniform analysis of various kinds of knowledge, including knowledge of questions and knowledge of dependencies of this kind has already been
given in the work \cite{CiardelliR15} on inquisitive epistemic logic (IEL).  
Indeed, conditional knowledge can be regarded as knowledge obtained by answering questions:
the formula $\dep(\varphi;\psi)$ with the meaning ``agent $a$ knows whether $\psi$ holds conditionally on the knowledge whether $\varphi$ holds'' can be expressed in IEL as 
$K_{a}(?\varphi \to ? \psi)$.

\medskip

\medskip

\noindent
\textbf{Acknowledgements.} 
Valentin Goranko was partly supported by a research grant 2015-04388 of the Swedish Research Council. 
Antti Kuusisto
was supported by Jenny and
Antti Wihuri Foundation and the ERC grant
647289 (CODA). The results of this paper
were presented at the Dagstuhl seminar
\emph{Logics for Dependence and Independence} in
June 2015, where an earlier manuscript of the present paper was circulated; 
we thank the participants for their interest.
After submitting the arXiv preprint  \cite{GorankoKuusistoArxiv2016}, we have 
also been in correspondence with Jie Fan and Lloyd Humberstone.
We thank them for their comments and discussion. In particular, Humberstone subsequently commented in depth some aspects of our work in \cite{Humberstone2017} and
Jie Fan produced the manuscript \cite{JieFan2016}.
%

%
%%%%%%%%%%%%%%%%%%%%
%\subsection{Concluding remarks}
%%%%%%%%%%%%%%%%%%%%

%\bibliographystyle{alpha}
\bibliography{determinacy}

\begin{thebibliography}{HKMV15}

\bibitem[Alo16]{aloni}
Maria Aloni.
\newblock Disjunction.
\newblock In Edvard~N. Zalta, editor, {\em The Stanford Encyclopedia of
  Philosophy}. Stanford University, Stanford, 2016.

\bibitem[Arm74]{Armstrong74}
William Armstrong.
\newblock Dependency structures of database relationships.
\newblock In {\em IFIP'74}, pages 580--583, 1974.

\bibitem[Cia09]{ciardellimasters}
Ivano Ciardelli.
\newblock {Inquisitive semantics and intermediate logics}.
\newblock Master's thesis, Institute of Logic, Language and Computation,
  University of Amsterdam, 2009.

\bibitem[Cia16a]{ciardellitwo}
Ivano Ciardelli.
\newblock Dependency as question entailment.
\newblock In Jouko~V\"{a}\"{a}n\"{a}nen Samson~Abramsky, Juha~Kontinen and
  Heribert Vollmer, editors, {\em Dependence Logic: theory and applications},
  pages 129--181. Springer International Publishing, Switzerland, 2016.

\bibitem[Cia16b]{giardelli}
Ivano Ciardelli.
\newblock {\em Questions in Logic}.
\newblock PhD thesis, Institute of Logic, Language and Computation, University
  of Amsterdam, 2016.

\bibitem[CR11]{ciardelliroelofsen}
Ivano Ciardelli and Floris Roelofsen.
\newblock Inquisitive logic.
\newblock {\em Journal of Philosophical Logic}, 40(1):55--94, 2011.

\bibitem[CR15]{CiardelliR15}
Ivano Ciardelli and Floris Roelofsen.
\newblock Inquisitive dynamic epistemic logic.
\newblock {\em Synthese}, 192(6):1643--1687, 2015.

\bibitem[Dem97]{DBLP:journals/sLogica/Demri97}
St{\'{e}}phane Demri.
\newblock A completeness proof for a logic with an alternative necessity
  operator.
\newblock {\em Studia Logica}, 58(1):99--112, 1997.

\bibitem[Fan16]{JieFan2016}
Jie Fan.
\newblock A modal logic of supervenience.
\newblock Manuscript, arXiv:1611.04740v1, 2016.

\bibitem[FWvD15]{FanWD15}
Jie Fan, Yanjing Wang, and Hans van Ditmarsch.
\newblock Contingency and knowing whether.
\newblock {\em Review of Symbolic Logic}, 8(1):75 --107, 2015.

\bibitem[Gal12]{galli}
Pietro Galliani.
\newblock Inclusion and exclusion dependencies in team semantics - on some
  logics of imperfect information.
\newblock {\em Annals of Pure and Applied Logic}, 163(1):68--84, 2012.

\bibitem[Gal13]{pietroagainwhocares}
Pietro Galliani.
\newblock Epistemic operators in dependence logic.
\newblock {\em Studia Logica}, 101(2):367--397, 2013.

\bibitem[GK16]{GorankoKuusistoArxiv2016}
Valentin Goranko and Antti Kuusisto.
\newblock Logics for propositional determinacy and independence.
\newblock Manuscript, arXiv:1609.07398, 2016.

\bibitem[GP92]{GP92}
Valentin Goranko and Solomon Passy.
\newblock Using the universal modality: Gains and questions.
\newblock {\em Journal of Logic and Computation}, 2(1):5--30, 1992.

\bibitem[Gre39]{grelling}
Kurt Grelling.
\newblock A logical theory of dependence.
\newblock In {\em Proceedings of the 5th International Congress for the Unity
  of Science}, 1939.

\bibitem[GV13]{GV13}
Erich Gr{\"{a}}del and Jouko V{\"{a}}{\"{a}}n{\"{a}}nen.
\newblock Dependence and independence.
\newblock {\em Studia Logica}, 101(2):399--410, 2013.

\bibitem[GV14]{gallianivaa}
Pietro Galliani and Jouko V\"{a}\"{a}n\"{a}nen.
\newblock On dependence logic.
\newblock In A.~Baltag and S.~Smets, editors, {\em Johan F. A. K. van Benthem
  on Logical and Informational Dynamics}, pages 101--119. Springer, 2014.

\bibitem[Hin96]{Hintikka96}
Jaakko Hintikka.
\newblock {\em The Principles of Mathematics Revisited}.
\newblock Cambridge UP, 1996.

\bibitem[HKMV15]{Vir}
Lauri Hella, Antti Kuusisto, Arne Meier, and Heribert Vollmer.
\newblock Modal inclusion logic: Being lax is simpler than being strict.
\newblock In {\em {MFCS} 2015, Proceedings, Part {I}}, pages 281--292, 2015.

\bibitem[HKVV15]{Kont15}
Miika Hannula, Juha Kontinen, Jonni Virtema, and Heribert Vollmer.
\newblock Complexity of propositional independence and inclusion logic.
\newblock In {\em {MFCS} 2015, Proceedings, Part {I}}, pages 269--280, 2015.

\bibitem[HLSV14]{He14}
Lauri Hella, Kerkko Luosto, Katsuhiko Sano, and Jonni Virtema.
\newblock The expressive power of modal dependence logic.
\newblock In R.~Gor{\'{e}}, B.~Kooi, and A.~Kurucz, editors, {\em Advances in
  Modal Logic 10}, pages 294--312. College Publications, 2014.

\bibitem[Hod97]{Hodges97a}
Wilfrid Hodges.
\newblock Compositional semantics for a language of imperfect information.
\newblock {\em Logic Journal of the {IGPL}}, 5(4):539--563, 1997.

\bibitem[HS89]{hisa89}
Jaakko Hintikka and Gabriel Sandu.
\newblock Informational independence as a semantical phenomenon.
\newblock In J.~E. Fenstad, I.~T. Frolov, and R.~Hilpinen, editors, {\em Logic,
  Methodology and Philosophy of Science}, volume~8, pages 571--589. Elsevier,
  Amsterdam, 1989.

\bibitem[Hum92]{Humberstone92}
Lloyd Humberstone.
\newblock Some structural and logical aspects of the notion of supervenience.
\newblock {\em Logique et Analyse}, 35:101--137, 1992.

\bibitem[Hum93]{Humberstone93}
Lloyd Humberstone.
\newblock Functional dependencies, supervenience, and consequence relations.
\newblock {\em Journal of Logic, Language and Information}, 2(4):309--336,
  1993.

\bibitem[Hum95]{Humberstone95}
Lloyd Humberstone.
\newblock The logic of non-contingency.
\newblock {\em Notre Dame Journal of Formal Logic}, 36(2):214--229, 1995.

\bibitem[Hum98]{Humberstone98}
Lloyd Humberstone.
\newblock Note on supervenience and definability.
\newblock {\em Notre Dame Journal of Formal Logic}, 39(2):243--252, 1998.

\bibitem[Hum02]{Humberstone02}
Lloyd Humberstone.
\newblock The modal logic of agreement and noncontingency.
\newblock {\em Notre Dame Journal of Formal Logic}, 43(2):95--127, 2002.

\bibitem[Hum17]{Humberstone2017}
Lloyd Humberstone.
\newblock Supervenience, dependence, disjunction.
\newblock Manuscript, June 2017.

\bibitem[Jan97]{janssen}
Theo Janssen.
\newblock An overview of compositional translations.
\newblock In {\em Compositionality: The Significant Difference, International
  Symposium, COMPOS'97, Bad Malente, Germany, September 8-12, 1997. Revised
  Lectures}, pages 327--349, 1997.

\bibitem[KMSV14]{Konti14}
Juha Kontinen, Julian{-}Steffen M{\"{u}}ller, Henning Schnoor, and Heribert
  Vollmer.
\newblock Modal independence logic.
\newblock In {\em Advances in Modal Logic 10}, pages 353--372, 2014.

\bibitem[KMSV15]{kontivan}
Juha Kontinen, Julian{-}Steffen M{\"{u}}ller, Henning Schnoor, and Heribert
  Vollmer.
\newblock A van benthem theorem for modal team semantics.
\newblock In {\em {CSL} 2015, Proceedings}, pages 277--291, 2015.

\bibitem[Kon13]{kontisurvey}
Juha Kontinen.
\newblock Dependence logic: A survey of some recent work.
\newblock {\em Philosophy Compass}, 8(10):950--963, 2013.

\bibitem[Kuu14]{doubleteam1}
Antti Kuusisto.
\newblock A double team semantics for generalized quantifiers.
\newblock {\em CoRR}, abs/1310.3032v5, 2014.

\bibitem[Kuu15]{doubleteam2}
Antti Kuusisto.
\newblock A double team semantics for generalized quantifiers.
\newblock {\em CoRR}, abs/1310.3032v8, 2015.

\bibitem[LV13]{LV13}
Peter Lohmann and Heribert Vollmer.
\newblock Complexity results for modal dependence logic.
\newblock {\em Studia Logica}, 101(2):343--366, 2013.

\bibitem[MB14]{SEP-Supervenience}
Brian McLaughlin and Karen Bennett.
\newblock Supervenience.
\newblock In Edward~N. Zalta, editor, {\em The Stanford Encyclopedia of
  Philosophy}. Stanford, 2014.

\bibitem[MR66]{montgomery}
Hugh Montgomery and Richard Routley.
\newblock Contingency and non-contingency bases for normal modal logics.
\newblock {\em Logique et Analyse}, 9(35):318--328, 1966.

\bibitem[Piz07]{Pizzi07}
Claudio Pizzi.
\newblock Necessity and relative contingency.
\newblock {\em Studia Logica}, 85(3):395--410, 2007.

\bibitem[Piz13]{Pizzi13}
Claudio Pizzi.
\newblock Relative contingency and bimodality.
\newblock {\em Logica Universalis}, 7(1):113--123, 2013.

\bibitem[Roe13]{Roelofsen2013}
Floris Roelofsen.
\newblock Algebraic foundations for the semantic treatment of inquisitive
  content.
\newblock {\em Synthese}, 190(1):79--102, 2013.

\bibitem[SvE88]{Smith}
Barry Smith and Christina von Ehrenfels.
\newblock {\em Foundations of Gestalt Theory}.
\newblock Philosophia resources library. Philosophia Verlag, 1988.

\bibitem[V{\"a}{\"a}07]{va07}
Jouko V{\"a}{\"a}n{\"a}nen.
\newblock {\em Dependence Logic - A New Approach to Independence Friendly
  Logic}, volume~70 of {\em London Mathematical Society student texts}.
\newblock Cambridge University Press, 2007.

\bibitem[V{\"{a}}{\"{a}}08]{vaa07}
Jouko V{\"{a}}{\"{a}}n{\"{a}}nen.
\newblock Modal dependence logic.
\newblock In K.~R. Apt and R.~van Rooij, editors, {\em New Perspectives on
  Games and Interaction}, pages 237--254. Amsterdam University Press, 2008.

\bibitem[Yan14]{Y14}
Fan Yang.
\newblock {\em Extensions and Variants of Dependence Logic}.
\newblock PhD thesis, University of Helsinki, 2014.

\bibitem[Yan16]{fanyangten}
Fan Yang.
\newblock Uniform definability in propositional dependence logic.
\newblock {\em arXiv:1501.00155}, 2016.

\bibitem[YV16]{vaanayan}
Fan Yang and Jouko V{\"{a}}{\"{a}}n{\"{a}}nen.
\newblock Propositional logics of dependence.
\newblock {\em Annals of Pure and Applied Logic}, 167(7):557--589, 2016.

\end{thebibliography}

\end{document}